\documentclass{article}%
\usepackage{amsmath}
\usepackage{amsfonts}
\usepackage{amssymb}
\usepackage{graphicx}%
\providecommand{\U}[1]{\protect\rule{.1in}{.1in}}
\newtheorem{theorem}{Theorem}[section]
\newtheorem{corollary}[theorem]{Corollary}
\newtheorem{lemma}[theorem]{Lemma}
\newtheorem{proposition}[theorem]{Proposition}
\newenvironment{proof}[1][Proof]{\noindent\textbf{#1.} }{\ \rule{0.5em}{0.5em}}
\begin{document}

\title{Generators and commutators in finite groups; abstract quotients of compact groups}
\author{Nikolay Nikolov and Dan Segal}
\maketitle

\section{Introduction}

Let $G$ be a group and $\{g_{1},\ldots,g_{r}\}$ a finite subset of $G$. If $G
$ is \emph{nilpotent}, then

\begin{description}
\item[(A)] $G=G^{\prime}\left\langle g_{1},\ldots,g_{r}\right\rangle $ implies
$G=\left\langle g_{1},\ldots,g_{r}\right\rangle $;

\item[(B)] $G=\left\langle g_{1},\ldots,g_{r}\right\rangle $ implies
$G^{\prime}=[G,g_{1}]\ldots\lbrack G,g_{r}]$,
\end{description}

\noindent where $G^{\prime}$ denotes the derived group of $G$ and
$\left\langle g_{1},\ldots,g_{r}\right\rangle $ the subgroup generated by
$\{g_{1},\ldots,g_{r}\}$; for $g\in G$ we write%
\[
\lbrack G,g]=\left\{  [x,g]\mid x\in G\right\}
\]
where $[x,g]=x^{-1}g^{-1}xg$ is the usual commutator.

(A) is an easy folklore result; (B) is also well known, and first appeared in
the unpublished 1966 PhD thesis of Peter Stroud; it is a key element in
Serre's proof that subgroups of finite index are open in a finitely generated
pro-$p$ group. Neither (A) nor (B) is true in general for groups that are not
nilpotent. Rather surprisingly, however, similar results hold without assuming
nilpotency, as long the group $G$ is assumed to be \emph{finite}. These are
very much harder, relying in their most general form on the classification of
finite simple groups. The main technical results of the paper \cite{NS}, which
enabled us to generalize Serre's theorem to all finitely generated profinite
groups, imply the following for a finite $d$-generator group $G$:

\begin{description}
\item[(C)] every element of $G^{\prime}$ is equal to a product of $f_{1}(d)$ commutators;

\item[(D)] if $G=G_{\ast}\left\langle g_{1},\ldots,g_{r}\right\rangle $ then
$G=\left\langle g_{ij}\mid i=1,\ldots,r,~j=1,\ldots,f_{2}(d,\alpha
)\right\rangle $ with $g_{ij}$ conjugate to $g_{i}$ for all $i$ and $j$;
\end{description}

\noindent here $G_{\ast}$ is a certain characteristic subgroup of $G$ with the
property that $G/G_{\ast}$ is semisimple-by-soluble, and $\alpha=\alpha(G)$ is
a certain measure of the complexity of $G$ (the largest $n$ such that $G$ has
$\mathrm{Alt}(n)$ as a section). In \cite{NS} we left open the question of
whether $f_{2}$ can be made independent of $\alpha(G)$; it appears as Problem
4.7.1 in the book \cite{S2}, where further background may be found.

The primary purpose of this paper is to answer that question, and more general
versions of it, positively. Although at first glance this may seem a mere
technical improvement, we shall see that it has diverse applications. These
are described in more detail below; among them are the new theorems:

\begin{itemize}
\item \emph{If }$G$\emph{\ is any compact Hausdorff topological group, then
every finitely generated (abstract) quotient of }$G$\emph{\ is finite}.

\item \emph{Let }$G$\emph{\ be a compact Hausdorff group such that }$G/G^{0}$
\emph{is (topologically) finitely generated. Then }$G$ \emph{has a countably
infinite (abstract) quotient if and only if }$G$\emph{\ has an infinite
virtually-abelian (continuous) quotient.}
\end{itemize}

\noindent(Here, $G^{0}$ denotes the connected component of $1$ in $G$).

Indeed, what motivated the present work was the need to develop machinery
powerful enough to establish results of this kind for profinite groups, for
which the methods of \cite{NS} are insufficient; the extension to more general
compact groups was then a relatively natural step.

Our second purpose is to provide a new and more streamlined route to the
results of \cite{NS} and \cite{NS2} -- including the solution of Serre's
problem on finite-index subgroups in finitely generated profinite groups --
and of \cite{NSP}, where it is proved that in those groups the power subgroups
are open. In setting out to prove stronger results, we have found an approach
that is both more unified and in some respects simpler than the original
proofs. Thus in a sense the present paper is a `mark 2' version of \cite{NS} +
\cite{NS2} + \cite{NSP}.

In the course of the proofs, we shall quote a few self-contained propositions
from \cite{NS}. Apart from these, this work is independent of \cite{NS}. In
particular, we shall not be needing the difficult structural results about
finite simple groups that form the substance of \cite{NS2}; these are replaced
by the material of Subsection \ref{quasi}. Some discussion of the new ideas
that we use instead appears at the end of this introduction.

\subsection{Main results on finite groups}

In this subsection all groups are assumed to be \emph{finite}. The minimal
size of a generating set for $G$ is denoted $\mathrm{d}(G)$. To a finite group
$G$ we associate the characteristic subgroup%
\begin{align}
G_{0}  &  =\bigcap\left\{  T\vartriangleleft G\mid G/T\text{ is almost-simple}%
\right\} \label{defG0}\\
&  =\bigcap_{M\in\mathcal{S}}\mathrm{C}_{G}(M)\nonumber
\end{align}
where $\mathcal{S}$ is \emph{the set of all non-abelian simple chief factors
of} $G$ (a group $H$ is \emph{almost-simple }if $S\vartriangleleft
H\leq\mathrm{Aut}(S)$ for some non-abelian simple group $S$). We remark that
$G/G_{0}$ is an extension of a semisimple group by a soluble group of derived
length at most $3$, because the outer automorphism group of any simple group
is soluble of derived length at most $3$ (strong form of the Schreier
conjecture, see Subsection \ref{facts}). (Note that $G_{0}=G$ if $\mathcal{S}$
is empty, by the usual convention.)

\subsubsection{Generators}

In Section \ref{gensec} we prove

\begin{theorem}
\label{ThmA}Let $G$ be a group and $K\leq G_{0}$ a normal subgroup of $G$.
Suppose that $G=K\left\langle y_{1},,\ldots,y_{r}\right\rangle =G^{\prime
}\left\langle y_{1},,\ldots,y_{r}\right\rangle $. Then there exist elements
$x_{ij}\in K$ such that%
\[
G=\left\langle y_{i}^{x_{ij}}\mid i=1,\ldots,r,~j=1,\ldots,f_{0}\right\rangle
\]
where $f_{0}=f_{0}(r,\mathrm{d}(G))=O(r\mathrm{d}(G)^{2})$.
\end{theorem}

\noindent It is clear that the $y_{i}$ must be assumed to generate $G$ modulo
$G^{\prime}$; the definition of $G_{0}$ serves to exclude obvious
counterexamples of the form $G=K\times\left\langle y_{1},,\ldots
,y_{r}\right\rangle $ where $K$ is simple or $G=\mathrm{Sym}(n)$ with $y_{1}$
a transposition.

Recall that $G_{0}=G$ if every non-abelian chief factor of $G$ has composition
length at least $2$, in particular if $G$ is soluble; the result in the
soluble case was established in \cite{S1}.

\subsubsection{Commutators}

For a subset $X$ of a group $G$, we write%
\[
X^{\ast f}=\left\{  x_{1}x_{2}\ldots x_{f}\mid x_{1},x_{2},\ldots,x_{f}\in
X\right\}  .
\]
The subset $X$ is \emph{symmetric} if $x\in X$ implies $x^{-1}\in X$.

For subgroups $H,K$ of $G$,%
\[
\lbrack H,K]=\left\langle [x,y]\mid x\in H,~y\in K\right\rangle .
\]

\begin{theorem}
\label{ThmB}Let $G$ be a group and $\{y_{1},\ldots,y_{r}\}$ a symmetric
generating set for $G$. If $H$ is a normal subgroup of $G$ then%
\[
\lbrack H,G]=\left(
{\displaystyle\prod\limits_{i=1}^{r}}
[H,y_{i}]\right)  ^{\ast f_{1}}%
\]
where $f_{1}=f_{1}(r,\mathrm{d}(G))=O(r^{2}\mathrm{d}(G))=O(r^{3})$.
\end{theorem}

This is proved in Section \ref{commsection}, together with the following
`relative' version, our main result on finite groups:

\begin{theorem}
\label{ThmC}Let $G$ be a group, $H\leq G_{0}$ a normal subgroup of $G$, and
$\{y_{1},\ldots,y_{r}\}$ a symmetric subset of $G$. If $H\left\langle
y_{1},\ldots,y_{r}\right\rangle =G^{\prime}\left\langle y_{1},\ldots
,y_{r}\right\rangle =G$ then%
\[
\lbrack H,G]=\left(
{\displaystyle\prod\limits_{i=1}^{r}}
[H,y_{i}]\right)  ^{\ast f_{2}}%
\]
where $f_{2}=f_{2}(r,\mathrm{d}(G))=O(r^{6}\mathrm{d}(G)^{6})$.
\end{theorem}

\noindent This is in effect `Key Theorem C' of \cite{NS}, with the fundamental
improvement that $f_{2}$ no longer depends on $\alpha(G)$. In fact Theorem
\ref{ThmC} simultaneously generalizes all three versions of the said `Key
Theorem' (and strengthens them, with our new definition of $G_{0}$).

A variant of Theorem \ref{ThmB} also holds, where $\{y_{1},\ldots,y_{r}\}$ is
merely assumed to generate $G$ modulo $\mathrm{C}_{G}(H)$ and $f_{1}=O(r^{3})$
is independent of $\mathrm{d}(G)$; the proof is a little more involved and
will appear elsewhere.

Sharper estimates for the functions $f_{0},~f_{1},~\,f_{2}$ will appear in the
course of the proofs.

\subsubsection{Verbal subgroups}

A group word $w$ has \emph{width} $m$ in a group $G$ if every product of
$w$-values in $G$ is equal to such a product of length $m$; here, by
$w$\emph{-value} we mean an element of the form $w(\mathbf{g})^{\pm1}$ with
$\mathbf{g}\in G^{(k)}$, where $w$ is a word on $k$ variables. In Subsection
\ref{vbl} we show how the following theorem, originally established in
\cite{NSP}, easily follows from the above results:

\begin{theorem}
\label{ThmV}Let $w$ be a non-commutator word and $G$ a finite $d$-generator
group. Then $w$ has width $f(w,d)$ in $G$, where $f(w,d)$ depends only on $w$
and $d$.
\end{theorem}

\subsection{Algebraic properties of compact groups\label{algprop}}

A compact group (which we take to mean a compact Hausdorff topological group)
is an extension $G$ of a compact connected group $G^{0}$, its identity
component, by a profinite group $G/G^{0}$. The Levi-Mal'cev Theorem shows that
the connected component is essentially a product of compact Lie groups; this
makes it relatively tractable, and most of our attention will be focused on
the profinite case.

\subsubsection{Finitely generated profinite groups\label{introfgpg}}

The significance of uniform bounds relating to all $d$-generator finite groups
is that they reflect qualitative properties of $d$-generator profinite groups.
Thus (C) implies that the derived group is closed in every finitely generated
profinite group; and the main `finite' results of \cite{NS} were used to show
that every subgroup of finite index in a finitely generated profinite group
$G$ is open. A more roundabout argument, using results from \cite{NS} related
to (D), was used in \cite{NSP} to show that the `power subgroups' $G^{q}$ are
open in $G$. The sharper results now at our disposal yield further dividends
when applied in the profinite context.

Routine compactness arguments (recalled in Subsection \ref{routine}) transform
Theorems \ref{ThmA}, \ref{ThmB} and \ref{ThmC} into the following.

\begin{theorem}
\label{prof-gen}Let $G$ be a finitely generated profinite group and $K\leq
G_{0}$ a closed normal subgroup of $G$. Suppose that $G=K\overline
{\left\langle y_{1},,\ldots,y_{r}\right\rangle }=\overline{G^{\prime
}\left\langle y_{1},,\ldots,y_{r}\right\rangle }$. Then there exist elements
$x_{ij}\in K$ such that%
\[
G=\overline{\left\langle y_{i}^{x_{ij}}\mid i=1,\ldots,r,~j=1,\ldots
,f_{0}\right\rangle }%
\]
where $f_{0}=f_{0}(r,\mathrm{d}(G))$.
\end{theorem}

\noindent Here, $G_{0}$ is defined by (\ref{defG0}) with $T$ ranging over
\emph{open} normal subgroups; and $\overline{X}$ denotes the closure of a
subset $X$ in $G$. As in the finite case, $G/G_{0}$ is an extension of a
semisimple group by a soluble group of derived length at most $3$ (a
semisimple profinite group is a Cartesian product of finite simple groups).

\begin{theorem}
\label{ThB-prof}Let $G$ be a profinite group and $\{y_{1},\ldots,y_{r}\}$ a
symmetric (topological) generating set for $G$. If $H$ is a closed normal
subgroup of $G$ then%
\begin{equation}
\lbrack H,G]=\left(
{\displaystyle\prod\limits_{i=1}^{r}}
[H,y_{i}]\right)  ^{\ast f_{1}} \tag{$\dagger$}\label{rhsB}%
\end{equation}
where $f_{1}=f_{1}(r,\mathrm{d}(G))$.
\end{theorem}

This implies that $[H,G]$ is closed in $G$, a result already established in
\cite{NS}.

\begin{theorem}
\label{ThmC-prof}Let $G$ be a finitely generated profinite group, $H\leq
G_{0}$ a closed normal subgroup of $G$, and $\{y_{1},\ldots,y_{r}\}$ a
symmetric subset of $G$. If $H\overline{\left\langle y_{1},\ldots
,y_{r}\right\rangle }=\overline{G^{\prime}\left\langle y_{1},\ldots
,y_{r}\right\rangle }=G$ then%
\begin{equation}
\lbrack H,G]=\left(
{\displaystyle\prod\limits_{i=1}^{r}}
[H,y_{i}]\right)  ^{\ast f_{2}} \tag{$\ddagger$}\label{rhsC}%
\end{equation}
where $f_{2}=f_{2}(r,\mathrm{d}(G))$.
\end{theorem}

Why is this important? Suppose that $N$ is a proper normal subgroup in a group
$G$. If $G$ is \emph{finite}, then $N$ is contained in some maximal normal
subgroup $M$ of $G$. If $G/M$ is abelian, then $NG^{\prime}\leq M<G$; if not,
then $G/M$ is a simple chief factor of $G$, so $M\geq G_{0}$ and $NG_{0}\leq
M<G$. So far, so trivial. Now suppose that $G$ is a \emph{profinite} group:
unless we assume that $N$ is \emph{closed} in $G$, we have no grounds to
assert that $N$ is contained in a maximal open normal subgroup -- indeed $N$
could be dense in $G$. If $G$ is a \emph{finitely generated} profinite group,
however, we claim that at least one of $NG^{\prime}$, $NG_{0}$ is necessarily
properly contained in $G$. For suppose that $NG^{\prime}=NG_{0}=G$. If $G$ is
topologically generated by $d$ elements, we can find $2d$ elements
$y_{1},\ldots,y_{2d}\in N$ such that $G_{0}\overline{\left\langle y_{1}%
,\ldots,y_{2d}\right\rangle }=\overline{G^{\prime}\left\langle y_{1}%
,\ldots,y_{2d}\right\rangle }=G$, and Theorem \ref{ThmC-prof} (with $H=G_{0})$
then implies that
\[
\lbrack G_{0},G]\leq\left\langle \lbrack G_{0},y_{i}],~[G_{0},y_{i}^{-1}%
]\mid1\leq i\leq2d\right\rangle \leq N.
\]
But then%
\[
G=NG^{\prime}=N[NG_{0},G]=N.
\]
Thus we may state

\begin{corollary}
\label{normal_mod_G_0}Let $G$ be a finitely generated profinite group and $N$
a normal subgroup of (the underlying abstract group) $G$. If $NG^{\prime
}=NG_{0}=G$ then $N=G$.
\end{corollary}

\noindent This is the key to understanding `abstract' normal subgroups. For
example, it quickly reduces Serre's problem on finite-index subgroups ((E)
stated below) to the special cases of abelian groups and semisimple groups,
where the answer has long been known: see Subsection \ref{Finite-index-sec}.
More generally, it shows that if $G$ has a \emph{dense} proper normal
subgroup, then at least one of $G/G^{\prime}$ or $G/G_{0}$ has a dense proper
normal subgroup; the point is that each of these quotients has relatively
transparent structure. This is exploited to good effect in Subsections
\ref{fgquot} and \ref{dense}.

\bigskip

In Subsection \ref{vbl} we discuss the profinite version of Theorem \ref{ThmV}:

\begin{theorem}
\label{prof-V}\emph{\cite{NSP} }Let $G$ be a finitely generated profinite
group and $w$ a non-trivial non-commutator word. Then the verbal subgroup
$w(G)$ is open in $G$.
\end{theorem}

Such results also imply certain \emph{rigidity} properties for profinite
groups, that is, conditions under which abstract group homomorphisms are
forced to be continuous. Let $G$ be a profinite group, $Q\neq1$ an abstract
group, and $f:G\rightarrow Q$ a surjective homomorphism, with kernel $N$.

We can restate the main result of \cite{NS} (re-proved in Subsection
\ref{Finite-index-sec}) as:

\begin{description}
\item[(E)] \emph{If }$G$\emph{ is finitely generated (topologically) and }%
$Q$\emph{ is finite, then }$N$\emph{ is open.}
\end{description}

\noindent This is also true if $G$ is a \emph{connected} compact group instead
of profinite: indeed, such a group is divisible, hence has no nontrivial
finite quotients at all (\cite{HM1}, Theorem 9.35).

An immediate consequence of (E) is

\begin{description}
\item[(F)] \emph{If }$G$\emph{ is finitely generated and }$Q$\emph{ is
residually finite, then }$N$\emph{ is closed, so }$Q$\emph{ is profinite (with
topology inherited from }$G/N$\emph{ via }$f$\emph{); hence }$Q$\emph{ cannot
be countably infinite.}
\end{description}

\noindent Rather surprisingly, it is easy to find countably infinite
non-(residually finite) images (if using the axiom of choice counts as
`finding'): if $\phi:\mathbb{Q}_{p}\rightarrow\mathbb{Q}$ is any $\mathbb{Q}%
$-vector space epimorphism then $\mathbb{Z}_{p}\phi$ is a countably infinite
image of $\mathbb{Z}_{p}$ (in fact it is an exercise, given (F), to show that
$\mathbb{Z}_{p}\phi=\mathbb{Q}$). This suggests the question: can $Q$ be
\emph{finitely generated} and infinite? This is answered below.

\subsubsection{Compact groups}

Many of the above results hold more generally for compact groups $G$, assuming
usually that the profinite quotient $G/G^{0}$ is finitely generated ($G^{0}$
denotes the connected component of the identity in $G$). The structure of a
connected compact group is relatively straightforward: it is semisimple modulo
its centre (where by a \emph{connected compact semisimple group} we mean a
Cartesian product of compact connected simple Lie groups). In Subsection
\ref{simplequot} we prove:

\begin{theorem}
\label{ssquotients}Let $G$ be a semisimple compact group that is either
finitely generated profinite or connected. If $Q$ is an infinite quotient of
$G$ then $\left\vert Q\right\vert \geq2^{\aleph_{{0}}}$\emph{.}
\end{theorem}

\noindent In the \emph{profinite} case, we also give a complete classification
of the \emph{maximal }normal subgroups of $G$. Both results depend on
associating to each normal subgroup an ultrafilter on the underlying index set
of the Cartesian product.

The main results on quotients of compact groups are established in Section
\ref{fgquot}, using Corollary \ref{normal_mod_G_0} and Theorem
\ref{ssquotients}:

\begin{theorem}
\label{countablefabthm}Let $G$ be a compact group such that $G/G^{0}$ is
(topologically) finitely generated. Let $N$ be a normal subgroup of (the
underlying abstract group) $G$. If $G/N$ is countably infinite then $G/N$ has
an infinite virtually-abelian quotient.
\end{theorem}

\begin{corollary}
Let $G$ be a compact group such that $G/G^{0}$ is (topologically) finitely
generated. Then $G$ has a countably infinite (abstract) quotient if and only
if $G$ has an infinite virtually-abelian (continuous) quotient.
\end{corollary}

Using (F) in conjunction with Theorem \ref{countablefabthm} it is easy to deduce

\begin{theorem}
\label{ThmD}Let $G$ be a compact group and $N$ a normal subgroup of (the
underlying abstract group) $G$ such that $G/N$ is finitely generated. Then
$G/N$ is finite.
\end{theorem}

If $G/N$ is a countable quotient of $G$ then the closure of $N$ must be open
in $G$; in this case we say that $N$ is \emph{virtually dense} in $G$. More
generally, one might ask: \emph{under what conditions is it possible for a
normal subgroup of infinite index to be virtually dense?} The answer is
`always' in abelian groups -- for example, $\mathbb{Z}$ is dense in
$\mathbb{Z}_{p}$; and the results of Subsection \ref{simplequot} show that a
semisimple group can have uncountably many dense normal subgroups. When $G$ is
finitely generated profinite, Corollary \ref{normal_mod_G_0} shows that these
extreme cases essentially account for all possibilities; when $G$ is
connected, the proof of Theorem \ref{ssquotients} enables us to draw a similar
conclusion. Let us say that a semisimple compact group is \emph{strictly
infinite} if it is the product of an infinite set of simple connected Lie
groups or finite simple groups. In Subsection \ref{dense} we prove

\begin{theorem}
\label{vdn1}Let $G$ be a compact group such that $G/G^{0}$ is (topologically)
finitely generated. Then $G$ has a virtually dense normal subgroup of infinite
index if and only if\emph{ }some open normal subgroup of $G$ has an infinite
abelian quotient or a strictly infinite semisimple quotient.
\end{theorem}

An easy consequence is

\begin{corollary}
\label{J-I}Let $G$ be a finitely generated just-infinite profinite group that
is not virtually abelian. Then every normal subgroup of $G$ is closed.
\end{corollary}

\noindent($G$ is \emph{just-infinite} if $G$ is infinite and every closed
non-identity normal subgroup is open. The corollary generalizes a result of A.
Jaikin \cite{JZ}, who proved it for pro-$p$ groups.)$\medskip$

If $G$ is connected, a virtually dense subgroup is the same thing as a dense
subgroup; if $G$ is profinite, however, the conditions for the existence of a
proper dense normal subgroup are more stringent. Their precise
characterization (which depends only on $G/G^{\prime}$ and $G/G_{0}$) is
stated in our final theorem, whose proof will appear elsewhere.

\subsection{Overview of the paper, conventions, remarks}

The basic idea is very simple. Suppose that $G=\left\langle g_{1},\ldots
,g_{r}\right\rangle $ is a finite group. If $M$ is a non-central chief factor
of $G$ then at least one of the generators $g_{i}$ must centralize a
relatively small proportion of the points of $M$, so the set of commutators
$[M,g_{i}]$ must be relatively large. Although we can't predict which value of
$i$ is the relevant one, we can in any case infer that the set%
\[
\prod_{i=1}^{r}[M,g_{i}]
\]
is relatively large: thus `many' of the elements of $M$ can be expressed as
products, of bounded length, of commutators with the original generators
$g_{i}$.

For this to be of any use, we need to replace `many' with `all'. The most
difficult parts of \cite{NS} and \cite{NS2} were devoted to that end; we can
now replace some of those arguments with the help of a new `portmanteau'
result, which we call `the Gowers trick'. This is explained below.

For many applications, one needs to have an analogous result for a subset
$\{g_{1},\ldots,g_{r}\}$ which may not generate the whole group. This was
achieved in \cite{NS} (`Key Theorem C') only under severe restrictions on the
structure of the group $G$. Somewhat to our surprise, these restrictions turn
out to be unnecessary: in Section \ref{gensec} we show that the $g_{i}$ have
the necessary `fixed-point' property on chief factors provided only that
$\{g_{1},\ldots,g_{r}\}$ satisfies the hypotheses of Theorem \ref{ThmA}. The
proof is in principle elementary, relying on the O'Nan-Scott Theorem to
analyse the action of $G$ on its chief factors.

In Section \ref{commsection} the main results on products of commutators are
reduced to Theorem \ref{comm20}: this technical result, the hard core of the
paper, concerns a (quasi-)semisimple group $N$ with operators $y_{ij}$, and
shows that every element of $N$ is equal to a certain product of `twisted
commutators' with the $y_{ij}$. The whole of Section \ref{semisimplesec} is
devoted to the proof of this theorem. While the combinatorial reduction
arguments are still quite complicated, the proofs in Subsection \ref{quasi} of
the necessary results about finite simple groups are relatively short and transparent.

The final Section \ref{profsec} can be read independently of the rest. Here we
derive all the above-stated applications to topological groups, using only the
statements of Theorems \ref{prof-gen} -- \ref{ThmC-prof} and Corollary
\ref{normal_mod_G_0}, with some additional material relating to connected
compact groups.

\bigskip

The main theorems stated above are not all stated in their sharpest form:
sharper, but less succinct, versions are formulated and proved in the body of
the paper.

We take as given the classification of finite simple groups. Some of the main
results depend on general consequences of CFSG, such as the facts that finite
simple groups can be generated by a bounded number of elements, have bounded
commutator width, and have soluble outer automorphism groups (the Schreier
conjecture). Others depend on specific properties of groups of Lie type, such
as the proportion of regular semisimple elements in these groups, and the
detailed structure of their automorphisms. Recent results such as the proof of
the Ore Conjecture \cite{LOST}, which says that simple groups have commutator
width equal to one, lead to sharper estimates for the implied constants in our
main theorems, but are not necessary if one is satisfied with qualitative
statements as given above.

\subsubsection{The `Gowers trick'}

A key tool in some of the proofs is a remarkable combinatorial result
discovered by Tim Gowers. The basic idea is this: to show that a finite group
is equal to the product of some of its subsets, it is enough to know merely
that the subsets have sufficiently big cardinalities. We will need the
following generalization of Gowers's result.

For a finite group $G$ let $l(G)$ denote \emph{the minimal dimension of any
non-trivial} $\mathbb{R}$\emph{-linear representation of} $G$.

\begin{theorem}
\emph{(\cite{BNP} Corollary 2.6)} Let $X_{1},\ldots,X_{t}$ be subsets of $G$,
where $t\geq3$. Then%
\[%
{\displaystyle\prod\limits_{i=1}^{t}}
\left\vert X_{i}\right\vert \geq\left\vert G\right\vert ^{t}\cdot
l(G)^{2-t}\text{ implies }X_{1}\cdot\ldots\cdot X_{t}=G.
\]

\end{theorem}

This holds in particular if $\left\vert X_{i}\right\vert \geq\left\vert
G\right\vert \cdot l(G)^{-\mu}$ for each $i$, where $t\mu\leq t-2$.

\subsubsection{Facts about simple groups\label{facts}}

Here we list some frequently quoted results, for ready reference. Here
$S^{\ast}$ will denote a quasisimple group (see below) and $S=S^{\ast
}/\mathrm{Z}(S^{\ast})$ a finite (non-abelian) simple group.

\begin{proposition}
\label{2-gen}\emph{\cite{AG} }$S^{\ast}$ can be generated by $2$ elements.
\end{proposition}

\noindent(This is usually stated for simple groups, but of course any
generating set for $S$ lifts to a generating set of $S^{\ast}$.)

\begin{proposition}
\label{schreier}\emph{(\cite{GLS}, Sections 7.1, 2.5)} The outer automorphism
group $\mathrm{Out}(S)$ is soluble of derived length at most $3$.
\end{proposition}

\begin{proposition}
\label{commwidth}\emph{(i) (\cite{W}, Proposition 2.4) }There exists
$\delta\in\mathbb{N}$ such that every element of $S$ is a product of $\delta$
commutators.\newline\emph{(ii)} There exists $\delta^{\ast}\in\mathbb{N}$ such
that every element of $S^{\ast}$ is a product of $\delta^{\ast}$ commutators.
\end{proposition}

\noindent((ii) follows from (i) by a theorem of Blau \cite{B}, which asserts
that every element of $\mathrm{Z}(S^{\ast})$ is a commutator unless $S^{\ast}$
is one of finitely many exceptions.)

\begin{corollary}
\label{quasigens}$S^{\ast}$ can be generated by $2\delta$ commutators.
\end{corollary}

For the record, we recall the validity of the \emph{Ore Conjecture} (not
strictly necessary for our results but yielding better values for the constants):

\begin{proposition}
\label{Ore}\emph{(\cite{LOST}, \cite{LOST2}) }$\delta=1,$ $\delta^{\ast}=2$.
\end{proposition}

\begin{proposition}
\label{l(simple)}\emph{(\cite{LaS}; \cite{KlLi} Table 5.3A.)} Let $S^{\ast}$
be a quasisimple group of Lie type, of untwisted Lie rank $r$ over
$\mathbb{F}_{q}$ where $q^{r}>27$. Then $l(S^{\ast})\geq(q^{r}-1)/2$.
\end{proposition}

\begin{proposition}
\label{conjclass}\emph{\cite{LiSh}} There is an absolute constant $c^{\prime}$
such that: if $Y$ is a normal subset of $S$ then%
\[
\left\vert Y\right\vert ^{n}\geq\left\vert S\right\vert \Longrightarrow
Y^{\ast c^{\prime}n}=S.
\]

\end{proposition}

It is convenient to define the \emph{rank} of a simple group as follows: if
$S$ is of Lie type, $\mathrm{rank}(S)$ is the (untwisted) Lie rank of $S$; if
$S\cong\mathrm{Alt}(n)$, $\mathrm{rank}(S)=n$; if $S$ is sporadic,
$\mathrm{rank}(S)=0$. The next result is essentially a special case of the
main theorem of \cite{BCP}:

\begin{proposition}
\label{bcpbound}If $C$ is a proper subgroup of $S$ then $\left\vert
S:C\right\vert \geq\left\vert S\right\vert ^{\varepsilon(r)}$ where
$\varepsilon(r)>0$ depends only on $r=\mathrm{rank}(S)$.
\end{proposition}

\subsubsection{Notation}

For a group $G$, the centre is $\mathrm{Z}(G)$ and the derived group is
$G^{\prime}$. For $n>1$, $G^{(n)}=(G^{(n-1)})^{\prime}$ where $G^{(1)}%
=G^{\prime}$.

For a subset $X$ and an element $y$ of $G$, $[X,y]$ denotes the \emph{set}
$\{[x,y]\mid x\in X\}$. When $X$ and $Y$ are both \emph{subgroups} of $G$,
$[X,Y]$ denotes the \emph{subgroup }$\left\langle [x,y]\mid x\in X,~y\in
Y\right\rangle $. In particular, the terms of the lower central series are
defined by $\gamma_{1}(G)=G$, $\gamma_{2}(G)=G^{\prime}$, and for $n>1$%

\[
\gamma_{n}(G)=[G,\gamma_{n-1}(G)].
\]

$\gamma_{\omega}(G)=%
{\textstyle\bigcap\nolimits_{n=1}^{\infty}}
\gamma_{n}(G)$ is the \emph{nilpotent residual} of $G$. If $G$ is finite, then
for some $n$ we have $\gamma_{\omega}(G)=\gamma_{n}(G)=[\gamma_{\omega}(G),G]$.

The notation $G^{(m)}$ is also used for the Cartesian power $G\times
\cdots\times G$ with $m$ factors; which meaning is intended should be clear
from the context. For $\mathbf{a},~\mathbf{b}\in G^{(m)}$ and $\alpha
\in\mathrm{Aut}(G)^{(m)},$%
\begin{align*}
\mathbf{a}\cdot\mathbf{b}  &  =(a_{1}b_{1},\ldots,a_{m}b_{m})\\
\lbrack\mathbf{a},\alpha]  &  =([a_{1},\alpha_{1}],\ldots,[a_{m},\alpha
_{m}])\\
\mathbf{c}(\mathbf{a},\alpha)  &  =%
{\displaystyle\prod\limits_{j=1}^{m}}
[a_{j},\alpha_{j}]
\end{align*}
where as usual $[a,\beta]=a^{-1}a^{\beta}$.$\medskip$

In sections \ref{gensec} \-- \ref{semisimplesec}, `group' means `finite
group', and `simple group' means `non-abelian simple group'.

A direct (or Cartesian) product of simple groups is called \emph{semisimple}.
A group $G$ is \emph{quasisimple} if $G$ is perfect (i.e. $G=G^{\prime}$) and
$G/\mathrm{Z}(G)$ is simple. A central product of quasisimple groups is called
\emph{quasi-semisimple}.

For a topological group $G$, the connected component of the identity is
denoted $G^{0}$ (not to be confused with $G_{0}$ defined above (\ref{defG0})).

For $m\in\mathbb{N}$ we write $[m]=\{1,\ldots,m\}$.

\bigskip

When, occasionally, a lemma is stated without proof, it can be verified by a
short direct calculation.

\section{Generators\label{gensec}}

\subsection{Fixed-point properties\label{fix}}

We begin by defining a key technical concept, in three flavours: the
\emph{fixed-point} property (fpp), the \emph{fixed-point} \emph{space
}property (fsp), and the \emph{fixed-group} property (fgp):$\medskip$

\textbf{Definition} Let $\Omega$ be a finite $G$-set, $V$ a finite-dimensional
$kG$-module ($k$ some field), and $M$ a $G$-group (group acted on by $G$). Let
$\varepsilon\in(0,1]$. An element $y\in G$ has the

\begin{itemize}
\item $\varepsilon$\emph{-fpp} on $\Omega$ if $y$ moves at least
$\varepsilon|\Omega|$ points of $\Omega$,

\item $\varepsilon$\emph{-fsp} on $V$ if $\dim V(y-1)\geq\varepsilon\dim V$.
\end{itemize}

\noindent If $Y$ is a subset of $G$, we say that $Y$ has the $\varepsilon
$\emph{-fpp }etc. if there exists $y\in Y$ having the given property.

\begin{itemize}
\item $Y$ has the $\varepsilon$\emph{-fgp} on $M$ if (i) $M=S_{1}\times
\cdots\times S_{n}$ with $n\geq2$ and the action of $G$ permutes the factors
$S_{i}$ transitively, and (ii) for each such decomposition of $M$, the set $Y$
has the $\varepsilon$-fpp on the set $\{S_{1},\ldots,S_{n}\}$.
\end{itemize}

\emph{Remarks}. (i) Each $\varepsilon$ property implies the corresponding
$\varepsilon^{\prime}$ property for any $\varepsilon^{\prime}\leq\varepsilon$.
If $Y$ acts non-trivially on $\Omega$, respectively $V$, then $y$ has the
$2/\left\vert \Omega\right\vert $-fpp on $\Omega$ and the $1/\dim(V)$-fsp on
$V$.

(ii) Suppose that $G$ is imprimitive on $\Omega$ and acts transitively on a
set $\overline{\Omega}$ of blocks. If $y$ has the $\varepsilon$-fpp on
$\overline{\Omega}$ then $y$ has the $\varepsilon$-fpp on $\Omega$.

(iii) If $M$ is a $G$-group and $y$ has the $\varepsilon$-fgp on $M$, then
$\left\vert \mathrm{C}_{M}(y)\right\vert \leq|M|^{1-\varepsilon/2}$.

(iv) Suppose that $G$ acts as an imprimitive linear group on $V$, permuting a
system of imprimitivity $\Omega$ transitively. If $y$ has the $\varepsilon
$-fpp on $\Omega$ then $y$ has the $\varepsilon/2$-fsp on $V$.

(v) Say $\left\vert Y\right\vert =r$. If $\mathrm{C}_{V}(\left\langle
Y\right\rangle )=0$ then $Y$ has the $1/r$-fsp on $V$; if $\left\langle
Y\right\rangle $ has no fixed points in $\Omega$ then $Y$ has the $1/r$-fpp on
$\Omega$.

These are all easy to see; for (iii) and (iv), suppose that $M=S_{1}%
\times\cdots\times S_{n}$ is a $G$-group and $y\in G$ permutes the factors
$S_{i}$, according to a permutation with $r$ cycles, including exactly $k$
cycles of length $1$. Choose representatives $i(1),\ldots,i(r)$ for these
cycles. Then any fixed point of $y$ in $M$ is determined by its projections to
$S_{i(1)},\ldots,S_{i(r)}$, so $\left\vert \mathrm{C}_{M}(y)\right\vert
\leq\left\vert S\right\vert ^{r}=\left\vert M\right\vert ^{r/n}$ if
$S_{1}\cong\ldots\cong S_{n}\cong S$. On the other hand, we have
\[
r\leq(n+k)/2\leq n(1-\varepsilon/2)
\]
if $y$ has the $\varepsilon$-fpp on $\{S_{1},\ldots,S_{n}\}$. This gives
(iii), and (iv) is similar, using $\dim V$ in place of $\left\vert
M\right\vert $.

$\medskip$

We recall the$\medskip$

\noindent\textbf{Definition.}%
\[
G_{0}=\bigcap_{M\in\mathcal{S}}\mathrm{C}_{G}(M)
\]
where $\mathcal{S}$ denotes the set of all (non-abelian) simple chief factors
of $G$.$\medskip$

Recall also that $\delta$ is a number such that every quasisimple group can be
generated by $2\delta$ commutators. Note that we can take $\delta=1$ (see
Subsection \ref{facts}).

\begin{theorem}
\label{fppThm}Suppose $G=G^{\prime}\left\langle Y\right\rangle =G_{0}%
\left\langle Y\right\rangle $. Then $Y$ has the $\varepsilon/2$-fsp on every
non-central abelian chief factor of $G$ and the $\varepsilon$-fgp on every
non-abelian chief factor of $G$ inside $G_{0}$, where%
\[
\varepsilon=\min\left\{  \frac{1}{1+6\delta},\frac{1}{|Y|}\right\}  .
\]

\end{theorem}

\emph{Reductions. }Let $M=S_{1}\times\cdots\times S_{n}$ be a non-abelian
chief factor of $G$, where $n>1$ and $G$ permutes $\Omega=\{S_{1},\ldots
,S_{n}\}$ transitively. Let $\overline{\Omega}$ be a primitive quotient of the
$G$-set $\Omega$. Suppose that $\left\vert \overline{\Omega}\right\vert =2$.
Then $G^{\prime}$ acts trivially on $\overline{\Omega}$, so $\left\langle
Y\right\rangle $ acts transitively on $\overline{\Omega}$, and it follows by
Remarks (i) and (ii) that $Y$ has the $1$-fpp on $\Omega$. Thus $y$ has the
$\varepsilon$-fgp on $M$.

Let $V$ be a non-central abelian chief factor of $G$, so $G$ acts as an
irreducible $\mathbb{F}_{p}$-linear group on $V$.

(i) Suppose that this action is not primitive, so it induces a primitive
permutation action of $G$ on a system of imprimitivity $\Omega$. If
$\left\vert \Omega\right\vert =2$ then as above we may deduce that $Y$ has the
$1$-fpp on $\Omega$, hence the $1/2$-fsp on $V$, by Remark (iv).

(ii) Suppose that $V$ is not inside $G_{0}$. Then $G_{0}$ centralizes $V$, so
$V$ is a non-trivial simple $\mathbb{F}_{p}\left\langle Y\right\rangle
$-module, and then $Y$ has the $\varepsilon$-fsp on $V$ by Remark (v).

(iii) Suppose that $\dim_{\mathbb{F}_{p}}V=1$. Then $G^{\prime}$ centralizes
$V$, whence $V(y-1)=V$ for some $y\in Y$; thus $Y$ has the $\varepsilon$-fsp
on $V$.

Arguing by induction on the number of non-central factors in a chief series of
$G$ inside $G_{0}$, it will therefore suffice to prove the following proposition.

\begin{proposition}
\label{epsilon-prop}Let $G$ be a group and $Y$ a subset of $G$ of size
$r\geq1$ such that $G=G^{\prime}\left\langle Y\right\rangle =G_{0}\left\langle
Y\right\rangle $. Suppose that $\left\langle Y\right\rangle $ does not
centralize any non-central abelian chief factor of $G$, and that if
$M=S_{1}\times\cdots\times S_{n}$ is a non-abelian chief factor of $G$, with
each $S_{i}$ simple and $n\geq2$, then $\left\langle Y\right\rangle $ does not
normalize every $S_{i}$. Put $\varepsilon=\min\{1/(1+6\delta),1/r\}$. Then $Y$
has the $\varepsilon$-fsp on every primitive irreducible $\mathbb{F}_{p}%
G$-module of dimension at least two, and the $\varepsilon$-fpp on every
primitive $G$-set of size at least $3$.
\end{proposition}

\subsubsection{Primitive modules}

Let $G$ be a group and $Y$ a subset of $G$ of size $r$ satisfying the
hypotheses of Proposition \ref{epsilon-prop}. Let $V$ be a primitive
irreducible $\mathbb{F}_{p}G$-module of dimension at least two; we may assume
that $G$ acts faithfully on $V$. Put $F=\mathrm{Fit}(G)$, the Fitting subgroup
of $G$.

\begin{lemma}
\label{comm}Let $S$ be a quasisimple subgroup of $G$ and $y$ an element of $G$
such that $[S,S^{y}]=1$. Then there exist $a_{j},b_{j}\in S$ such that
$S\leq\left\langle y,y^{a_{j}},y^{b_{j}},y^{a_{j}b_{j}}\mid1\leq j\leq
2\delta\right\rangle $.
\end{lemma}

\begin{proof}
For $u,v\in S$ we have%
\[
\lbrack u^{-1},v]=[[u,y],v]=y^{-1}y^{u}y^{-uv}y^{v}.
\]
The lemma follows since $S$ is generated by $2\delta$ commutators.
\end{proof}

\begin{lemma}
\label{gss}If $y\in G$ satisfies $[F,y]\neq\{1\}$ then $y$ has the $\frac
{1}{4}$-fsp on $V$.
\end{lemma}

\begin{proof}
(cf. \cite{GSS}, proof of Theorem 5.3) For $x\in G$ put $c(x)=\dim
\mathrm{C}_{V}(x)$. As $\mathrm{C}_{V}([x_{1},x_{2}])\geq\mathrm{C}_{V}%
(x_{1})\cap\mathrm{C}_{V}(x_{1})\cdot x_{2}$ we have%
\begin{equation}
c([x_{1},x_{2}])\geq2c(x_{1})-\dim(V). \label{c-formula}%
\end{equation}
The hypotheses imply that every abelian normal subgroup of $G$ is cyclic and
acts freely on $V$. It follows by a theorem of P. Hall (see \cite{As},
\textbf{23.9} or \cite{Go}, Theorem 5.4.9) that $F$ is metabelian. Thus if
$1\neq t\in F^{\prime}\cup\mathrm{Z}(F)$ then $c(t)=0$. Now there exists $x\in
F$ such that $[x,y]\neq1$. If $[x,y]\in\mathrm{Z}(F)$ we may infer using
(\ref{c-formula}) that $c(y)\leq\frac{1}{2}\dim(V)$. If $[x,y]\notin
\mathrm{Z}(F)$ then for some $h\in F$ we have $1\neq\lbrack\lbrack x,y],h]\in
F^{\prime}$. Then using (\ref{c-formula}) twice gives%
\[
c(y)\leq\frac{1}{2}\left(  c[x,y]+\dim(V)\right)  \leq\frac{1}{2}\left(
\frac{1}{2}\dim(V)+\dim(V)\right)  =\frac{3}{4}\dim(V).
\]
The result follows.
\end{proof}

\bigskip

In view of the preceding lemma, we may suppose for the rest of this subsection
that $[F,Y]=\{1\}$. Since $Y$ does not centralize any non-central abelian
chief factor of $G$, this implies that $F$ is contained in the hypercentre of
$G$, and hence that $[F,\gamma_{\omega}(G)]=1$. But $G=G^{\prime}\left\langle
Y\right\rangle $ implies $G=\gamma_{\omega}(G)\left\langle Y\right\rangle $;
therefore $[F,G]=1$, and so $F=\mathrm{Z}(G)$.

Now let $F^{\ast}=\mathrm{F}^{\ast}(G_{0})$ denote the generalized Fitting
subgroup of $G_{0}$ (see \cite{As}, Section 31).Then%
\[
\mathrm{C}_{G_{0}}(F^{\ast})=\mathrm{Z}(F^{\ast})=F\cap G_{0}.
\]

$\medskip$

\emph{Case 1.} Suppose that $F^{\ast}\leq F$. Then $F^{\ast}$ is central in
$G_{0}$ and it follows that $G_{0}=F\cap G_{0}\leq\mathrm{Z}(G)$. Hence
$\mathrm{C}_{V}(\left\langle Y\right\rangle )$ is a $G$-submodule of $V$; as
$V$ is faithful and irreducible for $G$ and $\left\langle Y\right\rangle
\neq1$ it follows that $\mathrm{C}_{V}(\left\langle Y\right\rangle )=0$. Hence
$Y$ has the $1/r$-fsp on $V$ by Remark (v).

$\medskip$

\emph{Case 2. }Suppose that $F^{\ast}\nleq F$. Then $F^{\ast}=E\cdot F_{0}$
where $E$ is a non-empty central product of quasisimple groups, $F_{0}=F\cap
G_{0}$, and $E$ is characteristic in $G_{0}$ with centre $Z_{0}=E\cap F$. Let
$\overline{N}=N/Z_{0}$ be a minimal normal subgroup of $G/Z_{0}$ contained in
$E/Z_{0}$. Then $\overline{N}=\overline{S}_{1}\times\cdots\times\overline
{S}_{n}$, where each $\overline{S}_{i}=S_{i}/Z_{0}$ is a simple group. By
hypothesis, there exists $y\in Y$ such that $y$ moves at least one of these
factors; say $y$ moves $\overline{S}_{1}$. Then $[S_{1},S_{1}^{y}]=1$, and
Lemma \ref{comm} now shows that $S_{1}\leq\left\langle y_{1},\ldots
,y_{t}\right\rangle $ where $t=1+6\delta$ and each $y_{j}$ is a conjugate of
$y$.

We claim that $\mathrm{C}_{V}(S_{1})=0$. Accepting the claim for now, it
follows by Remark (v) that some $y_{j}$ has the $1/t$-fsp on $V$; as $y_{j}$
is conjugate to $y$ we may conclude that $y$ has the $1/t$-fsp on $V$.

Since $V$ is a primitive irreducible $\mathbb{F}_{p}G$-module it is a direct
sum of copies of some simple $\mathbb{F}_{p}N$-module $W$. If $\mathrm{C}%
_{V}(S_{1})\neq0$ then $W$ is a composition factor of the $\mathbb{F}_{p}%
N$-module $\mathrm{C}_{V}(S_{1})$, so $W(S_{1}-1)=0$. But then $V(S_{1}-1)=0$,
a contradiction since $V$ is faithful for $G$. Thus $\mathrm{C}_{V}(S_{1})=0$
as claimed.$\medskip$

The first claim of Proposition \ref{epsilon-prop} clearly follows.

\subsubsection{Primitive $G$-sets}

Let $G$ be a group and $Y$ a subset of $G$ of size $r$ satisfying the
hypotheses of Proposition \ref{epsilon-prop}. Let $\Omega$ be a primitive
$G$-set of size $n\geq3$, on which $G$ acts faithfully.

If $\left\langle Y\right\rangle $ has no fixed points in $\Omega$ then $Y$ has
the $1/r$-fpp on $\Omega$, by Remark (v). We assume henceforth that
$\left\langle Y\right\rangle $ has at least one fixed point in $\Omega$; since
$G=G_{0}\left\langle Y\right\rangle $ is transitive this implies also that
$G_{0}\neq1$.

According to \cite{DM} Theorem 4.3B (part of the O'Nan-Scott Theorem), one of
the following holds:

\begin{description}
\item[(a)] $G$ has a unique minimal normal subgroup $N=\mathrm{C}_{G}(N)$ and
$N$ acts regularly on $\Omega$;

\item[(b)] $G$ has exactly two minimal normal subgroups $N$ and $\mathrm{C}%
_{G}(N)$, and each of them acts regularly on $\Omega$;

\item[(c)] $G$ has a unique minimal normal subgroup $N$ and $\mathrm{C}%
_{G}(N)=1$.
\end{description}

Since $G_{0}>1$, in cases (a) and (c) we have $N\leq G_{0}$; in case (b) at
least one of $N$ and $\mathrm{C}_{G}(N)$ must lie in $G_{0}$, and we choose to
call that one $N$.

$\medskip$

\emph{Case 1.} Suppose that the minimal normal subgroup $N$ of $G$ contained
in $G_{0}$ acts regularly on $\Omega$. Then $\left\vert N\right\vert =n$ and
$N$ is a non-central chief factor of $G$ (\cite{DM}, Theorem 4.3B). Let
$\alpha\in\Omega$ be a fixed point for $\left\langle Y\right\rangle $. Then
for $x\in N$ and $y\in Y$ we have $(\alpha x)y=\alpha x^{y}$, so $y$ has
exactly $\left\vert \mathrm{C}_{N}(y)\right\vert $ fixed points on $\alpha
N=\Omega$. By hypothesis, there exists $y\in Y$ such that $\mathrm{C}%
_{N}(y)\neq N$. The number of fixed points of $y$ in $\Omega$ is then at most%
\[
\left\vert \mathrm{C}_{N}(y)\right\vert \leq\frac{1}{2}\left\vert N\right\vert
=\frac{1}{2}n\text{,}%
\]
so $y$ has the $\frac{1}{2}$-fpp on $\Omega$.

$\medskip$

\emph{Case 2.} The unique minimal normal subgroup $N$ of $G$ is not regular on
$\Omega$. Then $N$ is not abelian, so $N=S_{1}\times\cdots\times S_{m}$ where
each $S_{i}$ is simple and $m\geq2$ since $N\leq G_{0}$. According to
\cite{DM} Theorem 4.6A there are now two possibilities.

$\medskip$

\emph{Subcase 2.1.} $G$ acts as a group of diagonal type on $\Omega$. Fixing
an identification of each $S_{i}$ with a group $T$, we identify $\Omega$ with
with the right coset space $T^{\ast}\smallsetminus T^{(m)}$ where $T^{\ast}$
denotes the diagonal subgroup. The action of $N$ is induced by the right
regular action, so%
\[
T^{\ast}(t_{1},\ldots,t_{m})\cdot s_{1}\ldots s_{m}=T^{\ast}(t_{1}s_{1}%
,\ldots,t_{m}s_{m})
\]
for $(t_{1},\ldots,t_{m})\in T^{(m)}$ and $s_{i}\in S_{i}$. Write
$k=\left\vert T\right\vert $, so that $n=k^{m-1}$.

Let $\alpha=T^{\ast}(t_{1},\ldots,t_{m})$ be a fixed point for $y\in G$. The
stabilizer of $\alpha$ in $N$ is%
\[
N_{\alpha}=\left\{  \left(  u^{t_{1}},\ldots,u^{t_{m}}\right)  \mid u\in
T\right\}  ,
\]
so for $x\in N$ we have%
\begin{equation}
(\alpha x)\cdot y=\alpha x\Longleftrightarrow\alpha x^{y}=\alpha
x\Longleftrightarrow x^{y}x^{-1}=\left(  u^{t_{1}},\ldots,u^{t_{m}}\right)
\text{, some }u\in T\text{.} \label{fix1}%
\end{equation}
Suppose that the conjugation action of $y$ permutes $S_{1},S_{2},\ldots,S_{e}
$ cyclically, and that (\ref{fix1}) holds with $x=s_{1}s_{2}\ldots s_{m}$
($s_{i}\in S_{i}$). Then $s_{2},\ldots,s_{e}$ are uniquely determined by $u$
and $s_{1}$. Thus if $y$ has $q=q(y)$ cycles in its action on $\{S_{1}%
,S_{2},\ldots,S_{m}\}$, then the number of $x\in N$ satisfying (\ref{fix1}) is
at most $k\cdot k^{q}$. The mapping $x\mapsto\alpha x$ from $N$ to $\Omega$ is
surjective and each fibre has size $k$. It follows that $y$ has at most
$k^{q}$ fixed points in $\Omega$.

Suppose that some $y\in Y$ moves at least $3$ of the $S_{i}$. Then $q(y)\leq
m-2$, and so the number of fixed points of $y$ in $\Omega$ is at most%
\[
k^{q(y)}\leq k^{m-2}=nk^{-1}.
\]
If this holds for no element $y\in Y$, then $Y$ must contain an element
$y_{1}$ that acts as a transposition $(12)$, say, on $\{S_{1},S_{2}%
,\ldots,S_{m}\}$.

Assume first that $m\geq3$. There exists $g\in G$ such that $S_{2}^{g}=S_{3}$;
then $y=[y_{1},g]$ moves at least $3$ of the $S_{i}$, and hence fixes at most
$nk^{-1}$ points in $\Omega$. It follows that $y_{1}$ has at most
$(n+nk^{-1})/2$ fixed points.

Suppose now that $m=2$. Set $y=[y_{1},a]$ where $1\neq a\in S_{1}$. Suppose
that $y$ fixes $\alpha\in\Omega$. Each element of $\Omega$ can be put uniquely
in the form $\alpha x$ with $x=(s,1),~s\in S_{1}$, and then (\ref{fix1}) gives%
\begin{align*}
(\alpha x)\cdot y=\alpha x  &  \Longleftrightarrow(s^{a}s^{-1},1)=(u^{t_{1}%
},u^{t_{2}})\\
&  \Longleftrightarrow s\in\mathrm{C}_{S_{1}}(a).
\end{align*}
Thus $y$ has at most $\left\vert \mathrm{C}_{S_{1}}(a)\right\vert \leq\frac
{1}{5}\left\vert S_{1}\right\vert =\frac{1}{5}n$ fixed points in $\Omega$. It
follows that $y_{1}$ has at most $\frac{3}{5}n$ fixed points.

\bigskip

Thus in any case, we may conclude (since $k\geq60$) that $y$ contains an
element with the $\varepsilon$-fpp as long as $\varepsilon\leq\frac{2}{5}$.

$\medskip$

\emph{Subcase 2.2.} $G$ is contained in a wreath product $W=H\wr\pi(G)$ where
$H\leq\mathrm{Sym}(\Gamma)$, $\pi:G\rightarrow\mathrm{Sym}(d)$ where $d>1$,
and $W$ acts on $\Omega=\Gamma^{(d)}$ by the product action. In this case
$N=N_{1}\times\cdots\times N_{d}\leq H^{(d)}$, and $G$ permutes the factors
$N_{i}$ via $\pi$. Put $k=\left\vert \Gamma\right\vert $, so $n=k^{d}$. Note
that $k\geq5$ since $N$ is not soluble.

Suppose that $y=b\cdot\pi(y)$ fixes $(\gamma_{1},\ldots,\gamma_{d})\in
\Gamma^{(d)}$, where $b\in H^{(d)}$. If $\pi(y)$ has a cycle $(1,2,\ldots,e)$
then $\gamma_{i+1}=\gamma_{i}^{b_{i}}$ for $i=1,\ldots,e-1$. Thus if $\pi(y)$
has $q=q(y)$ cycles then the number of fixed points of $y$ in $\Omega$ is at
most $k^{q}$.

By hypothesis, there exists $y\in Y$ such that $\pi(y)\neq1$. Then $q(y)\leq
d-1$ and so $y$ has at most $k^{d-1}\leq n/5$ fixed points in $\Omega$. Thus
$y$ has the $\frac{4}{5}$-fpp on $\Omega$.

\bigskip

The proof of Proposition \ref{epsilon-prop} is now complete.

\subsection{Small chief factors}

We quote a mild generalization of a well-known result due to Gasch\"{u}tz
\cite{Gch} ; the proof given (for example) in \cite{FJ}, Lemma 15.30 adapts
easily to yield this version:

\begin{lemma}
\label{Gasch}Let $Y_{1}\subseteq G$ and $D\vartriangleleft G$. Suppose that%
\[
G=D\left\langle y_{1},\ldots,y_{d},Y_{1}\right\rangle
\]
where $d\geq\mathrm{d}(G)$. Then there exist $h_{1},\ldots,h_{d}\in D$ such
that $G=\left\langle h_{1}y_{1},\ldots,h_{d}y_{d},Y_{1}\right\rangle $.
\end{lemma}

We have defined $\delta$ to be a number such that each element of every simple
group is a product of $\delta$ commutators, and observed that in fact one can
take $\delta=1$ (Subsection \ref{facts}).

\begin{lemma}
\label{4 factors}Suppose that $M=S_{1}\times S_{2}$ and $\alpha\in
\mathrm{Aut}(M)$ satisfies $S_{1}^{\alpha}=S_{2},~S_{2}^{\alpha}=S_{1}$. Let
$C=\left\{  [x,y]\mid x,y\in S_{1}\right\}  $. Then%
\[
C\subseteq\lbrack M,\alpha]^{\ast4}.
\]
If $S_{1}$ is simple then%
\[
M=[M,\alpha]^{\ast8\delta}.
\]

\end{lemma}

\begin{proof}
Let $x,y\in S_{1}$. Then%
\begin{align*}
\lbrack y,x^{-1}]  &  =[y,[x,\alpha]]=[y,\alpha][y,\alpha^{-x}][[y,\alpha
^{-x}],\alpha]\\
&  =[y,\alpha]\cdot\lbrack y^{x^{-1}},\alpha^{-1}]^{x}\cdot\lbrack\lbrack
y,\alpha^{-x}],\alpha]
\end{align*}
and the middle factor lies in $[M,\alpha]^{\ast2}$ because for any $z\in
S_{1}$ we have%
\begin{align*}
\lbrack z,\alpha^{-1}]^{x}  &  =[zx,\alpha^{-1}][x^{\alpha^{-1}},\alpha]\\
&  =[(zx)^{-\alpha^{-1}},\alpha][x^{\alpha^{-1}},\alpha]
\end{align*}
(for the final equality note that $(zx)^{\alpha}$ commutes with $zx$). This
establishes the first claim.

If $S_{1}$ is simple, then $S_{1}=C^{\ast\delta}$, so $M=C^{\ast\delta}%
\cdot(C^{\ast\delta})^{\alpha}\subseteq\lbrack M,\alpha]^{\ast8\delta}$ since
$[M,\alpha]=[M,\alpha]^{\alpha}$.
\end{proof}

\bigskip

For technical reasons, we need to introduce a slightly smaller analogue of the
subgroup $G_{0}$:

$\medskip$

\textbf{Definition} For a group $G,$ let%
\begin{equation}
G_{1}=%
{\displaystyle\bigcap\limits_{M\in\mathcal{C}(G)}}
\mathrm{C}_{G}(M) \label{G012}%
\end{equation}
where $\mathcal{C}(G)$ denote the set of all non-abelian chief factors of $G$
that have composition length at most two. We shall call such chief factors `bad'.

\bigskip

\emph{Remarks.} (vi) A non-abelian chief factor belongs to $\mathcal{C}(G)$ if
and only if it is either simple or a product of two simple groups. Hence such
a factor that occurs inside $G_{1}$ is a product of at least $3$ simple groups.

(vii) $(G^{2})^{(3)}G_{1}/G_{1}$ is semisimple: for if $M\in\mathcal{C}(G)$
then $G/\mathrm{C}_{G}(M)$ is an extension of $M$ by $\mathrm{Out}(M)$,
$\mathrm{Out}(M)$ is isomorphic to $\mathrm{Out}(S)$ or $\mathrm{Out}(S)\wr
C_{2}$ where $S$ is simple, and $\mathrm{Out}(S)^{(3)}=1$ (Proposition
\ref{schreier}).

(viii) If $G>1$ then $G_{1}<G$ or $G^{\prime}<G$.

\begin{proposition}
\label{D-lifting}Let $G$ be a group and $W=\{w_{1},,\ldots,w_{s}\}$ a subset
such that $G=D\left\langle W\right\rangle $ where $D\leq G_{0}\cap
G^{(4)}G_{1}$. Then there exist elements $b_{ij}\in D$ such that%
\[
G=\left\langle w_{i}^{b_{ij}}\mid i=1,\ldots,s,~j=1,\ldots,m\right\rangle
(D\cap G_{1})
\]
where $m=1+8\delta\mathrm{d}(G)$.
\end{proposition}

\begin{proof}
Note that $G_{1}=\bigcap_{M\in\mathcal{C}\smallsetminus\mathcal{S}}%
\mathrm{C}_{G_{0}}(M)$, where $\mathcal{S}$ denotes the set of all simple
chief factors of $G$. The section $G^{(4)}G_{1}/G_{1}$ is semisimple, and is a
product of minimal normal subgroups of $G/G_{1}$ belonging to $\mathcal{C}$.
We may suppose that $D\cap G_{1}=1$. In that case, $D$ is a product of minimal
normal subgroups of $G$ belonging to $\mathcal{C}\smallsetminus\mathcal{S}$.

Let $M=S_{1}\times S_{2}$ be one of these. Then $D$ normalizes $S_{1}$ and
$S_{2}$, so there exists $y\in W$ such that $S_{1}^{y}=S_{2}$ and $S_{2}%
^{y}=S_{1}$. Now Lemma \ref{4 factors} shows that%
\[
M=[M,y]^{\ast8\delta}.
\]
As $D$ is the direct product of such normal subgroups $M$ of $G$, it follows
that%
\[
D=[D,y]^{\ast8\delta}.
\]

If $r<d=\mathrm{d}(G)$ put $w_{r+1}=\ldots=w_{d}=w_{r}$. Now applying Lemma
\ref{Gasch} we find elements $h_{j}\in D$ such that $G=\left\langle h_{1}%
w_{1},\ldots,h_{d}w_{d},w_{d+1},\ldots,w_{r}\right\rangle $. Each $h_{j}$ lies
in the subgroup generated by $W$ and $8\delta$ $D$-conjugates of the $w_{i}$.
The result follows.
\end{proof}

\subsection{Lifting generators\label{liftigen}}

Recall that a chief factor of $G$ is \emph{bad} if it is either simple or the
product of two simple groups.

\begin{proposition}
\label{prop5.1new}Let $G=N\left\langle y_{1},\ldots,y_{m}\right\rangle $ be a
$d$-generator group where $N$ is a non-central minimal normal subgroup of $G$.
If $N$ is non-abelian, assume that $N$ is not bad. Let%
\[
\mathcal{X}=\left\{  \mathbf{a}\in N^{(m)}\mid\left\langle y_{1}^{a_{1}%
},\ldots,y_{m}^{a_{m}}\right\rangle =G\right\}  .
\]
\emph{(i)} Suppose that $N$ is abelian and that $y_{j}$ has the $\varepsilon
$-fsp on $N$ for at least $k$ values of $j$. Then%
\[
\left\vert \mathcal{X}\right\vert \geq\left\vert N\right\vert ^{m}%
(1-\left\vert N\right\vert ^{d-k\varepsilon}).
\]
\emph{(ii)} Suppose that $N$ is non-abelian and that $y_{j}$ has the
$\varepsilon$-fgp on $N$ for at least $k$ values of $j$, where $k\varepsilon
\geq\max\{2d+4,C\}$ for a certain absolute constant $C$. Then%
\[
\left\vert \mathcal{X}\right\vert \geq\left\vert N\right\vert ^{m}%
(1-2^{2-k\varepsilon}).
\]

\end{proposition}

\begin{proof}
Part (i) is \cite{NS}, Proposition 5.1(i). In the situation of (ii), the proof
of \cite{NS}, Proposition 5.1(ii) shows that $\left\vert \mathcal{X}%
\right\vert \geq\left\vert N\right\vert ^{m}(1-z)$ where $z\leq\zeta
(k\varepsilon)-1$ (Riemann zeta function). A crude estimate gives
$\zeta(t)-1\leq2^{2-t}$ for $t>2$.
\end{proof}

\bigskip

The main result is now

\begin{theorem}
\label{gens}Let $G$ be a group and $K\leq G_{0}$ a normal subgroup of $G$. Let
$Y=\{y_{1},,\ldots,y_{r}\}$ be a subset of $G$ such that $G=G^{\prime
}\left\langle Y\right\rangle =K\left\langle Y\right\rangle $. Then there exist
elements $x_{ij}\in K$ such that%
\[
G=\left\langle y_{i}^{x_{ij}}\mid i=1,\ldots,r,~j=1,\ldots,k\right\rangle
\]
where%
\begin{align}
k=\max\{(1+2\mathrm{d}(G)\widehat{r})  &  (1+8\delta\mathrm{d}(G)),~\widehat
{r}C\}\label{k}\\
&  =f_{0}(r,\mathrm{d}(G))\leq C_{0}r\mathrm{d}(G)^{2},\nonumber
\end{align}
$\widehat{r}=\max\{r,~1+6\delta\}$, and $C$ and $C_{0}$ are absolute constants.
\end{theorem}

\begin{corollary}
If $G$ has no simple chief factors and $G=G^{\prime}\left\langle y_{1}%
,,\ldots,y_{r}\right\rangle $ then $G=\left\langle y_{i}^{c_{ij}}\mid
i=1,\ldots,r,~j=1,\ldots,k\right\rangle $.
\end{corollary}

\begin{proof}
[Proof of Theorem \ref{gens}]Write $d=\mathrm{d}(G)$ and set $\varepsilon
=\widehat{r}^{-1}$. Let $N$ be a non-central chief factor of $G$. We will say
that a subset $W$ of $G$ has the $(k,\varepsilon)$\emph{-property} w.r.t. $N$
if $N$ is abelian and at least $k$ elements of $W$ have the $\varepsilon
/2$-fsp on $N$, or if $N$ is a product of at least $3$ simple groups and at
least $k$ elements of $W$ have the $\varepsilon$-fpp on the set of simple
factors of $N$. According to Theorem \ref{fppThm}, the set $\left\{
y_{1},\ldots,y_{r}\right\}  $ has the $(1,\varepsilon)$-property w.r.t. $N$.

Put $D=K\cap G^{(4)}G_{1}$. We begin by proving

\begin{description}
\item[\textbf{(*)}] there exists elements $a_{ij}\in K$ such that%
\begin{equation}
G=D\left\langle y_{i}^{a_{ij}}\mid i=1,\ldots,r,~j=1,\ldots,k_{1}\right\rangle
\label{genmodD}%
\end{equation}
where $k_{1}=1+2d\widehat{r}$.
\end{description}

Replacing $G$ by $G/D$ for the moment, we may assume that $K$ is soluble. If
$K=1$ we can take all $a_{ij}=1$ and there is nothing to prove.

Suppose that $K>1$ and let $N$ be a minimal normal subgroup of $G$ contained
in $K$; then $N$ is abelian. Arguing by induction on $\left\vert K\right\vert
$, we may suppose that $G=N\left\langle W\right\rangle $ where%
\[
W=\left\{  y_{i}^{a_{ij}}\mid i=1,\ldots,r,~j=1,\ldots,k_{1}\right\}
\]
and each $a_{ij}\in K$. If $N\leq\mathrm{Z}(G)$ then $G^{\prime}%
\leq\left\langle W\right\rangle $, so $\left\langle W\right\rangle \geq
G^{\prime}\left\langle Y\right\rangle =G$ and we are done.

If $N$ is non-central, the set $W$ has the $(k_{1},\varepsilon)$-property
w.r.t. $N$. As $d-k_{1}\varepsilon/2<0$, Proposition \ref{prop5.1new}(i) shows
that there exist elements $b_{ij}\in N$ such that%
\[
G=\left\langle y_{i}^{a_{ij}b_{ij}}\mid i=1,\ldots,r,~j=1,\ldots
,k_{1}\right\rangle ,
\]
and (\ref{genmodD}) follows on replacing $a_{ij}$ by $a_{ij}b_{ij}$. This
completes the proof of (*).

Now we apply Proposition \ref{D-lifting} to find elements $c_{ijl}\in D$ such
that%
\[
G=G_{1}\left\langle y_{i}^{a_{ij}c_{ijl}}\mid i=1,\ldots,r,~j=1,\ldots
,k_{1},~l=1,\ldots,m\right\rangle
\]
where $m=1+8\delta d$.

If $G_{1}=1$ we are done. Otherwise, let $N$ be a minimal normal subgroup of
$G$ contained in $G_{1}$, and suppose inductively that $G=N\left\langle
W\right\rangle $ where $W=\left\{  y_{i}^{x_{ij}}\mid i=1,\ldots
,r,~j=1,\ldots,k_{1}m\right\}  $. If $N$ is abelian we deduce as above that
$G$ is generated by a set of the form $\left\{  y_{i}^{x_{ij}b_{ij}}\mid
i=1,\ldots,r,~j=1,\ldots,k_{1}m\right\}  $, and the result follows since
$k=\max\{k_{1}m,\widehat{r}C\}\geq k_{1}m$.

Suppose that $N$ is non-abelian; then $N$ is not bad. If $k_{1}m<\widehat{r}%
C$, enlarge the family $W$ by repeating some of its elements to obtain a
family containing $k$ conjugates of each $y_{i}$ ($i=1,\ldots,r$). Then in any
case, $W$ has the $(k,\varepsilon)$-property w.r.t. $N$; Proposition
\ref{prop5.1new}(ii) now shows that that $G$ is generated by a set of the form
$\left\{  y_{i}^{x_{ij}c_{ij}}\mid i=1,\ldots,r,~j=1,\ldots,k\right\}  $ with
$c_{ij}\in N$, as required. (Note that $k\varepsilon\geq\max\{2d+4,C\}$ since
$k_{1}m>18d$.)
\end{proof}

\bigskip

\emph{Remarks.} (ix) Recall that $\delta=1$ if we accept the validity of the
Ore Conjecture (Subsection \ref{facts}).

(x) If we assume that $K\leq G_{1}$ we can take $k=\widehat{r}\cdot
\max\{2d+4,C\}=O(rd)$. In particular, if $G$ has no bad chief factors then the
Corollary holds with this smaller value of $k$.

\section{Commutators\label{commsection}}

In this section we begin the proof of the two main `commutator' results.

\begin{theorem}
\label{abs}Let $G=\left\langle g_{1},\ldots,g_{r}\right\rangle $ be a group
and $H$ a normal subgroup of $G$. Then%
\[
\lbrack H,G]=\left(
{\displaystyle\prod\limits_{i=1}^{r}}
[H,g_{i}]^{H}\right)  ^{\ast f_{3}}=\left(
{\displaystyle\prod\limits_{i=1}^{r}}
[H,g_{i}][H,g_{i}^{-1}]\right)  ^{\ast f_{3}}%
\]
where $f_{3}=O(rd)=O(r^{2})$ depends only on $r$ and $d=\mathrm{d}(G)$.
\end{theorem}

\begin{theorem}
\label{rel}Let $G=G^{\prime}\left\langle g_{1},\ldots,g_{r}\right\rangle $ be
a group and $H$ a normal subgroup of $G$ such that $H\left\langle g_{1}%
,\ldots,g_{r}\right\rangle =G$.\newline\emph{(i)} If $H\leq G_{0}$ then%
\[
\lbrack H,G]=\left(
{\displaystyle\prod\limits_{i=1}^{r}}
[H,g_{i}]^{H}\right)  ^{\ast f_{4}}=\left(
{\displaystyle\prod\limits_{i=1}^{r}}
[H,g_{i}][H,g_{i}^{-1}]\right)  ^{\ast f_{4}},
\]
\emph{(ii)} if $H\leq G_{1}$ then%
\[
\lbrack H,G]=\left(
{\displaystyle\prod\limits_{i=1}^{r}}
[H,g_{i}]\right)  ^{\ast f_{5}},
\]
where $f_{4}=O(r^{5}d^{6})$ and $f_{5}=O(rd)$ depend only on $r$ and
$d=\mathrm{d}(G)$.
\end{theorem}

These are not quite the same as Theorems \ref{ThmB} and \ref{ThmC}, which
refer to a \emph{symmetric} set $Y=\{y_{1},\ldots,y_{r}\}$, and omit the
factors $[H,g_{i}^{-1}]$. To deduce the stated results, note that if $Y$ is
symmetric then%
\[
\left(
{\displaystyle\prod\limits_{j=1}^{r}}
[H,y_{j}]\right)  ^{\ast2r}\supseteq%
{\displaystyle\prod\limits_{j=1}^{r}}
[H,y_{i}][H,y_{i}^{-1}],
\]
and we may take $f_{1}=2rf_{3}$, $f_{2}=2rf_{4}$; of course if we are allowed
to order $Y$ so that $y_{2i}=y_{2i-1}^{-1}$ for $i=1,\ldots,r/2$ then we can
take $f_{1}=f_{3}$ and $f_{2}=f_{4}$.

\subsection{Acceptable normal subgroups}

Suppose that $A<B$ are normal subgroups of a group $G$. Recall that $B/A$ is a
\emph{bad chief factor} of $G$ if $B/A$ is a minimal normal subgroup of $G/A$
and $B/A$ is either simple or the direct product of two simple groups. Thus
$G_{1}$ (defined in Subsection \ref{fix}) is precisely the intersection of the
centralizers of all bad chief factors of $G$.

A normal subgroup $H$ of $G$ is said to be \emph{acceptable} in $G$ if

\begin{description}
\item[(a)] $H=[H,G]$ and

\item[(b)] if $A<B\leq H$ \emph{are normal subgroups of} $G$ \emph{then} $B/A$
\emph{is not a bad chief factor of }$G$\emph{.}
\end{description}

Here we show how the main results may be reduced to the consideration of
acceptable normal subgroups.

\begin{lemma}
\label{accGone}$H\vartriangleleft G$ is acceptable if and only if $H=[H,G]\leq
G^{\prime}\cap G_{1}$.
\end{lemma}

\begin{proof}
If $H\geq B>A$ and $B/A$ is a bad chief factor then $H$ does not centralize
$B/A$, so $H\nleq G_{1}$. Conversely, if $H\nleq G_{1}$ then $H$ does not
centralize some bad chief factor $B/A$; then $(B\cap H)A=A$ so $(B\cap
H)/(A\cap H)\cong B/A$ and $A\cap H<B\cap H\leq H$ contradicts (b), showing
that $H$ is not acceptable.
\end{proof}

\bigskip

The next result is elementary; it is the general form of facts (A) and (B)
mentioned in the introduction:

\begin{lemma}
\label{basiccomm}Let $H\vartriangleleft G=G^{\prime}\left\langle g_{1}%
,\ldots,g_{r}\right\rangle $ and let $n\geq1$. Then%
\[
\lbrack H,G]=[H,_{n}G]%
{\displaystyle\prod\limits_{i=1}^{r}}
[H,g_{i}].
\]
If in addition we have $G=H\left\langle g_{1},\ldots,g_{r}\right\rangle $ then%
\[
G=[H,_{n}G]\left\langle g_{1},\ldots,g_{r}\right\rangle .
\]

\end{lemma}

\begin{proof}
The first claim is \cite{NS}, Lemma 2.4 or \cite{S2}, Prop. 1.2.5. For the
second, we argue by induction on $n$ and reduce to the case where $[H,G]=1$.
Then $G^{\prime}\leq\left\langle g_{1},\ldots,g_{r}\right\rangle $ and the
claim is evident.
\end{proof}

\begin{lemma}
\label{L1}Let $G$ be a group and $\alpha,\beta\in\mathrm{Aut}(G)$. Then%
\begin{align*}
\lbrack G,\alpha]^{G}  &  \subseteq\lbrack G,\alpha][G,\alpha^{-1}],\\
\lbrack G,\alpha\beta]^{G}  &  \subseteq\lbrack G,\beta][G,\beta
^{-1}][G,\alpha][G,\alpha^{-1}][G,\beta][G,\beta^{-1}]\\
\lbrack G,\alpha^{-1}\beta\alpha]^{G}  &  \subseteq\lbrack G,\beta
][G,\beta^{-1}][G,\alpha][G,\alpha^{-1}]
\end{align*}

\end{lemma}

\begin{proof}%
\begin{align*}
\lbrack x,\alpha]^{y}  &  =[xy,\alpha][y^{\alpha},\alpha^{-1}],\\
\lbrack x,\alpha\beta]  &  =[x,\beta][x,\alpha][[x,\alpha],\beta]\\
\lbrack x,\alpha^{-1}\beta\alpha]  &  =[x^{\alpha^{-1}},\beta][[x^{\alpha
^{-1}},\beta],\alpha].
\end{align*}

\end{proof}

\begin{lemma}
\label{LS}Let $G$ be a quasisimple group and $\alpha\in\mathrm{Aut}(G)$. Put
$\overline{G}=G/\mathrm{Z}(G)$. If $\left\vert [\overline{G},\alpha
]\right\vert ^{s}\geq\left\vert \overline{G}\right\vert $ then%
\[
G=\left(  [G,\alpha]^{G}\right)  ^{\ast cs}%
\]
where $c\in\mathbb{N}$ is an absolute constant.
\end{lemma}

\begin{proof}
Proposition \ref{conjclass} shows that if $Y$ is a normal subset of
$\overline{G}$ with $\left\vert Y\right\vert ^{s}\geq\left\vert \overline
{G}\right\vert $ then $\overline{G}=Y^{\ast c^{\prime}s}$, where $c^{\prime}$
is an absolute constant. Applying this with $Y=X\mathrm{Z}(G)/\mathrm{Z}(G)$
where $X=[G,\alpha]^{G}$ we get%
\[
G=X^{\ast c^{\prime}s}\mathrm{Z}(G).
\]
Now for $g,h,k\in G$ we have
\[
\lbrack\lbrack g,\alpha]^{k},h]=[g,\alpha]^{-k}[g,\alpha]^{kh}=[g^{-1}%
,\alpha]^{gk}[g,\alpha]^{kh}\in X^{\ast2},
\]
so if $w\in X^{\ast c^{\prime}s}$ then%
\[
\lbrack w,h]\in X^{\ast2c^{\prime}s}.
\]
According to Proposition \ref{commwidth}, there exists an absolute constant
$\delta^{\ast}$ such that every element of $G$ is a product of $\delta^{\ast}$
commutators (In fact $\delta^{\ast}=2$). It follows that%
\[
G=X^{\ast2c^{\prime}s\delta^{\ast}}.
\]

\end{proof}

\begin{lemma}
\label{L3}Let $G=\left\langle g_{1},\ldots,g_{r}\right\rangle $ and suppose
that $T\vartriangleleft G$ is quasisemisimple with one or two simple
composition factors. Then%
\[
T=\left(
{\displaystyle\prod\limits_{i=1}^{r}}
[T,g_{i}]^{T}\right)  ^{\ast k_{0}r},
\]
where $k_{0}$ is an absolute constant.
\end{lemma}

\begin{proof}
Suppose that $T$ is quasisimple, with centre $Z$. Put $\overline{T}=T/Z$. Then
$\mathrm{C}_{\overline{T}}(G)=1$ so $\left\vert \overline{T}\right\vert \leq%
{\textstyle\prod\nolimits_{i=1}^{r}}
[\overline{T},g_{i}]$ and so $\left\vert [\overline{T},g_{i}]\right\vert
\geq\left\vert \overline{T}\right\vert ^{1/r}$ for some $i$. Now Lemma
\ref{LS} implies that $T=\left(  [T,g_{i}]^{T}\right)  ^{\ast cr}$.

If $T$ is not quasisimple, then $T=S_{1}S_{2}$ with each $S_{i}$ quasisimple
and $[S_{1},S_{2}]=1$. If $G$ normalizes the factors $S_{i}$, we apply the
preceding paragraph to each factor and obtain the same result as before.
Otherwise, $G$ permutes them transitively by conjugation. The action of $G$
lifts to an action on the universal cover $\widetilde{T}=\widetilde{S_{1}%
}\times\widetilde{S_{2}}$, and for some $i$ we have $\widetilde{S_{1}}^{g_{i}%
}=\widetilde{S_{2}},~\widetilde{S_{2}}^{g_{i}}=\widetilde{S_{1}}$. Let $C_{j}$
denote the set of commutators in $\widetilde{S_{j}}$; then Lemma
\ref{4 factors} shows that%
\[
C_{j}\subseteq\lbrack\widetilde{T},g_{i}]^{\ast4}%
\]
for $j=1,~2$. Since $\widetilde{S_{j}}=C_{j}^{\ast\delta^{\ast}}$ (Proposition
\ref{commwidth}), it follows that%
\[
\widetilde{T}=[\widetilde{T},g_{i}]^{\ast4\delta^{\ast}},
\]
which implies $T=[T,g_{i}]^{\ast4\delta^{\ast}}.$

The result follows on setting $k_{0}=\max\{c,~4\delta^{\ast}\}$.
\end{proof}

\bigskip

Let us say that $N\vartriangleleft G$ is \emph{narrow} if
\[%
{\displaystyle\bigcap\limits_{T\in\mathcal{M}}}
T\leq\mathrm{Z}(N)
\]
where $\mathcal{M}$ is the set of normal subgroups $T$ of $G$ contained in $N$
such that $N/T$ is semisimple with composition length at most two. This is
equivalent to saying that $N/\mathrm{Z}(N)$ is a direct product of bad chief
factors of $G$ (occurring as minimal normal subgroups of $G/\mathrm{Z}(N)$).

\begin{lemma}
\label{L4}Let $G=\left\langle g_{1},\ldots,g_{r}\right\rangle $ and let $N$ be
a perfect narrow normal subgroup of $G$. Then%
\[
N=\left(
{\displaystyle\prod\limits_{i=1}^{r}}
[N,g_{i}]^{N}]\right)  ^{\ast k_{0}r}%
\]
where $k_{0}$ is given in Lemma \ref{L3}.
\end{lemma}

\begin{proof}
The hypotheses imply that $N$ is a central product $N=T_{1}\ldots T_{n}$ where
each $T_{i}$ is a quasisemisimple normal subgroup of $G$ having one or two
simple composition factors. As the $T_{i}$ commute elementwise the claim
follows from Lemma \ref{L3}.
\end{proof}

\begin{proposition}
\label{accept}Let $G=\left\langle g_{1},\ldots,g_{r}\right\rangle $ and let
$H\vartriangleleft G$. Then $G$ has normal subgroups $H_{3}\leq H_{2}\leq
H_{1}\leq\lbrack H,G]$ such that%
\begin{align}
\lbrack H,G]  &  =%
{\displaystyle\prod\limits_{i=1}^{r}}
[H,g_{i}]\cdot H_{1},\label{line 1}\\
H_{2}  &  =\left(
{\displaystyle\prod\limits_{i=1}^{r}}
[H_{2},g_{i}]^{H}\right)  ^{\ast k_{0}r}\cdot H_{3}, \label{line 2}%
\end{align}
$H_{1}/H_{2}$ is acceptable in $G/H_{2}$ and $H_{3}$ is acceptable in $G$.
\end{proposition}

\begin{proof}
Let $G_{s}$ be the soluble residual of $G$ and set $H_{1}=[H,_{\omega
}G],~H_{2}=[H_{1},G_{s}]$. Let $D$ be the intersection of all
$M\vartriangleleft G$ such that $M<H_{2}$ and $H_{2}/M$ is either simple or a
product of two simple groups, and put $H_{3}=[D,G_{s}]$.

Then (\ref{line 1}) follows from Lemma \ref{basiccomm}. Also $H_{1}=[H_{1},G]$
and $H_{1}/H_{2}$ is soluble, so $H_{1}/H_{2}$ is acceptable in $G/H_{2}$.

Now%
\[
\lbrack H_{3},G]\geq\lbrack H_{3},G_{s}]=[D,G_{s},G_{s}]=[D,G_{s}]=H_{3}%
\]
since $G_{s}$ is perfect. To complete the proof that $H_{3}$ is acceptable,
suppose that $K\leq H_{3}$ is a minimal normal subgroup of $G$ and that $K$ is
either simple or a product of two simple groups. Then $G/K\mathrm{C}_{G}(K)$
is soluble by the Schreier conjecture (Proposition \ref{schreier}), so
$G_{s}\leq K\mathrm{C}_{G}(K)$ and as $K\leq H_{2}\leq G_{s}$ it follows that
$H_{2}=K\times\mathrm{C}_{H_{2}}(K)$. This implies that $\mathrm{C}_{H_{2}%
}(K)\geq D\geq K,$ a contradiction. Applying this argument to an arbitrary
quotient of $G$ we infer that $H_{3}$ is acceptable in $G$.

Finally, $H_{2}/H_{3}$ is narrow in $G/H_{3}$ so Lemma \ref{L4} gives
(\ref{line 2}).
\end{proof}

\subsection{The `Key Theorem'}

The `Key Theorem' of \cite{NS} described certain product decompositions of an
acceptable normal subgroup in a $d$-generator group. As one of us wrote in
\cite{S2}, `each part has an undesirable feature in either its hypothesis or
its conclusion'. These are now swept away in our core technical result. To
state this we need some notation:

$\medskip$

\noindent\textbf{Definition} For $\mathbf{g},\mathbf{v}\in G^{(m)}$ and $1\leq
j\leq m$,%
\[
\tau_{j}(\mathbf{g},\mathbf{v})=v_{j}[g_{j-1},v_{j-1}]\ldots\lbrack
g_{1},v_{1}].
\]

\begin{theorem}
\label{KT}There exists a function $k:\mathbb{N}^{(2)}\rightarrow\mathbb{N}$
with the following property. Let $G$ be a $d$-generator group and $H$ an
acceptable normal subgroup of $G$. Suppose that $G=H\left\langle g_{1}%
,\ldots,g_{r}\right\rangle $. Put $m=r\cdot k(d,r)$, and for $1\leq j<k(d,r)$
and $1\leq i\leq r$ set%
\[
g_{i+jr}=g_{i}.
\]
Then for each $h\in H$ there exist $\mathbf{v}(i)\in H^{(m)}$ ($i=1,\ldots
,10$) such that%
\begin{equation}
h=%
{\displaystyle\prod\limits_{i=1}^{10}}
{\displaystyle\prod\limits_{j=1}^{m}}
[v(i)_{j},g_{j}] \label{hprod}%
\end{equation}
and%
\begin{equation}
\left\langle g_{1}^{\tau_{1}(\mathbf{g},\mathbf{v}(i))},\ldots,g_{m}^{\tau
_{m}(\mathbf{g},\mathbf{v}(i))}\right\rangle =G\qquad\text{for }i=1,\ldots,10.
\label{G-gen}%
\end{equation}

\end{theorem}

In fact we can take%
\[
k(d,r)=1+\max\{r,1+6\delta\}\cdot\max\{4d+4,\widehat{C}\}\leq C_{1}dr,
\]
where $\widehat{C}$ and $C_{1}$ are absolute constants.

The proof will occupy the next three subsections. Accepting the theorem for
now, we deduce the main results stated above.

$\medskip$

\noindent\textbf{Proof of Theorem \ref{abs}.} We are given $H\vartriangleleft
G=\left\langle g_{1},\ldots,g_{r}\right\rangle $. Let $H_{3}\leq H_{2}\leq
H_{1}\leq\lbrack H,G]$ be the normal subgroups given by Proposition
\ref{accept}. Thus $H_{1}/H_{2}$ is acceptable in $G/H_{2}$ and $H_{3}$ is
acceptable in $G$. Theorem \ref{KT} shows that%
\[
H_{1}=\left(
{\displaystyle\prod\limits_{j=1}^{r}}
[H_{1},g_{j}]\right)  ^{\ast10k(d,r)}\cdot H_{2}%
\]
and that%
\[
H_{3}=\left(
{\displaystyle\prod\limits_{j=1}^{r}}
[H_{3},g_{j}]\right)  ^{\ast10k(d,r)}%
\]
where $d=\mathrm{d}(G)$. Combining these with (\ref{line 1}) and
(\ref{line 2}) from Proposition \ref{accept} we deduce that%
\[
\lbrack H,G]=\left(
{\displaystyle\prod\limits_{j=1}^{r}}
[H,g_{j}]^{H}\right)  ^{\ast f_{3}}%
\]
where $f_{3}=1+k_{0}r+20k(d,r)$; here $k_{0}$ is the absolute constant
introduced in Lemma \ref{L3}. Finally, Lemma \ref{L1} shows that
$[H,g_{j}]^{H}$ can be replaced by $[H,g_{j}][H,g_{j}^{-1}]$ for each $j$.

We observe that $f_{3}=O(r+k(d,r))=O(dr)=O(r^{2})$.

\bigskip

\noindent\textbf{Proof of Theorem \ref{rel} (i).} Now $H\vartriangleleft G$
satisfies $H\leq G_{0}$, and $G=G^{\prime}\left\langle g_{1},\ldots
,g_{r}\right\rangle =H\left\langle g_{1},\ldots,g_{r}\right\rangle $.
According to Theorem \ref{gens}, there exist element $x_{ij}\in H$ such that%
\[
G=\left\langle g_{i}^{x_{ij}}\mid i=1,\ldots,r,~j=1,\ldots,k\right\rangle
\]
where $k=f_{0}(r,\mathrm{d}(G))$. Using this generating set in Theorem
\ref{abs} gives%
\begin{align*}
\lbrack H,G]  &  =\left(  \prod_{j=1}^{k}%
{\displaystyle\prod\limits_{i=1}^{r}}
[H,g_{i}^{x_{ij}}]^{H}\right)  ^{\ast f_{3}(kr)}\\
&  =\left(  \prod_{j=1}^{k}%
{\displaystyle\prod\limits_{i=1}^{r}}
[H,g_{i}]^{H}\right)  ^{\ast f_{3}(kr)}=\left(
{\displaystyle\prod\limits_{i=1}^{r}}
[H,g_{i}]^{H}\right)  ^{\ast f_{4}}%
\end{align*}
where $f_{4}=kf_{3}(kr)$. Again, we may replace $[H,g_{j}]^{H}$ by
$[H,g_{j}][H,g_{j}^{-1}]$, by Lemma \ref{L1}.

Since $k=f_{0}(r,d)\leq C_{0}rd^{2}$ where $d=\mathrm{d}(G)$, we have
$f_{4}=O(k^{3}r^{2})=O(r^{5}d^{6})$.

\bigskip

We remark that this bound for $f_{4}(r,d)$ is very crude; a much better bound
emerges if, instead of quoting Theorem \ref{gens}, one uses the method of
proof of that theorem to reduce Theorem \ref{rel} (i) to Theorem \ref{rel} (ii).

\bigskip

\noindent\textbf{Proof of Theorem \ref{rel} (ii).} Now we assume that $H$, as
above, satisfies $H\leq G_{1}$. Put $H_{1}=[H,_{\omega}G]$. Then $H_{1}$ is
acceptable in $G$, by Lemma \ref{accGone}, and $G=H_{1}\left\langle
g_{1},\ldots,g_{r}\right\rangle $ by Lemma \ref{basiccomm}. Thus Theorem
\ref{KT} and Lemma \ref{basiccomm} together yield%
\begin{align*}
\lbrack H,G]  &  =%
{\displaystyle\prod\limits_{j=1}^{r}}
[H,g_{j}]\cdot H_{1}\\
&  =%
{\displaystyle\prod\limits_{j=1}^{r}}
[H,g_{j}]\cdot\left(
{\displaystyle\prod\limits_{j=1}^{r}}
[H_{1},g_{j}]\right)  ^{\ast10k(d,r)}=\left(
{\displaystyle\prod\limits_{i=1}^{r}}
[H,g_{i}]\right)  ^{\ast f_{5}}%
\end{align*}
where $f_{5}=1+10k(d,r)=O(rd).$

\subsection{Proof of the Key Theorem: reductions}

We follow the strategy of \cite{NS}, Section 4.\bigskip

\textbf{Notation} For $\mathbf{u},\mathbf{g}\in G^{(m)}$,%
\[
\mathbf{u}\cdot\mathbf{g}=(u_{1}g_{1},\ldots,u_{m}g_{m}),~~~\mathbf{c}%
(\mathbf{u},\mathbf{g})=%
{\displaystyle\prod\limits_{j=1}^{m}}
[u_{j},g_{j}].
\]

\begin{lemma}
\label{u-v}%
\[
\left(
{\displaystyle\prod\limits_{i=1}^{s}}
\mathbf{c}(\mathbf{a}(i)\cdot\mathbf{u}(i),\mathbf{g)}\right)  \left(
{\displaystyle\prod\limits_{i=1}^{s}}
\mathbf{c}(\mathbf{u}(i),\mathbf{g})\right)  ^{-1}=%
{\displaystyle\prod\limits_{i=1}^{s}}
\left(
{\displaystyle\prod\limits_{j=1}^{m}}
[a(i)_{j},g_{j}]^{\tau_{j}(\mathbf{g},\mathbf{u}(i))}\right)  ^{w(i)}%
\]
where $w(i)=\mathbf{c(u}(i-1),\mathbf{g})^{-1}\ldots\mathbf{c}(\mathbf{u}%
(1),\mathbf{g})^{-1}$.
\end{lemma}

This is a direct calculation. The next lemma is easily verified by induction
on $m$ (see \cite{NS}, Lemma 4.5):\bigskip

\begin{lemma}
\label{newgens}%
\[
\left\langle g_{j}^{\tau_{j}(\mathbf{g},\mathbf{u})}\mid j=1,\ldots
,m\right\rangle =\left\langle g_{j}^{u_{j}h_{j}}\mid j=1,\ldots
,m\right\rangle
\]
where $h_{j}=g_{j-1}^{-1}\ldots g_{1}^{-1}$.
\end{lemma}

Now let $H\vartriangleleft G=H\left\langle g_{1},\ldots,g_{r}\right\rangle $
be as in Theorem \ref{KT}. If $H=1$ there is nothing to prove, so we suppose
that $H>1$ and argue by induction on $\left\vert H\right\vert $. Since $H$ is
acceptable, we have $H=[H,G]$. Choose $N\vartriangleleft G$ with $N\leq H$
minimal subject to $1<N=[N,G]$ (in \cite{NS} such an $N$ was called a
\emph{quasi-minimal normal subgroup} of $G$). Let $Z$ be a normal subgroup of
$G$ maximal subject to $Z<N$. Then $[Z,_{n}G]=1$ for some $n$, which implies
(i) that $Z=N\cap\zeta_{\omega}(G)$ is uniquely determined, and (ii) that
$[Z,N]\leq\lbrack Z,H]\leq\lbrack Z,G_{\omega}]=1$. By definition,
$\overline{N}=N/Z$ is a chief factor of $G$; it is not bad because $H$ is acceptable.

Applying Lemma \ref{basiccomm} to $Z$ we note that $Z$ is contained in the
Frattini subgroup $\Phi(G)$ of $G$.

We fix a natural number $k$, the candidate for $k(d,r),$ and define $g_{j}$
for $j=1,\ldots,kr$ as in Theorem \ref{KT}.

Depending on the nature of $N$, we shall choose a certain normal subgroup $K$
of $G$ with $1\neq K\leq N$.

Suppose now that $h\in H$. We have to find elements $\mathbf{v}(i)\in H^{(m)}$
($i=1,\ldots,10$) such that (\ref{hprod}) and (\ref{G-gen}) hold. By inductive
hypothesis, we can do this `modulo $K$': thus there exist $\mathbf{u}(i)\in
H^{(m)}$ and $\kappa\in K$ such that%
\[
h=\kappa%
{\displaystyle\prod\limits_{i=1}^{10}}
\mathbf{c}(\mathbf{u}(i),\mathbf{g})
\]
and%
\begin{align}
G  &  =K\left\langle g_{j}^{\tau_{j}(\mathbf{g},\mathbf{u}(i))}\mid
j=1,\ldots,m\right\rangle \nonumber\\
&  =K\left\langle g_{j}^{u(i)_{j}h_{j}}\mid j=1,\ldots,m\right\rangle
\qquad\text{for }i=1,\ldots,10, \label{u-genmodK}%
\end{align}
the second equality thanks to Lemma \ref{newgens}.

The idea now is to find elements $\mathbf{a}(i)\in N^{(m)}$ such that
(\ref{hprod}) and (\ref{G-gen}) are satisfied on setting%
\[
\mathbf{v}(i)=\mathbf{a}(i)\cdot\mathbf{u}(i).
\]

Lemma \ref{u-v} shows that (\ref{hprod}) is then equivalent to%
\begin{equation}%
{\displaystyle\prod\limits_{i=1}^{10}}
\left(
{\displaystyle\prod\limits_{j=1}^{m}}
[a(i)_{j},g_{j}]^{\tau_{j}(\mathbf{g},\mathbf{u}(i))}\right)  ^{w(i)}%
=\kappa\text{.} \label{A-prod}%
\end{equation}
This can be further simplified by setting%
\begin{align}
y(i)_{j}  &  =g_{j}^{\tau_{j}(\mathbf{g},\mathbf{u}(i))w(i)},~~t(i)_{j}%
=g_{j}^{u(i)_{j}h_{j}}\nonumber\\
b(i)_{j}  &  =a(i)_{j}^{\tau_{j}(\mathbf{g},\mathbf{u}(i))w(i)},~~c(i)_{j}%
=a(i)_{j}^{u(i)_{j}h_{j}}. \label{a-b-c}%
\end{align}

Define $\phi(i):N^{(m)}\rightarrow N$ by%
\[
\mathbf{b}\phi(i)=\mathbf{c}(\mathbf{b},\mathbf{y}(i)).
\]
Then (\ref{A-prod}) becomes%
\begin{equation}%
{\displaystyle\prod\limits_{i=1}^{10}}
\mathbf{b}(i)\phi(i)=\kappa, \label{kappa}%
\end{equation}
and (\ref{u-genmodK}) is equivalent to%
\begin{align}
G  &  =K\left\langle y(i)_{1},\ldots,y(i)_{m}\right\rangle \nonumber\\
&  =K\left\langle t(i)_{1},\ldots,t(i)_{m}\right\rangle \qquad\text{for
}i=1,\ldots,10. \label{y-t-gen}%
\end{align}

Similarly, (\ref{G-gen}) holds if and only if for $i=1,\ldots,10$ we have%
\begin{equation}
G=\left\langle t(i)_{j}^{c(i)_{j}}\mid j=1,\ldots,m\right\rangle Z
\label{t-c-gen}%
\end{equation}
(where $Z$ is added harmlessly since $Z\leq\Phi(G)$). Let $\mathcal{X}(i)$
denote the set of all $\mathbf{c}(i)\in N^{(m)}$ such that (\ref{t-c-gen})
holds, and write $W(i)$ for the image of $\mathcal{X}(i)$ under the bijection
$N^{(m)}\rightarrow N^{(m)}$ defined in (\ref{a-b-c}) sending $\mathbf{c}%
(i)\longmapsto\mathbf{b}(i)$.

To sum up: to establish the existence of $\mathbf{a}(1),\ldots,\mathbf{a}%
(10)\in N^{(m)}$ such that the $\mathbf{v}(i)=\mathbf{a}(i)\cdot\mathbf{u}(i)$
satisfy (\ref{hprod}) and (\ref{G-gen}), it suffices to find $(\mathbf{b}%
(1),\ldots,\mathbf{b}(10))\in W(1)\times\cdots\times W(10)$ such that
(\ref{kappa}) holds.

$\medskip$

Set $\varepsilon=\min\{\frac{1}{1+6\delta},\frac{1}{r}\}$, and write
$^{-}:G\rightarrow G/Z$ for the quotient map. Now we separate four cases.

\subsubsection{The easy case}

If $[Z,G]>1$ we define $K=[Z,G]$. Since $[Z,H]=1$ and $G=H\left\langle
g_{1},\ldots,g_{r}\right\rangle $, we have $K=%
{\textstyle\prod\nolimits_{j=1}^{r}}
[Z,g_{j}]$. Thus $\kappa=%
{\textstyle\prod\nolimits_{j=1}^{r}}
[z_{j},g_{j}]$ with $z_{1},\ldots,z_{r}\in Z$. In this case, (\ref{kappa}) is
satisfied if we set%
\begin{align*}
b(1)_{j}  &  =z_{j}\qquad(1\leq j\leq r)\\
b(1)_{j}  &  =1\qquad(r<j\leq m)\\
b(i)_{j}  &  =1\qquad(2\leq i\leq10,~1\leq j\leq m),
\end{align*}
because $y(i)_{j}$ is conjugate to $g_{j}$ under the action of $H$ and
$[Z,H]=1$.

For each $i$ we have $W(i)\supseteq Z^{(m)}$, since in this case
(\ref{y-t-gen}) implies (\ref{t-c-gen}) if $c(i)_{j}\in Z$ for all $j$. So
$\mathbf{b}(i)\in W(i)$ for each $i$, as required.

\subsubsection{The abelian case}

If $[Z,G]=1$ and $N$ is \emph{abelian} we set $K=N$. We use additive notation
for $N$ and consider it as a $G$-module. Then (\ref{y-t-gen}) implies that%
\[
\phi(1):\mathbf{b\longmapsto}%
{\textstyle\sum\nolimits_{j=1}^{m}}
b_{j}(y(1)_{j}-1)
\]
is a surjective ($\mathbb{Z}$-module) homomorphism $N^{(m)}\rightarrow N$. It
follows that%
\[
\left\vert \phi(1)^{-1}(c)\right\vert =\left\vert \ker\phi(1)\right\vert
=\left\vert N\right\vert ^{m-1}%
\]
for each $c\in N$.

Now fix $i\in\{1,\ldots,10\}$. According to Theorem \ref{fppThm}, at least one
of the elements $g_{j}$ has the $\varepsilon/2$-fsp on $\overline{N}$;
therefore at least $k$ of the elements $\overline{t(i)_{j}}$ have this
property. Now we apply Proposition \ref{prop5.1new}(i): this shows that
(\ref{t-c-gen}) holds for at least $\left\vert \overline{N}\right\vert
^{m}(1-\left\vert \overline{N}\right\vert ^{d-k\varepsilon/2})$ values of
$\overline{\mathbf{c}(i)}$ in $\left\vert \overline{N}\right\vert ^{m}$. It
follows that%
\begin{equation}
\left\vert W(i)\right\vert =\left\vert \mathcal{X}(i)\right\vert
\geq\left\vert Z\right\vert ^{m}\cdot\left\vert \overline{N}\right\vert
^{m}(1-\left\vert \overline{N}\right\vert ^{d-k\varepsilon/2})=\left\vert
N\right\vert ^{m}(1-\left\vert \overline{N}\right\vert ^{d-k\varepsilon/2}).
\label{W(i)-sol}%
\end{equation}

We need to compare $\left\vert \overline{N}\right\vert $ with $\left\vert
N\right\vert $. Let $\{x_{1},\ldots,x_{d}\}$ be a generating set for $G$. Then
$\mathbf{b\longmapsto}%
{\textstyle\sum\nolimits_{j=1}^{d}}
b_{j}(x_{j}-1)$ induces an epimorphism from $\overline{N}^{(d)}$ onto $N$;
consequently $\left\vert N\right\vert \leq\left\vert \overline{N}\right\vert
^{d}$. Thus provided $k\varepsilon/2d>1$ we have%
\[
\left\vert W(i)\right\vert \geq\left\vert N\right\vert ^{m}(1-\left\vert
N\right\vert ^{1-k\varepsilon/2d}).
\]

Assume now that $k\varepsilon>4d$. Then $W(i)$ is non-empty for each $i$. For
$i=2,\ldots,10$ choose $\mathbf{b}(i)\in W(i)$ and put%
\[
c=\kappa\left(
{\displaystyle\prod\limits_{i=2}^{10}}
\mathbf{b}(i)\phi(i)\right)  ^{-1}.
\]
Then%
\[
\left\vert \phi(1)^{-1}(c)\right\vert +\left\vert W(1)\right\vert
\geq\left\vert N\right\vert ^{m}(\left\vert N\right\vert ^{-1}+1-\left\vert
N\right\vert ^{1-k\varepsilon/2d})>\left\vert N\right\vert ^{m}.
\]
It follows that $\phi(1)^{-1}(c)\cap W(1)$ is non-empty. Thus we may choose
$\mathbf{b}(1)\in\phi(1)^{-1}(c)\cap W(1)$ and ensure that (\ref{kappa}) is satisfied.

\subsubsection{The soluble case}

Suppose next that $[Z,G]=1$ and $N>N^{\prime}>1$. In this case we take
$K=N^{\prime}$. Since $N^{\prime}\leq Z$, the argument above again gives
(\ref{W(i)-sol}).

The maps $\phi(i)$ are no longer homomorphisms, however, and it is quite a
major undertaking to obtain a good estimate for the fibres. The outcome is
Proposition 7.1 of \cite{NS}; translated into the present notation it is

\begin{proposition}
\label{NS6.2}Assume that $G=Z\left\langle y(i)_{1},\ldots,y(i)_{m}%
\right\rangle $ for $i=1,2,3$. Then for each $c\in N^{\prime}$ there exist
$c_{1},c_{2},c_{3}\in N$ such that $c=c_{1}c_{2}c_{3}$ and%
\begin{equation}
\left\vert \phi(i)^{-1}(c_{i})\right\vert \geq\left\vert N\right\vert
^{m}\cdot\left\vert \overline{N}\right\vert ^{-d-1}\qquad(i=1,2,3).
\label{fibres-sol}%
\end{equation}

\end{proposition}

The initial hypothesis follows from (\ref{y-t-gen}) since now $K\leq Z$.

Assume now that $k\varepsilon>4d+2$. Then (\ref{W(i)-sol}) and
(\ref{fibres-sol}) together imply that $\phi(i)^{-1}(c_{i})\cap W(i)$ is
non-empty for $i=1,2,3$, while (\ref{W(i)-sol}) implies that $W(i)$ is
non-empty for every $i$.

Choose $\mathbf{b}(i)\in W(i)$ for $i=4,\ldots,10$. Put%
\[
c=\kappa\left(
{\displaystyle\prod\limits_{i=4}^{10}}
\mathbf{b}(i)\phi(i)\right)  ^{-1},
\]
and choose $c_{1},c_{2},c_{3}$ as in Proposition \ref{NS6.2}. Then for
$i=1,2,3$ we can find $\mathbf{b}(i)\in\phi(i)^{-1}(c_{i})\cap W(i)$, and so
ensure that (\ref{kappa}) is satisfied.

\subsubsection{The semisimple case}

If $[Z,G]=1$ and $N=N^{\prime}$, define $K=N$. Now $\overline{N}$ is
semisimple with at least $3$ simple factors, and $N$ is quasi-semisimple. In
this case, Theorem \ref{fppThm} shows that at least one the elements $g_{j}$
has the $\varepsilon$-fgp on $\overline{N}$; therefore for each $i$, at least
$k$ of the elements $\overline{t(i)_{j}}$ and at least $k$ of the elements
$y(i)_{j}$ have this property. Proposition \ref{prop5.1new}(ii) now shows
that
\[
\left\vert \overline{\mathcal{X}(i)}\right\vert \geq\left\vert \overline
{N}\right\vert ^{m}(1-2^{2-k\varepsilon}),
\]
provided we assume that $k\varepsilon\geq\max\{2d+4,C\}$ for a certain
absolute constant $C$. This implies%
\[
\left\vert W(i)\right\vert =\left\vert \mathcal{X}(i)\right\vert
\geq\left\vert N\right\vert ^{m}(1-2^{2-k\varepsilon}).
\]

Now Theorem \ref{comm20}, proved below in Subsection \ref{ssns}, gives the
following: there are absolute constants $D,$ $\varepsilon_{0}$ such that if
for each $i=1,\ldots,10$

\begin{description}
\item[(a)] the group $\left\langle y(i)_{1},\ldots,y(i)_{m}\right\rangle $
permutes the quasisimple factors of $N$ transitively,

\item[(b)] at least $k$ of the $y(i)_{j}$ have the $\varepsilon$-fgp on
$\overline{N}$, where $k\varepsilon\geq4+2D$,

\item[(c)] the subset $W(i)\subseteq N^{(m)}$ satisfies $\left\vert
W(i)\right\vert \geq(1-\varepsilon_{0}/6)\left\vert N\right\vert ^{m}$,
\end{description}

\noindent then%
\[%
{\displaystyle\prod\limits_{i=1}^{10}}
W(i)\phi(i)=N.
\]
Condition (a) follows from (\ref{y-t-gen}). Thus we can find $\mathbf{b}(i)\in
W(i)$ ($i=1,\ldots,10$) such that (\ref{kappa}) is satisfied provided we
assume that%
\begin{align*}
k\varepsilon &  >\max\{2d+4,C,4+2D,2+\log_{2}(6/\varepsilon_{0})\}\\
&  =\max\{2d+4,C^{\ast}\}
\end{align*}
where $C^{\ast}$ is an absolute constant.

\subsubsection{Conclusion of the proof}

Recall that we defined $\varepsilon=\min\{\frac{1}{1+6\delta},\frac{1}{r}\}$.
So if we now define%
\[
k(d,r)=1+\max\{r,1+6\delta\}\cdot\max\{4d+4,\left\lceil C^{\ast}\right\rceil
\},
\]
then $k=k(d,r)$ fulfils the requirements of all the preceding steps. This
concludes the proof of Theorem \ref{KT} modulo Proposition \ref{prop5.1new},
Theorem \ref{comm20}, and \cite{NS}, Proposition 6.2.

\section{Semisimple groups\label{semisimplesec}}

This section is devoted to the proof of Theorem \ref{comm20}. This will be
stated in Subsection \ref{ssns}. Like Proposition 9.2 of \cite{NS}, which it
in effect generalizes, its proof has two components: (1) a result about
products of commutators in quasisimple groups, and (2) a complicated
combinatorial reduction argument. These will occupy the next two subsections.

As remarked in the Introduction, the proof of (1) given here is significantly
simpler (and shorter) than \cite{NS2}, which played the analogous role in our
earlier work. The reduction argument (2) is essentially the same as in
\cite{NS}, though we are now using it to prove something different
(specifically, we have to control the image of a certain mapping rather than
its fibres). We have re-cast the argument from scratch, in an attempt to make
it more transparent (the reader will judge whether we have succeeded!)
However, we shall quote one combinatorial result from Section 8 of \cite{NS}.

\subsection{Twisted commutators in quasisimple groups\label{quasi}}

For automorphisms $\alpha,\beta$ of a group $S$ and $x,y\in S$ we write%
\[
T_{\alpha,\beta}(x,y)=x^{-1}y^{-1}x^{\alpha}y^{\beta}.
\]
For $\alpha=(\alpha_{1},\ldots,\alpha_{D})$ and $\beta=(\beta_{1},\ldots
,\beta_{D})$ in $\mathrm{Aut}(S)^{(D)}$ the mapping $\mathbf{T}_{\alpha,\beta
}:S^{(D)}\times S^{(D)}\rightarrow S$ is defined by%
\[
\mathbf{T}_{\alpha,\beta}(\mathbf{x},\mathbf{y})=%
{\displaystyle\prod\limits_{i=1}^{D}}
T_{\alpha_{i},\beta_{i}}(x_{i},y_{i}).
\]

\begin{theorem}
\label{T1}There exist $\varepsilon>0$ and $D\in\mathbb{N}$ such that if $S$ is
a finite quasisimple group, $\alpha,\beta\in\mathrm{Aut}(S)^{(D)}$, and
$X\subseteq S^{(2D)}$ has size at least $(1-\varepsilon)\left\vert
S^{(2D)}\right\vert $, then $\left\vert \mathbf{T}_{\alpha,\beta
}(X)\right\vert \geq\lambda\left\vert S\right\vert $, where%
\begin{equation}
\lambda=\left\{
\begin{array}
[c]{ccc}%
l(S)^{-3/5} &  & \text{if }l(S)\geq3\\
&  & \\
1 &  & \text{if }l(S)=2
\end{array}
\right.  . \label{lambda}%
\end{equation}

\end{theorem}

The following corollary is Theorem 1.1 of \cite{NS2}:

\begin{corollary}
\label{NS2Thm1.1}There exists $D_{1}\in\mathbb{N}$ such that if $S$ is a
finite quasisimple group and $\alpha,\beta\in\mathrm{Aut}(S)^{(D_{1})}$ then%
\[%
{\displaystyle\prod\limits_{i=1}^{D_{1}}}
T_{\alpha_{i},\beta_{i}}(S,S)=S.
\]

\end{corollary}

\begin{proof}
Set $D_{1}=5D$, and divide $\alpha$ and $\beta$ into $5$ $D$-tuples
$\alpha(j),~\beta(j)$. Taking $X=S^{(2D)}$ in the theorem gives $\left\vert
T_{\alpha(j),\beta(j)}(X)\right\vert \geq l(S)^{-3/5}\left\vert S\right\vert $
for $j=1,\ldots,5$. The result now follows by the `Gowers trick', since
$5\times\frac{3}{5}=5-2$.
\end{proof}

\subsubsection{Reductions for Theorem \ref{T1}\label{P-reductions}}

In this subsection, we fix a finite group $S$ and $\lambda\in(0,1]$. For any
$\alpha,\beta\in\mathrm{Aut}(S)^{(D)}$ we consider the statement

\begin{description}
\item[$\mathcal{P}(\alpha,\beta;D,\varepsilon):$] \emph{For }$X\subseteq
S^{(2D)}$,%
\[
\left\vert X\right\vert \geq(1-\varepsilon)\,|S^{(2D)}|\,\Longrightarrow
\left\vert \mathbf{T}_{\alpha,\beta}(X)\right\vert \geq\lambda\left\vert
S\right\vert .
\]

\end{description}

\noindent If $\Gamma$ is a subgroup of $\mathrm{Aut}(S)$, we write%
\[
\mathcal{P}(\Gamma;D,\varepsilon)\Leftrightarrow\mathcal{P}(\alpha
,\beta;D,\varepsilon)~\forall\alpha,\beta\in\Gamma^{(D)}.
\]

Thus Theorem \ref{T1} asserts the existence of $D$ and $\varepsilon$ such that
$\mathcal{P}(\mathrm{Aut}(S);D,\varepsilon)$ holds with $\lambda$ defined by
(\ref{lambda}) for every quasisimple group $S$.

Our aim in the rest of this subsection is to establish the reduction steps
Propositions \ref{obviousredn}, \ref{firstredn} and \ref{DtoCredn}.

\begin{proposition}
\label{obviousredn}If $D_{1}\leq D$ and $\varepsilon_{1}\geq\varepsilon$ then
$\mathcal{P}(\Gamma;D_{1},\varepsilon_{1})$ implies $\mathcal{P}%
(\Gamma;D,\varepsilon)$.
\end{proposition}

\begin{proof}
If $D_{1}=D$ the claim is obvious. Suppose that $D>D_{1}$. We write%
\[
\mathbf{T}_{\alpha,\beta}(\mathbf{x},\mathbf{y})=\mathbf{T}_{\alpha^{\prime
},\beta^{\prime}}(\mathbf{x}^{\prime},\mathbf{y}^{\prime})\mathbf{T}%
_{\alpha^{\prime\prime},\beta^{\prime\prime}}(\mathbf{x}^{\prime\prime
},\mathbf{y}^{\prime\prime})
\]
where $\mathbf{x}^{\prime}=(x_{1},\ldots,x_{D_{1}}),~\mathbf{x}^{\prime\prime
}=(x_{D_{1}+1},\ldots,x_{D})$ etc. Now if $X\subseteq S^{(2D)}$ satisfies
$\left\vert X\right\vert \geq(1-\varepsilon)\left\vert S^{(2D)}\right\vert $
then there exist $(\mathbf{x}^{\prime\prime},\mathbf{y}^{\prime\prime})\in
S^{2(D-D_{1})}$ and $X_{1}\subseteq S^{(2D_{1})}$ such that $X_{1}%
\times\{(\mathbf{x}^{\prime\prime},\mathbf{y}^{\prime\prime})\}\subseteq X$
and $\left\vert X_{1}\right\vert \geq(1-\varepsilon)\left\vert S^{(2D_{1}%
)}\right\vert $. Then%
\[
T_{\alpha,\beta}(X)\supseteq T_{\alpha^{\prime},\beta^{\prime}}(X_{1}%
)\cdot\mathbf{T}_{\alpha^{\prime\prime},\beta^{\prime\prime}}(\mathbf{x}%
^{\prime\prime},\mathbf{y}^{\prime\prime}),
\]
a set of size at least $\lambda\left\vert S\right\vert $ since $1-\varepsilon
\geq1-\varepsilon_{1}$.
\end{proof}

\begin{proposition}
\label{firstredn}If $\Delta\vartriangleleft\Gamma$ and $\left\vert
\Gamma:\Delta\right\vert \leq n$ then $\mathcal{P}(\Delta;D,\varepsilon)$
implies $\mathcal{P}(\Gamma;n^{2}D,\varepsilon)$.
\end{proposition}

This is a little more complicated. It will follow from

\begin{proposition}
\label{bigred}Let $\Delta$ be a normal subgroup of index $n$ in $\Gamma$, and
let $\alpha,~\beta\in\Gamma^{(n^{2}D)}$. Then there exist $\overline{\alpha
},~\overline{\beta}\in\Delta^{(D)}$ and a bijection $\pi:S^{(n^{2}%
D)}\rightarrow S^{((2n^{2}-2)D)}\times S^{(2D)}$ such that, for each
$\mathbf{x\in}S^{(n^{2}D)}$,%
\[%
{\displaystyle\prod\limits_{i=1}^{n^{2}D}}
T_{\alpha_{i},\beta_{i}}(x_{2i-1},x_{2i})=%
{\displaystyle\prod\limits_{i=1}^{D}}
T_{\overline{\alpha}_{i},\overline{\beta}_{i}}(\widehat{x}_{2i-1},\widehat
{x}_{2i})\cdot R(\widetilde{x}),
\]
where $(\widetilde{x},\widehat{x})=\mathbf{x}\pi$ and $R(\widetilde{x})$
depends only on $\widetilde{x}$.
\end{proposition}

Accepting this for now we deduce Proposition \ref{firstredn}. Let
$\varepsilon\in(0,1)$ and suppose that $W\subseteq S^{(2n^{2}D)}$ satisfies
$\left\vert W\right\vert \geq(1-\varepsilon)|S^{(2n^{2}D)}|$. Then $\left\vert
W\pi\right\vert =\left\vert W\right\vert $; so for at least one element $u\in
S^{((2n^{2}-2)D)}$ the set%
\[
Y_{u}:=\left\{  y\in S^{(2D)}\mid(u,y)\in W\pi\right\}
\]
satisfies $\left\vert Y_{u}\right\vert \geq(1-\varepsilon)\left\vert
S^{(2D)}\right\vert $. Then%
\[
\mathbf{T}_{\alpha,\beta}(W)\supseteq\mathbf{T}_{\overline{\alpha}%
,\overline{\beta}}(Y_{u})\cdot R(u).
\]
If $\mathcal{P}(\Delta;D,\varepsilon)$ holds then $|\mathbf{T}_{\overline
{\alpha},\overline{\beta}}(Y_{u})|\,\geq\lambda\left\vert S\right\vert $, and
so $\left\vert \mathbf{T}_{\alpha,\beta}(W)\right\vert \geq\lambda\left\vert
S\right\vert $. Thus $\mathcal{P}(\Gamma;n^{2}D,\varepsilon)$ holds as claimed.

\bigskip

Now we embark on the proof of Proposition \ref{bigred}.

\begin{lemma}
\label{T...T}%
\begin{align*}
T_{a_{1},\beta_{1}}(x_{1},x_{2})T_{a_{2},\beta_{2}}(x_{3},x_{4})  &
=A(x\nu)\cdot T_{a_{1}\alpha_{2},\beta_{2}}(x\mu)\cdot B(x\nu)\\
&  =C(x\tau)\cdot T_{a_{1},\beta_{1}\beta_{2}}(x\sigma)
\end{align*}
where $x\mapsto(x\mu,x\nu)$ and $x\mapsto(x\sigma,x\tau)$ are bijections from
$S^{(4)}$ to $S^{(2)}\times S^{(2)}$.
\end{lemma}

\begin{proof}
Take $z=x_{3}(x_{1}^{\alpha_{1}}x_{2}^{\beta_{1}})^{-1},~w=z^{-\alpha_{2}%
}x_{4}zx_{2}$ and $t=x_{1}x_{2}^{\beta_{1}\alpha_{1}^{-1}}z^{\alpha_{2}%
\beta_{2}\alpha_{2}^{-1}\alpha_{1}^{-1}}$ and set%
\begin{align*}
x\mu &  =(t,w),~~x\nu=(x_{2},z)\\
A(y,z)  &  =(y^{\beta_{1}}z^{\alpha_{2}\beta_{2}\alpha_{2}^{-1}})^{\alpha
_{1}^{-1}},~B(y,z)=(yz)^{-\beta_{2}}.
\end{align*}
Take $u=x_{2}^{\beta_{1}}x_{3}^{-1}x_{4}^{-1},~=x_{1}u^{\alpha_{1}^{-1}}%
x_{3}^{\alpha_{2}\alpha_{1}^{-1}}u^{-\beta_{2}\alpha_{1}^{-1}}$ and set%
\begin{align*}
x\sigma &  =(v,x_{2}),~~x\tau=(u,x_{3})\\
C(u,y)  &  =(y^{\alpha_{2}}u^{-\beta_{2}})^{\alpha_{1}^{-1}.}%
\end{align*}

\end{proof}

\begin{lemma}
\label{T-conj}%
\[
zT_{\alpha,\beta}(x,y)=T_{\alpha,\beta}(x^{\prime},y^{\prime})z^{-\gamma}%
\]
where $x^{\prime}=z^{\alpha\beta^{-1}\alpha^{-1}}xz^{-1}$, $y^{\prime
}=z^{\alpha\beta^{-1}}yz^{-\alpha\beta^{-1}\alpha^{-1}}$ and $\gamma
=[\alpha^{-1},\beta]$.
\end{lemma}

\begin{lemma}
\label{four}Suppose that $\left\vert \Gamma:\Delta\right\vert =2$. Given
$\alpha_{i},~\beta_{i}\in\Gamma$ ($i=1,\ldots,4$), there exist $\gamma
,~\delta\in\Delta$, a bijection $x\mapsto(x^{\ast},\widetilde{x})$ from
$S^{(8)}$ to $S^{(2)}\times S^{(6)}$ and maps $P,~Q:S^{(6)}\rightarrow S$ such
that%
\[%
{\displaystyle\prod\limits_{i=1}^{4}}
T_{\alpha_{i},\beta_{i}}(x_{2i-1},x_{2i})=P(\widetilde{x})T_{\gamma,\delta
}(x^{\ast})Q(\widetilde{x}).
\]

\end{lemma}

\begin{proof}
Define%
\[
(\widehat{x}_{1},\widehat{x}_{2},\widetilde{x}_{1},\widetilde{x}_{2}%
;\gamma_{1},\eta_{1})=\left\{
\begin{array}
[c]{cccc}%
(x_{1},x_{2},x_{3},x_{4};\alpha_{1},\beta_{1}) &  & \text{if} & \alpha_{1}%
\in\Delta\\
&  &  & \\
(x_{3},x_{4},x_{1},x_{2};\alpha_{2},\beta_{2}) &  & \text{if} & \alpha_{2}%
\in\Delta\\
&  &  & \\
(x\mu,x\nu;\alpha_{1}\alpha_{2},\beta_{2}) &  & \text{if} & \alpha_{1}%
\alpha_{2}\in\Delta
\end{array}
\right.
\]
(assuming in the 2nd and 3d lines that $\alpha_{1}\notin\Delta$). Then (using
Lemma \ref{T...T} in the 3d case) we see that%
\[
T_{a_{1},\beta_{1}}(x_{1},x_{2})T_{a_{2},\beta_{2}}(x_{3},x_{4})=P_{1}%
(\widetilde{x}_{1},\widetilde{x}_{2})T_{\gamma_{1},\eta_{1}}(\widehat{x}%
_{1},\widehat{x}_{2})Q_{1}(\widetilde{x}_{1},\widetilde{x}_{2})
\]
for suitable maps $P_{1},~Q_{1}$. Note that $\gamma_{1}\in\Delta$ and
$(x_{1},x_{2},x_{3},x_{4})\mapsto(\widehat{x}_{1},\widehat{x}_{2}%
,\widetilde{x}_{1},\widetilde{x}_{2})$ is bijective. Similarly%
\[
T_{a_{3},\beta_{3}}(x_{5},x_{6})T_{a_{4},\beta_{4}}(x_{7},x_{8})=P_{2}%
(\widetilde{x}_{3},\widetilde{x}_{4})T_{\gamma_{2},\eta_{2}}(\widehat{x}%
_{3},\widehat{x}_{4})Q_{2}(\widetilde{x}_{3},\widetilde{x}_{4})
\]
where $\gamma_{2}\in\Delta$ and $(x_{3},x_{4},x_{5},x_{6})\mapsto(\widehat
{x}_{3},\widehat{x}_{4},\widetilde{x}_{3},\widetilde{x}_{4})$ is bijective.

Put $z=Q_{1}(\widetilde{x}_{1},\widetilde{x}_{2})P_{2}(\widetilde{x}%
_{3},\widetilde{x}_{4})$ and set%
\[
\overline{x}_{3}=z^{\gamma_{2}\eta_{2}^{-1}\gamma_{2}^{-1}}\widehat{x}%
_{3}z^{-1},~\overline{x}_{4}=z^{\gamma_{2}\eta_{2}^{-1}}\widehat{x}%
_{3}z^{-\gamma_{2}\eta_{2}^{-1}\gamma_{2}^{-1}}.
\]
Lemma \ref{T-conj} gives%
\[
T_{\gamma_{1},\eta_{1}}(\widehat{x}_{1},\widehat{x}_{2})\cdot z\cdot
T_{\gamma_{2},\eta_{2}}(\widehat{x}_{3},\widehat{x}_{4})=T_{\gamma_{1}%
,\eta_{1}}(\widehat{x}_{1},\widehat{x}_{2})T_{\gamma_{2},\eta_{2}}%
(\overline{x}_{3},\overline{x}_{4})R(\widetilde{x})
\]
where $\widetilde{x}=(\widetilde{x}_{1},\widetilde{x}_{2},\widetilde{x}%
_{3},\widetilde{x}_{4})$. Now we repeat the first procedure, applied to the
second pair of automorphisms $\eta_{1},\eta_{2}$. This gives $\delta\in
\{\eta_{1},\eta_{2},\eta_{1}\eta_{2}\}\cap\Delta$, $\gamma\in\{\gamma
_{1},\gamma_{2}\}$ and a bijection $(\widehat{x}_{1},\widehat{x}_{2}%
,\overline{x}_{3},\overline{x}_{4})\mapsto(x_{1}^{\ast},x_{2}^{\ast
},\widetilde{x}_{5},\widetilde{x}_{6})$ such that%
\[
T_{\gamma_{1},\eta_{1}}(\widehat{x}_{1},\widehat{x}_{2})T_{\gamma_{2},\eta
_{2}}(\overline{x}_{3},\overline{x}_{4})=P_{3}(\widetilde{x}_{5},\widetilde
{x}_{6})T_{\gamma,\delta}(x_{1}^{\ast},x_{2}^{\ast})Q_{3}(\widetilde{x}%
_{5},\widetilde{x}_{6}).
\]
Then%
\[%
{\displaystyle\prod\limits_{i=1}^{4}}
T_{\alpha_{i},\beta_{i}}(x_{2i-1},x_{2i})=PT_{\gamma,\delta}(x_{1}^{\ast
},x_{2}^{\ast})Q
\]
where $P=P_{1}(\widetilde{x}_{1},\widetilde{x}_{2})P_{3}(\widetilde{x}%
_{5},\widetilde{x}_{6})$ and $Q=Q_{3}(\widetilde{x}_{5},\widetilde{x}%
_{6})R(\widetilde{x}_{1},\widetilde{x}_{2},\widetilde{x}_{3},\widetilde{x}%
_{4})$. The result follows.\bigskip
\end{proof}

\textbf{Proof of Proposition \ref{bigred}.} Suppose first that $n=2$\textbf{.
}Write $T_{i}=T_{\alpha_{i},\beta_{i}}(x_{2i-1},x_{2i})$. Grouping these four
at a time and applying the preceding lemma we see that%
\[%
{\displaystyle\prod\limits_{i=1}^{4D}}
T_{i}=%
{\displaystyle\prod\limits_{j=1}^{D}}
P_{j}(y_{j})T_{\gamma_{j},\delta_{j}}(u_{j})Q_{j}(y_{j})
\]
where $\gamma_{j},~\delta_{j}\in D$, $y_{j}\in S^{(6)}$, $u_{j}\in S^{(2)}$
and $(x_{1},\ldots,x_{4D})\mapsto(y_{1},\ldots,y_{D};u_{1},\ldots,u_{D})$ is a
bijection. Using Lemma \ref{T-conj} we now conjugate the factors
$T_{\gamma_{j},\delta_{j}}(u_{j})$ by $z_{j}=(P_{1}^{\lambda_{1j}}Q_{1}%
^{\mu_{1j}}\ldots Q_{j-1}^{\mu_{j-1,j}}P_{j}^{\lambda_{jj}})^{-1}$, for
suitable automorphisms $\lambda_{ij},\mu_{ij}\in\Delta$, to obtain%
\[%
{\displaystyle\prod\limits_{i=1}^{4D}}
T_{i}=%
{\displaystyle\prod\limits_{j=1}^{D}}
T_{\gamma_{j},\delta_{j}}(\widehat{x}_{2j-1},\widehat{x}_{2j})\cdot
R(\widetilde{x})
\]
where $\widetilde{x}=(y_{1},\ldots,y_{D}),$ $R=%
{\textstyle\prod\nolimits_{j=1}^{D}}
P_{j}(y_{j})^{\lambda_{j}}Q_{j}(y_{j})^{\mu_{j}}$ for certain automorphisms
$\lambda_{j},\mu_{j}\in\Delta$, and $\widehat{x}_{2j-1},\widehat{x}_{2j}$ are
obtained from $y_{j}$ by multiplying on the left and right by expressions
depending only on $z_{j}=z_{j}(y_{1},\ldots,y_{D})$. The result follows.

Now we consider the general case where $\left\vert \Gamma:\Delta\right\vert
=n>2$\textbf{.} This follows the same pattern. Suppose first that $D=1$. There
exist $i,j$ with $1\leq i\leq j\leq n$ such that $\gamma=\alpha_{i}%
\alpha_{i+1}\ldots\alpha_{j}\in\Delta$. Using Lemmas \ref{T...T} and
\ref{T-conj} repeatedly we get%
\[%
{\displaystyle\prod\limits_{i=1}^{n}}
T_{i}=P(\widetilde{x})T_{\gamma,\beta_{j}}(x^{\ast})Q(\widetilde{x})
\]
where $x\mapsto(x^{\ast},\widetilde{x})$ is a bijection $S^{(2n)}\rightarrow
S^{(2)}\times S^{(2n-2)}$. Grouping the factors together $n$ at a time and
applying this to each group of $n$ factors we get%
\begin{align*}%
{\displaystyle\prod\limits_{i=1}^{n^{2}}}
T_{i}  &  =%
{\displaystyle\prod\limits_{i=1}^{n}}
P_{i}(\widetilde{x}_{i})T_{\gamma_{i},\beta_{j(i)}}(x_{i}^{\ast}%
)Q_{i}(\widetilde{x}_{i})\\
&  =%
{\displaystyle\prod\limits_{i=1}^{n}}
T_{\gamma_{i},\beta_{j(i)}}(\overline{x}_{i})\cdot R_{1}(x^{\dag}),
\end{align*}
using Lemma \ref{T-conj} for the second step; here $x\mapsto(\overline
{x},x^{\dag})$ is a bijection $S^{(2n^{2})}\rightarrow S^{(2n)}\times
S^{(2n^{2}-2n)}$ and each $\gamma_{i}\in\Delta$.

There exist $k,l$ with $1\leq k\leq l\leq n$ such that $\beta_{j(k)}%
\ldots\beta_{j(l)}=\delta\in\Delta$, and repeating the procedure we get%
\[%
{\displaystyle\prod\limits_{i=1}^{n}}
T_{\gamma_{i},\beta_{j(i)}}(\overline{x}_{i})=T_{\gamma_{k},\delta}(x^{\ddag
})\cdot R_{2}(x^{\sharp})
\]
where $\overline{x}\mapsto(x^{\ddag},x^{\sharp})$ is a bijection
$S^{(2n)}\rightarrow S^{(2)}\times S^{(2n-2)}$. So putting $\gamma=\gamma_{k}$
we have%
\begin{equation}%
{\displaystyle\prod\limits_{i=1}^{n^{2}}}
T_{i}=T_{\gamma,\delta}(x^{\ddag})\cdot R(x^{\dag},x^{\sharp})
\label{n-sqaured}%
\end{equation}
with $\gamma,\delta\in\Delta$ and $x\mapsto(x^{\ddag},x^{\dag},x^{\sharp})$ a
bijection $S^{(n^{2})}\rightarrow S^{2}\times S^{(2n^{2}-2n)}\times
S^{(2n-2)}$.

In the general case where $D>1$ we group the $n^{2}D$ factors $T_{i}$ together
$n^{2}$ at a time, apply (\ref{n-sqaured}) to each product of $n^{2}$ factors,
and then conjugate the resulting terms $T_{\gamma(i),\delta(i)}(x_{i}^{\ddag
})$ by the intervening factors $R$ using Lemma \ref{T-conj} to obtain%
\[%
{\displaystyle\prod\limits_{i=1}^{n^{2}D}}
T_{i}=%
{\displaystyle\prod\limits_{i=1}^{D}}
T_{\gamma(i),\delta(i)}(\widehat{x})\cdot R(\widetilde{x})
\]
for a certain bijection $x\mapsto(\widetilde{x},\widehat{x}):S^{(n^{2}%
D)}\rightarrow S^{((2n^{2}-2)D)}\times S^{(2D)}$.

This completes the proof.

\bigskip

The final reduction step needs the next three lemmas.

\begin{lemma}
\label{sdp}Let $\varepsilon\in(0,1)$. If $Z\subseteq X\times Y$ satisfies
$\left\vert Z\right\vert \geq(1-\varepsilon^{2})\left\vert X\times
Y\right\vert $ then for at least $(1-\varepsilon)\left\vert X\right\vert $
elements $u\in X$ we have $\left\vert Z\cap(\{u\}\times Y)\right\vert
\geq(1-\varepsilon)\left\vert Y\right\vert $.
\end{lemma}

\begin{proof}
Suppose the number of such elements $u$ is $\rho\left\vert X\right\vert $.
Then%
\[
(1-\varepsilon^{2})\left\vert X\times Y\right\vert \leq(1-\rho)\left\vert
X\right\vert \cdot(1-\varepsilon)\left\vert Y\right\vert +\rho\left\vert
X\right\vert \cdot\left\vert Y\right\vert
\]
whence $\rho\geq1-\varepsilon$.
\end{proof}

\begin{lemma}
\label{T-C}%
\begin{align*}
T_{\alpha,\beta}(x,y)  &  =[x,\alpha y][y,\beta]\\
z[x,\gamma]  &  =[x^{\prime},\gamma]z^{\gamma}%
\end{align*}
where $x^{\prime}=xz^{-1}$.
\end{lemma}

Recall that for $D$-tuples $x,~\beta,$ we use the notation $\mathbf{c}%
(x,\beta)=%
{\textstyle\prod\nolimits_{i=1}^{D}}
[x_{i},\beta_{i}]$, and $x\cdot\beta=(x_{1}\beta_{1},\ldots,x_{D}\beta_{D})$.

\begin{lemma}
\label{big T-C}There is a bijection $y\mapsto\overline{y}:S^{(D)}\rightarrow
S^{(D)}$, and for each fixed $y\in S^{(D)}$ a bijection $x\mapsto x^{\prime
}:S^{(D)}\rightarrow S^{(D)}$ (depending on $y$), such that%
\[
\mathbf{T}_{\alpha,\beta}(x,y)=\mathbf{c}(x^{\prime},\overline{y}\cdot
\alpha)\cdot h(y),
\]
where $h(y)$ depends only on $y$.
\end{lemma}

\begin{proof}
Using Lemma \ref{T-C} we get%
\begin{align*}
\mathbf{T}_{\alpha,\beta}(x,y)  &  =%
{\displaystyle\prod\limits_{i=1}^{D}}
[x_{i},\alpha_{i}y_{i}][y_{i},\beta_{i}]\\
&  =%
{\displaystyle\prod\limits_{i=1}^{D}}
[x_{i}^{\prime},\alpha_{i}y_{i}]\cdot z_{D}%
\end{align*}
where $x_{i}^{\prime}=x_{i}z_{i}^{-1}$ and $z_{1}=1,$ $z_{i}=(z_{i-1}%
[y_{i-1},\beta_{i-1}])^{\alpha_{i}y_{i}}$ for $1<i\leq D$. The result follows
on setting $\overline{y}_{i}=y_{i}^{\alpha_{i}^{-1}}.$
\end{proof}

\begin{proposition}
\label{DtoCredn}Let $\alpha,\beta\in\mathrm{Aut}(S)$. Suppose that for each
$Y\subseteq S^{(D)}$ with $\left\vert Y\right\vert \geq(1-\varepsilon
)|S^{(D)}|$ there exists $\mathbf{y}\in Y$ such that%
\begin{equation}
X\subseteq S^{(D)},\left\vert X\right\vert \geq(1-\varepsilon)\,|S^{(D)}%
|\,\Longrightarrow\left\vert \mathbf{c}(X,\mathbf{y}\cdot\alpha)\right\vert
\geq\lambda|S|. \label{Xyformula}%
\end{equation}
Then $\mathcal{P}(\alpha,\beta;D,\varepsilon^{2})$ holds.
\end{proposition}

Here $\mathbf{c}(X,\mathbf{y}\cdot\alpha)=\left\{  \mathbf{c}(\mathbf{x}%
,\mathbf{y}\cdot\alpha)\mid\mathbf{x}\in X\right\}  $.

$\medskip$

\begin{proof}
Suppose that $W\subseteq S^{(2D)}$ satisfies $\left\vert W\right\vert
\geq(1-\varepsilon^{2})\left\vert S^{(2D)}\right\vert $. Let%
\[
Y=\left\{  \mathbf{y}\in S^{(D)}\mid~~|W\cap(S^{(D)}\times\{\mathbf{y}%
\})|~~\geq(1-\varepsilon)|S^{(D)}|\right\}  .
\]
Lemma \ref{sdp} shows that $\left\vert Y\right\vert \geq(1-\varepsilon
)|S^{(D)}|$, so we can choose $\mathbf{y}\in Y$ so that (\ref{Xyformula})
holds. There exists $X\subseteq S^{(D)}$ with $\left\vert X\right\vert
\geq(1-\varepsilon)|S^{(D)}|$ and $X\times\{\mathbf{y}\}\subseteq W$. Let
$x\mapsto x^{\prime}$ be the bijection $S^{(D)}\rightarrow S^{(D)}$ given in
Lemma \ref{big T-C}. Then%
\[
\mathbf{T}_{\alpha,\beta}(W)\supseteq\mathbf{T}_{\alpha,\beta}(X\times
\{\mathbf{y}\})=\mathbf{c}(X^{\prime},\overline{\mathbf{y}}\cdot\alpha)\cdot
h(\mathbf{y}),
\]
a set of size at least $\lambda\left\vert S\right\vert $.
\end{proof}

\subsubsection{Small groups}

Let $N^{\ast}$ be an upper bound for the orders of quasisimple groups $S$ such
that $l(S)=2$; that $N^{\ast}$ is finite follows from Proposition
\ref{l(simple)} and well-known facts about the alternating groups. We fix a
natural number $N_{0}\geq N^{\ast}$, to be specified later, and denote by
$\overline{\mathcal{S}}$ the class of all quasisimple groups of order less
than $N_{0}$. Set%
\[
N_{1}=\max_{S\in\overline{\mathcal{S}}}~\left\vert \mathrm{Out}(S)\right\vert
.
\]
There is a natural number $\delta_{1}$ such that for each $S\in\overline
{\mathcal{S}}$, every element of $S$ is a product of $\delta_{1}$ commutators
(obviously $\delta_{1}\leq\delta^{\ast}$, given in Proposition \ref{commwidth}%
; in fact we can take $\delta_{1}\leq2$).

Define%
\begin{align*}
\gamma &  :G\times G\rightarrow G\\
\gamma(x,y)  &  =[x,y].
\end{align*}

\begin{lemma}
\label{inn}Let $\alpha,\beta\in\mathrm{Inn}(S)$. Then there exist a bijection
$(x,y)\longmapsto(\overline{x},\overline{y})$ from $S^{(2)}$ to $S^{(2)}$ and
an element $t\in S$ such that%
\[
T_{\alpha,\beta}(x,y)=[\overline{x},\overline{y}]t
\]
for all $x,y\in S$.
\end{lemma}

\begin{proof}
For simplicity, let $\alpha$ and $\beta$ denote also elements of $S$ inducing
the given inner automorphisms. Now define%
\begin{align*}
t  &  =[\alpha^{-1},\beta],\\
(\overline{x},\overline{y})  &  =(t\beta^{-1}x,t\beta^{-1}\alpha y\beta
t^{-1}).
\end{align*}

\end{proof}

\begin{proposition}
\label{Small}If $S\in\overline{\mathcal{S}}$ then $\mathcal{P}(\mathrm{Aut}%
(S);D,\varepsilon)$ holds for $\lambda=1$, with%
\[
D=N_{1}^{2}\delta_{1},~~\varepsilon=N_{0}^{-2\delta_{1}}.
\]

\end{proposition}

\begin{proof}
In view of Proposition \ref{firstredn}, it will suffice to establish
$\mathcal{P}(\mathrm{Inn}(S);\delta_{1},\varepsilon)$. Let $X\subseteq
S^{(2\delta_{1})}$ satisfy $\left\vert X\right\vert \geq(1-\varepsilon
)\left\vert S\right\vert ^{2\delta_{1}}$; then $X=S^{(2\delta_{1})}$. Let
$\alpha,\beta\in\mathrm{Inn}(S)^{(\delta_{1})}$. Using Lemma \ref{inn} we
obtain%
\begin{align*}
\mathbf{T}_{\alpha,\beta}(X)  &  =%
{\displaystyle\prod\limits_{i=1}^{\delta_{1}}}
T_{\alpha_{i},\beta_{i}}(S,S)\\
&  =%
{\displaystyle\prod\limits_{i=1}^{\delta_{1}}}
\gamma(S\times S)t_{i}=\gamma(S\times S)^{\ast\delta_{1}}\cdot t=S,
\end{align*}
where $t=t_{1}\ldots t_{\delta_{1}}$. The result follows.
\end{proof}

\subsubsection{Inner automorphisms}

The key to this case is a result is due to Garion and Shalev. In \cite{GaSh}
they define for each finite group $G$ the invariant%
\[
\epsilon(G)=\left(  \zeta^{G}(2)-1\right)  ^{1/4},
\]
where $\zeta^{G}(2)=%
{\textstyle\sum}
\chi(1)^{-2}$ summed over irreducible characters $\chi$ of $G$.

\begin{proposition}
\label{gash}\emph{ (\cite{GaSh}, Corollary 1.4(ii))} For $W\subseteq G\times
G$ and $\eta\in(0,1)$,%
\[
\left\vert W\right\vert \geq(1-\eta)\left\vert G\right\vert ^{2}%
\Longrightarrow\left\vert \gamma(W)\right\vert \geq(1-\eta-3\epsilon
(G))\left\vert G\right\vert .
\]

\end{proposition}

This is useful in combination with Theorem 1.1 of \cite{LiSh2}, which implies
that $\zeta^{G}(2)\rightarrow1$ as $\left\vert G\right\vert \rightarrow\infty$
when $G$ ranges over quasisimple groups. We may therefore choose $N_{2}%
\in\mathbb{N}$ so that $\epsilon(S)<\frac{1}{24}$ for every quasisimple group
$S$ with $\left\vert S\right\vert \geq N_{2}$.

\begin{proposition}
\label{Inner}Let $S$ be a quasisimple group with $\left\vert S\right\vert \geq
N_{2}$. Then $\mathcal{P}(\mathrm{Inn}(S);1,\frac{1}{8})$ holds with
$\lambda=l(S)^{-3/5}$.
\end{proposition}

\begin{proof}
Let $\alpha,\beta\in\mathrm{Inn}(S)$ and let $X\subseteq S^{(2)}$ satisfy
$\left\vert X\right\vert \geq\frac{7}{8}\left\vert S\right\vert ^{2}$.
According to Lemma \ref{inn}, there exist $t\in S$ and a subset $Y$ of
$S^{(2)}$ with $\left\vert Y\right\vert =\left\vert X\right\vert $ such that
$\mathbf{T}_{\alpha.\beta}(X)=\gamma(Y)t.$

By Proposition \ref{gash} we have%
\[
\left\vert \gamma(W)\right\vert \geq\left(  \frac{7}{8}-3\epsilon(S)\right)
\left\vert S\right\vert \geq\frac{3}{4}\left\vert S\right\vert .
\]
Therefore%
\[
\left\vert \mathbf{T}_{\alpha.\beta}(X)\right\vert =\left\vert \gamma
(Y)\right\vert \geq\left\vert \gamma(W)\right\vert \geq\frac{3}{4}\left\vert
S\right\vert >l(S)^{-3/5}\left\vert S\right\vert
\]
since $l(S)\geq2$.
\end{proof}

\subsubsection{Diagonal automorphisms}

In this subsection and the next, we consider a quasisimple group $S$ of Lie
type, of untwisted rank $r$. This means (\cite{GLS}, Section 2.2) that $S$ is
the group of fixed points of a Steinberg automorphism $\sigma$ of order
$k\in\{1,2,3\}$ of some untwisted Lie type group $S^{\diamondsuit}%
\leq\mathrm{GL}_{d}(q^{k})$ of rank $r$ (where $k=1$ precisely when
$S=S^{\diamondsuit}$ is untwisted). We denote by $D\leq\mathrm{GL}_{d}(q^{k})$
the group of diagonal matrices that induce diagonal automorphisms on $S$. Thus
$S\vartriangleleft SD$ and the restriction to $S$ of the inner automorphisms
of $SD$ is the group $\mathrm{InnDiag}(S)$ of inner-diagonal automorphisms of
$S$. We will use the facts (\emph{loc. cit.} Section 2.5):%
\begin{align*}
\left\vert SD:S\right\vert  &  \leq r+1,\\
\left\vert \mathrm{Z}(SD)\right\vert  &  =\left\vert \mathrm{Z}(S)\right\vert
\leq r+1.
\end{align*}

An abelian subgroup of $S$ consisting of semisimple elements and maximal with
this property will be called a \emph{maximal torus} of $S$ (this is the same
as the intersection with $S$ of a maximal torus in the underlying algebraic
group). The following estimate is easily derived from \cite{C}, Proposition 3.3.5:

\begin{lemma}
\label{torus}The size of a maximal torus of $S$ is at most $(q+1)^{r}$.
\end{lemma}

\begin{proposition}
\label{inndi}There exists $N_{3}\in\mathbb{N}$ such that if $\left\vert
S\right\vert \geq N_{3}$, $r\geq9$ and $q>10$ then $\mathcal{P}%
(\mathrm{InnDiag}(S);8,10^{-3})$ holds with $\lambda=l(S)^{-3/5}$.
\end{proposition}

This will be deduced from the next two results:

\begin{proposition}
\label{reg}\emph{\cite{GL}} If $S$ is a classical group and $h\in D$ then the
number of regular semisimple elements in the coset $Sh$ is at least $\left(
1-\frac{3}{q-1}-\frac{1}{(q-1)^{2}}\right)  \left\vert S\right\vert $, which
exceeds $\frac{2}{3}\left\vert S\right\vert $ if $q>10$.
\end{proposition}

\noindent(This follows from the proof of \cite{GL}, though it is not
explicitly stated there in this form.)

\begin{proposition}
\label{qsimple}Assume that $r\geq9$ and $q>10$. Let $h_{1},h_{2},\ldots,h_{8}$
be regular semisimple elements of $SD$, and let $X\subseteq S^{(8)}$ satisfy
$\left\vert X\right\vert \geq\frac{1}{4}\left\vert S\right\vert ^{8}$. Then
provided $\left\vert S\right\vert $ is sufficiently large, the number of
elements $g\in S$ such that
\[
\mathbf{c}(\mathbf{x},\mathbf{h})=\prod_{i=1}^{8}[x_{i},h_{i}]=g
\]
has a solution $\mathbf{x}=(x_{1},\ldots,x_{8})\in X$ is at least $\frac{1}%
{6}|S|$.
\end{proposition}

Before proving this let us deduce Proposition \ref{inndi}. Let $\alpha
,\beta\in\mathrm{InnDiag}(S)^{(8)}$ and let $Y\subseteq S^{(8)}$ satisfy
$\left\vert Y\right\vert \geq(1-(\frac{2}{3})^{8})\left\vert S\right\vert
^{8}$. There exist $c_{i}\in S$ and $h_{i}^{\prime}\in D$ such that
$\alpha_{i}$ is induced by $c_{i}h_{i}^{\prime}$ ($i=1,\ldots,8$). Put
$Y^{\prime}=\{\mathbf{y}\cdot\mathbf{c}\cdot\mathbf{h}^{\prime}\mid
\mathbf{y}\in Y\}$. Then $\left\vert Y^{\prime}\right\vert =\left\vert
Y\right\vert $, so Proposition \ref{reg} ensures that $Y^{\prime}$ contains at
least one element $\mathbf{y}\cdot\mathbf{c}\cdot\mathbf{h}^{\prime}%
=(h_{1},\ldots,h_{8})$ with each $h_{i}$ regular semisimple. Then provided
$\left\vert S\right\vert $ is sufficiently large, Proposition \ref{qsimple}
gives%
\[
\left\vert \mathbf{c}(X,\mathbf{y}\cdot\alpha)\right\vert =\left\vert
\mathbf{c}(X,\mathbf{h})\right\vert \geq\tfrac{1}{6}|S|
\]
whenever $X\subseteq S^{(8)}$ satisfies $\left\vert X\right\vert \geq\frac
{1}{4}\left\vert S\right\vert ^{8}$. Applying Proposition \ref{DtoCredn} we
infer that $\mathcal{P}(\alpha,\beta;8,(\frac{2}{3})^{16})$ holds with
$\lambda=\frac{1}{6}$. Now Proposition \ref{inndi} follows, since $(\frac
{2}{3})^{16}>10^{-3}$ and $l(S)\geq\frac{1}{2}(11^{9}-1)>6^{5/3}$ by
Proposition \ref{l(simple)}.

\bigskip

\textbf{Proof of Proposition \ref{qsimple}.} Relabelling $(h_{1}^{-1}%
,h_{2}^{-h_{1}^{-1}},\ldots,h_{8}^{-(h_{1}\ldots h_{7})^{-1}})$ as
$(k_{1},k_{2},\ldots,k_{8})$ and $(x_{1},x_{2}^{h_{1}^{-1}},\ldots
,x_{8}^{(h_{1}\cdots h_{7})^{-1}})$ as $(y_{1},\ldots,y_{8})$, it will suffice
to prove that the image of the map
\[
f:(y_{1},\ldots,y_{8})\mapsto k_{1}^{y_{1}}k_{2}^{y_{2}}\cdots k_{8}^{y_{8}%
}\cdot(k_{1}\cdots k_{8})^{-1}\in S
\]
has size at least $\frac{1}{6}|S|$ when $(y_{1},\ldots,y_{8})$ ranges over a
subset of $S^{(8)}$ of proportion $\frac{1}{4}$.

Write $G=SD$. We observe that if $g$ is a semisimple element of $S$ then
$\mathrm{C}_{G}(g)$ contains a maximal torus of $G$ and so maps onto $G/S$.
This means that the conjugacy class $g^{S}$ of $g$ in $S$ is the same as the
conjugacy class of $g$ in $G$. Now we count solutions in conjugacy classes of
$G$:

\begin{lemma}
\label{char} Assume that $r\geq9$, $q>10$. Let $\delta>0$ and let
$k_{1},\ldots,k_{8}$ be regular semisimple elements of $G$. Put $c_{i}%
=\left\vert k_{i}^{G}\right\vert $. There is an integer $N_{\delta}$ such that
if $\left\vert S\right\vert \geq N_{\delta}$ then the following holds:

For every $g\in S$ the number of $8$-tuples $(a_{1},\ldots,a_{8})\in k_{1}%
^{G}\times\cdots\times k_{8}^{G}$ such that
\[
a_{1}\ldots a_{8}=gk_{1}\cdots k_{8}%
\]
is
\[
\frac{c_{1}\ldots c_{8}}{|S|}(1+\gamma_{g})~\text{where }|\gamma_{g}|<\delta.
\]

\end{lemma}

Assuming this for the moment we can finish the proof of Proposition
\ref{qsimple}. Take $\delta=\frac{1}{2}$ and assume that $\left\vert
S\right\vert \geq N_{\delta}$. Then Lemma \ref{char} implies that for each
$g\in S$ we have%
\begin{align*}
\left\vert f^{-1}(g)\right\vert  &  =\prod_{i=1}^{8}|\mathrm{C}_{S}%
(k_{i})|\cdot\frac{c_{1}\ldots c_{8}}{\left\vert S\right\vert }(1+\gamma
_{g})\\
&  =\left\vert S\right\vert ^{7}(1+\gamma_{g})<\tfrac{3}{2}\left\vert
S\right\vert ^{7}.
\end{align*}
Suppose that $Y\subseteq S^{(8)}$ satisfies $\left\vert Y\right\vert \geq
\frac{1}{4}\left\vert S\right\vert ^{8}$. Then%
\[
\left\vert f(Y)\right\vert >\frac{\left\vert Y\right\vert }{\tfrac{3}%
{2}\left\vert S\right\vert ^{7}}\geq\tfrac{1}{6}\left\vert S\right\vert ,
\]
as required.

\bigskip\medskip

\textbf{Proof of Lemma \ref{char}} Let $\chi$ be an irreducible character of
$G$. By Clifford theory $\chi\!\downarrow_{S}$ is a sum of irreducible
characters of $S$, say $\psi+\phi+\cdots$. Then $\chi(1)\geq\psi(1)$. Now if
$\chi$ is nonlinear then $\psi\in\mathrm{Irr}(S)$ is also nonlinear, and hence
$\chi(1)\geq\psi(1)\geq cq^{r}$ for some absolute constant $c$, by Proposition
\ref{l(simple)}.

Put $p=gk_{1},\ldots,k_{8}$ and let $s(p)$ denote the number of the number of
$8$-tuples $(a_{1},\ldots,a_{8})\in k_{1}^{G}\times\cdots\times k_{8}^{G}$
such that $a_{1}a_{2}\ldots a_{8}=p$.

A well-known formula (cf. \cite{SGT}, 7.2) gives
\[
s(p)=\frac{c_{1}\ldots c_{8}}{|G|}\sum_{\chi\in\mathrm{Irr}(G)}\frac
{\chi(k_{1})\ldots\chi(k_{8})\chi(p^{-1})}{\chi(1)^{7}}.
\]
Since $k_{1}\cdots k_{8}p^{-1}\in S$ and hence lies inside $\ker\chi$ for any
linear character $\chi$ of $G$, these contribute precisely $\left\vert
G/G^{\prime}\right\vert =|G|/|S|$ to the above sum. It therefore suffices to
show that%

\[
\left\vert G:S\right\vert \sum_{\chi\in\mathrm{Irr}_{0}(G)}\frac{\chi
(k_{1})\ldots\chi(k_{8})\chi(p^{-1})}{\chi(1)^{7}}\rightarrow0\text{ as
}|S|\rightarrow\infty,
\]
where $\mathrm{Irr}_{0}(G)$ denotes the set of non-linear irreducible
characters of $G$.

Since $|\chi(p^{-1})|/\chi(1)\leq1$ it is enough to show that
\[
V:=(r+1)\sum_{\chi\in\mathrm{Irr}_{0}(G)}\frac{\chi(k_{1})\ldots\chi(k_{8}%
)}{\chi(1)^{6}}\rightarrow0\text{ as }|S|\rightarrow\infty.
\]

Now since $k_{i}$ is regular semisimple, $\mathrm{C}_{G}(k_{i})$ is a torus of
$G=SD$, and so $\left\vert \mathrm{C}_{G}(k_{i})\right\vert \leq(q+1)^{r+1}$
by Lemma \ref{torus}.

Hence $\left\vert \chi(k_{i})\right\vert \leq\sqrt{|\mathrm{C}_{G}(k_{i}%
)|}\leq(q+1)^{(r+1)/2}$, and we obtain
\[
\left\vert \chi(k_{1})\ldots\chi(k_{8})\right\vert \chi(1)^{-6}\leq
\frac{((q+1)^{r+1})^{4}}{c^{6}q^{6r}}=c^{-6}(q+1)^{4+4r}q^{-6r}.
\]
By Corollary 1.2 (3) of \cite{FG}, $\left\vert \mathrm{Irr}(G/\mathrm{Z}%
(G))\right\vert \leq100q^{r}$, whence $\left\vert \mathrm{Irr}(G)\right\vert
\leq100q^{r}(r+1)$. Moreover $q+1<q^{1.1}$ when $q>10$. Consequently%

\[
V\leq c_{5}(r+1)^{2}q^{4.4-.6r}%
\]
for some absolute constant $c_{5}>0$.

As $r\geq9$ we have $0.6r>4.4$; consequently $V\rightarrow0$ as $\left\vert
S\right\vert \rightarrow\infty$, as required.

\subsubsection{Field automorphisms\label{fieldautos}}

As in the preceding subsection, $S$ denotes a quasisimple group of Lie type,
of untwisted rank $r$. We assume that $S$ is universal, and introduce some
more notation (cf. \cite{GLS}, Section 2.2). $L$ is a simple simply connected
algebraic group $L$ defined over $\mathbb{F}_{p}$, and $S=L_{\sigma}\leq
L(\mathbb{F}_{q^{k}})$ is the group of $\sigma$-fixed points of a Steinberg
automorphism $\sigma$ acting on $L$. Here $k\in\{1,2,3\}$ and $\sigma^{k}$ is
the smallest power of $\sigma$ which is a power of the Frobenius automorphism
$[p]$ of $L$. In fact $\sigma$ is the product of a graph automorphism of $L$
and some power of $[p]$, so $\sigma$ commutes with all field automorphisms of
$L$.

We consider $L$ as embedded in some $\mathrm{GL}_{d}$. Then $\mathrm{GL}_{d}$
contains a torus $T$ that normalizes $L$ and induces the diagonal
automorphisms on $L$. In the same way $D=T_{\sigma}$ induces the diagonal
automorphisms of $S=L_{\sigma}$.

We consider a field automorphism $\phi$ of $S$. Thus $\phi$ is the restriction
to $S$ of $[p]^{f}$ for some $f$, and we shall denote $[p]^{f}$ also by $\phi
$. Then $\phi^{n}=\sigma^{k}$ where $n$ is the order of $\phi$ as an
automorphism of $S$.

Let $q_{0}$ denote the cardinality of the fixed field of $\phi$. Thus
$q_{0}=p^{f}$, while $\mathbb{F}_{q^{k}}$ is the fixed field of $\phi^{n}$, so
$q^{k}=p^{nf}$ and%

\begin{equation}
q^{k}=q_{0}^{n}.\label{qandq0}%
\end{equation}
We remark that $k\leq2$ unless $S$ is of type $^{3}\!D_{4}$, with $r=4$; and
$q$ might be the square root of a non-square integer if $S$ is a Suzuki or Ree group.

For an algebraic subgroup $M$ of $\mathrm{GL}_{d}(\overline{\mathbb{F}_{p}})$
we denote by $M_{\phi}$ the fixed-point set of $\phi$ in $M$. Later, we shall
need to consider the groups%
\[
G=L_{\phi},\quad H=T_{\phi}%
\]
Thus $G$ is an untwisted quasisimple group of Lie type, say $X$, over
$\mathbb{F}_{q_{0}}$ of rank $r$ equal to the rank of $L$. The group $H$
induces the diagonal automorphisms on $G$. Since $\sigma$ commutes with $\phi$
it preserves $G$ and acts on it as an automorphism of order $k$ (since
$\sigma^{k}=\phi^{n}$).

\medskip

We shall consider automorphisms%
\[
\alpha=ch\phi^{-1}%
\]
where

\begin{itemize}
\item $\phi$ is a field automorphism of $S$ having order $n>50$,

\item $h$ is a diagonal automorphism of $S$ (we identify $h$ with an element
of $D$),

\item $c$ is an inner automorphism of $S$ (we will identify $c$ with an
element of $S$).
\end{itemize}

\begin{proposition}
\label{largefieldauto}With $\alpha$ as above, $\mathcal{P}(\alpha
,\beta;1,\frac{3}{5})$ and $\mathcal{P}(\beta,\alpha^{-1};1,\frac{3}{5})$ hold
for every $\beta\in\mathrm{Aut}(S)$, with $\lambda=l(S)^{-3/5}$.
\end{proposition}

This will follow from

\begin{proposition}
\label{1/2}Let%
\begin{equation}
W=\left\{  x\in S\mid\left\vert \mathrm{C}_{S}(xh\phi^{-1})\right\vert
<l(S)^{1/2}\right\}  . \label{Wdef}%
\end{equation}
Then $\left\vert W\right\vert >\frac{4}{5}\left\vert S\right\vert $.
\end{proposition}

To deduce Proposition \ref{largefieldauto}, suppose $Y\subseteq S$ satisfies
$\left\vert Y\right\vert \geq\frac{1}{5}\left\vert S\right\vert $. Then
$Yc\cap W$ is non-empty; choose $y\in Y$ with $yc\in W$. Then $\left\vert
\mathrm{C}_{S}(y\alpha)\right\vert <l(S)^{1/2}$, so for any subset $X$ of $S$
with $\left\vert X\right\vert \geq\frac{1}{5}\left\vert S\right\vert $ we have%
\begin{align*}
\left\vert \mathbf{c}(X,y\alpha)\right\vert  &  =\left\vert \left\{
[x,y\alpha]\mid x\in X\right\}  \right\vert \\
&  \geq\left\vert X\right\vert l(S)^{-1/2}\geq\frac{1}{5}l(S)^{-1/2}\left\vert
S\right\vert .
\end{align*}
With Proposition \ref{DtoCredn}, this shows that $\mathcal{P}(\alpha
,\beta;1,\frac{3}{5})$ holds with $\lambda=\frac{1}{5}l(S)^{-1/2}$ (as
$\frac{3}{5}<(\frac{4}{5})^{2}$). Since%
\[
T_{\beta^{-1},\alpha^{-1}}(x,y)=T_{\alpha,\beta}(y^{\alpha^{-1}},x^{\beta
^{-1}})^{-1},
\]
this implies also that $\mathcal{P}(\beta^{-1},\alpha^{-1};1,\frac{3}{5})$
holds with the same value of $\lambda$.

Suppose that $k\leq2$. Then (\ref{qandq0}) implies that $q>2^{25}$.
Proposition \ref{l(simple)} then implies that $l(S)\geq(q-1)/2\geq2^{24}$. If
$k=3$ then $S=~^{3}\!D_{4}(q)$ and Proposition \ref{l(simple)} gives
$l(S)\geq(q^{4}-1)/2>(2^{4\cdot50/3}-1)/2>2^{65}$. In any case, then,
$l(S)^{-1/10}\leq2^{-2.4}<\frac{1}{5}$, whence%
\[
\frac{1}{5}l(S)^{-1/2}>l(S)^{-3/5}.
\]

Proposition \ref{largefieldauto} follows.

\bigskip

We proceed to the proof of Proposition \ref{1/2}. We are given $h\in
D=T_{\sigma}$. By Lang's theorem (\cite{GLS}, Theorem 2.1.1) we may choose
$\kappa\in T$ with $h=\kappa^{-1}\kappa^{\phi}$. Put $h^{\prime}=\kappa
\kappa^{-\sigma}$. Note that%
\begin{align*}
(\kappa^{-1}\kappa^{\phi})^{\sigma}  &  =h^{\sigma}=h=\kappa^{-1}\kappa^{\phi
},\\
h^{\prime\phi}  &  =(\kappa\kappa^{-\sigma})^{\phi}=\kappa\kappa^{-\sigma
}=h^{\prime},
\end{align*}
so $h^{\prime}\in H$. Define%
\begin{align*}
\mu,\nu &  :L\rightarrow LT\\
\mu(x)  &  =[x\kappa,\phi],~\nu(x)=[(x\kappa)^{-1},\sigma].
\end{align*}

\begin{lemma}
\emph{(i) }$\mu^{-1}(Sh)=\nu^{-1}(Gh^{\prime})$;

\emph{(ii) }if $g\in Sh$ then $\left\vert \mu^{-1}(g)\right\vert =\left\vert
G\right\vert $;

\emph{(iii)} if $z\in Gh^{\prime}$ then $\left\vert \nu^{-1}(z)\right\vert
=\left\vert S\right\vert .$
\end{lemma}

\begin{proof}
(i).%
\begin{align*}
\mu(x)  &  \in Sh\Longleftrightarrow\mu(x)^{\sigma}=\mu(x)\\
&  \Longleftrightarrow\kappa^{-1}x^{-1}x^{\phi}\kappa^{\phi}=\kappa^{-\sigma
}x^{-\sigma}x^{\sigma\phi}\kappa^{\sigma\phi}\\
&  \Longleftrightarrow x^{\sigma}\kappa^{\sigma}\kappa^{-1}x^{-1}%
=x^{\sigma\phi}\kappa^{\sigma\phi}\kappa^{-\phi}x^{-\phi}\\
&  \Longleftrightarrow\nu(x)=v(x)^{\phi}\Longleftrightarrow\nu(x)\in
Gh^{\prime}.
\end{align*}
(ii), (iii). Let $g\in Sh$. Then $g=\kappa^{-1}g^{\prime}\kappa^{\phi}$ with
$g^{\prime}\in L$, and by Lang's theorem again we have $g^{\prime}=[x,\phi]$
for some $x\in L$. Then $\mu(x)=g$, and we see that $\mu^{-1}(g)=xL_{\phi}%
=xG$. Similarly we find that $\nu^{-1}(z)=y\kappa L_{\sigma}\kappa
^{-1}=yS^{\kappa^{-1}}$ where $y=\kappa y_{1}\kappa^{-1}$ and $z=\kappa\cdot
y_{1}y_{1}^{-\sigma}\cdot\kappa^{\sigma}$.
\end{proof}

\bigskip

Now consider the semi-direct product $G_{1}=GH\rtimes\langle\sigma\rangle$. We
define a permutation action of $G$ on $G_{1}$ as follows: for $x\in G$ and
$a\in G_{1}$,%
\[
a^{\widehat{x}}=x^{-1}ax^{\sigma}.
\]
We will call this the \emph{twisted action}. For $a\in G_{1}$ we denote the
stabilizer of $a$ in $G$ under this action by $C(a)$, i.e.%
\[
C(a)=\{x\in G\mid ax^{\sigma}=xa\}.
\]
Set $Y=\mu^{-1}(Sh)=\nu^{-1}(Gh^{\prime})$.\bigskip

\begin{lemma}
\label{central} Let $y\in Y$ and put $g=\mu(y)$, $z=\nu(y)$ Then
\[
\left\vert \mathrm{C}_{S}(g\phi^{-1})\right\vert =\left\vert C(z)\right\vert
.
\]

\end{lemma}

\begin{proof}
Let $a\in L$ and put $b=y\kappa a\kappa^{-1}y^{-1}$. The condition
$a^{g\phi^{-1}}=a$ is equivalent to $b^{\phi}=b$, i.e. $b\in L_{\phi}=G$. The
condition $a\in S=L_{\sigma}$ is equivalent to $(b^{y\kappa})^{\sigma
}=b^{y\kappa}$, i.e. $zb^{\sigma}=bz$. So%
\[
\mathrm{C}_{S}(g\phi^{-1})=(y\kappa)^{-1}C(z)y\kappa.
\]

\end{proof}

\bigskip

If we put%
\begin{align*}
Z  &  =\left\{  z\in Gh^{\prime}\mid\left\vert C(z)\right\vert <l(S)^{1/2}%
\right\} \\
Y^{\ast}  &  =\nu^{-1}(Z),
\end{align*}
the two preceding lemmas give%
\begin{equation}
\left\vert W\right\vert =\left\vert G\right\vert ^{-1}\left\vert Y^{\ast
}\right\vert =\left\vert S\right\vert \left\vert G\right\vert ^{-1}\left\vert
Z\right\vert . \label{WandZ}%
\end{equation}

\begin{lemma}
\label{fixed}\emph{(i) }If $S\neq~^{3}\!D_{4}(q)$ then $l(S)^{1/2}>q_{0}%
^{11r}$;

\emph{(ii)} If $S=~^{3}\!D_{4}(q)$ then $l(S)^{1/2}>\left\vert G\right\vert $.
\end{lemma}

\begin{proof}
Proposition \ref{l(simple)} says that $l(S)$ is at least $(q^{r}-1)/2$. Also
$q=q_{0}^{n/k}\geq q_{0}^{51/k}$.

In case (i) we have $k\leq2$. Then $l(S)>\frac{1}{2}(q_{0}^{25r}-1)$, whence
$l(S)\geq q_{0}^{24r}$ and the result follows.

In case (ii), $k=3$ and $G=D_{4}(q_{0})$. In this case, we have%
\begin{align*}
l(S)  &  \geq(q^{4}-1)/2\geq(q_{0}^{68}-1)/2,\\
\left\vert G\right\vert  &  <q_{0}^{28}<l(S)^{1/2}.
\end{align*}

\end{proof}

\bigskip

Since $C(z)\leq G$ for each $z\in Gh^{\prime}$, it follows in case (ii) that
$Z=Gh^{\prime}$ and hence that $\left\vert W\right\vert =\left\vert
S\right\vert $.

Henceforth, \textbf{we assume that} $S\neq~^{3}\!D_{4}(q)$.

Let $c(G)$ denote the number of conjugacy classes of $G$.

\begin{lemma}
\label{orbitequiv}The coset $Gh^{\prime}\subseteq G_{1}$ is a union of at most
$\left\vert G_{1}:G\right\vert c(G)$ orbits of $G$ with the twisted action.
\end{lemma}

\begin{proof}
For $z\in Gh^{\prime}$ and $x\in G$ we have%
\[
(z\cdot\sigma^{-1})^{x}=x^{-1}zx^{\sigma}\cdot\sigma^{-1}=z^{\widehat{x}}%
\cdot\sigma^{-1}%
\]
in $G_{1}$. This shows that the twisted action on $G$ on the coset
$Gh^{\prime}$ is equivalent to the conjugation action of $G$ on $Gh^{\prime
}\sigma^{-1}\subseteq G_{1}$. The number of orbits of $G$ acting by
conjugation on $G_{1}$ is%
\begin{align*}
\left\vert G\right\vert ^{-1}\sum_{g\in G}\left\vert \mathrm{C}_{G_{1}%
}(g)\right\vert  &  \leq\left\vert G\right\vert ^{-1}\left\vert G_{1}%
:G\right\vert \sum_{g\in G}\left\vert \mathrm{C}_{G}(g)\right\vert \\
&  =\left\vert G\right\vert ^{-1}\left\vert G_{1}:G\right\vert \left\vert
G\right\vert c(G)=\left\vert G_{1}:G\right\vert c(G).
\end{align*}
The result follows.
\end{proof}

\bigskip

Since $G=L_{\phi}$ is a quasisimple group of untwisted Lie type,%
\[
\left\vert GH:G\right\vert \leq\left\vert \mathrm{Outdiag}(G)\right\vert
\leq\min\{r+1,q_{0}-1\}<q_{0}.
\]
The automorphism $\sigma$ has order $1$ or $2$. Thus $\left\vert
G_{1}:G\right\vert \leq2q_{0}.$ Now Theorem 1.1 (1) in \cite{FG} shows that
$c(G)\leq30q_{0}^{r}$. Applying Lemma \ref{fixed}, we deduce that if $y\in
Gh^{\prime}\smallsetminus Z$ then
\[
\left\vert y^{\widehat{G}}\right\vert =\frac{\left\vert G\right\vert
}{\left\vert C(y)\right\vert }<\frac{\left\vert G\right\vert }{q_{0}^{11r}};
\]
Hence by Lemma \ref{orbitequiv} $Gh^{\prime}\smallsetminus Z$ is the union of
at most $60q_{0}^{r+1}$ orbits of this size, whence%
\[
\left\vert Gh^{\prime}\smallsetminus Z\right\vert <60q_{0}^{-10r+1}\left\vert
G\right\vert .
\]
Therefore $\left\vert Z\right\vert \geq\eta\left\vert G\right\vert $ where
$\eta=1-60/2^{9}>\frac{4}{5}$.

Now Proposition \ref{1/2} follows from (\ref{WandZ}).

\subsubsection{Proof of Theorem \ref{T1}}

As explained in Subsection \ref{P-reductions}, we have to find $D\in
\mathbb{N}$ and $\varepsilon>0$ such that $\mathcal{P}(\mathrm{Aut}%
(S);D,\varepsilon)$ holds with $\lambda$ given by (\ref{lambda}) for every
quasisimple group $S$: i.e. $\lambda=l(S)^{-3/5}$ if $l(S)\geq3$, $\lambda=1$
if $l(S)=2$. Henceforth, when we say that $\mathcal{P}(\ldots)$ holds for some
group $S$, we will mean that it holds with $\lambda$ given by (\ref{lambda}).

Set $N_{0}=\max\{N_{2},N_{3},1+\left\vert M\right\vert \}$ where $N_{i}$ are
the bounds introduced above and $M$ denotes the largest sporadic (quasi)simple
group (it happens to be simple).

Now let $S$ be a quasisimple group. We consider several cases.

$\medskip$

\emph{Case 1.} Where $\left\vert S\right\vert <N_{0}$. Proposition \ref{Small}
shows that $\mathcal{P}(\mathrm{Aut}(S);D_{1},\varepsilon_{1})$ holds for some
$D_{1}$ and $\varepsilon_{1}$.

$\medskip$

We assume henceforth that $\left\vert S\right\vert \geq N_{0}$. Putting
$\Gamma_{0}=\mathrm{Inn}(S)$, Proposition \ref{Inner} shows that
$\mathcal{P}(\Gamma_{0};1,\frac{1}{8})$ holds.

$\medskip$

\emph{Case 2.} Where $S/\mathrm{Z}(S)$ is an alternating group. Then
$\left\vert \mathrm{Aut}(S):\Gamma_{0}\right\vert =2$, and Proposition
\ref{firstredn} gives $\mathcal{P}(\mathrm{Aut}(S);4,\frac{1}{8})$.

$\medskip$

From now on, $S$ is a group of Lie type, of rank $r$ over $\mathbb{F}_{q}$. We
denote by $\Phi$ the group of field automorphisms of $S$. Then $\mathrm{Aut}%
(S)$ has normal subgroups%
\[
\mathrm{Aut}(S)\geq\Gamma\geq\Gamma_{1}\geq\Gamma_{2}\geq\Gamma_{0}%
=\mathrm{Inn}(S)
\]
where $\Gamma_{2}=\mathrm{InnDiag}(S)$, $\Gamma=\Gamma_{2}\Phi,$ and
$\Gamma_{1}=\Gamma_{2}\Phi_{1}$ where $\Phi_{1}$ is the subgroup of $\Phi$
generated by all elements of order at most $50$.

Put $n_{0}=\operatorname{lcm}[50]$, and define $n_{1}=\min\{q+1,r+1\}$ if $S$
has type $A_{r}$ or $^{2}A_{r}$, $n_{1}=4$ otherwise. We have%
\begin{align*}
\left\vert \mathrm{Aut}(S):\Gamma\right\vert  &  \leq6,\\
\left\vert \Gamma:\Gamma_{2}\right\vert  &  \leq\log_{p}(q^{3})\leq3\log
_{2}(q),\\
\left\vert \Gamma_{1}:\Gamma_{2}\right\vert  &  \leq n_{0},\\
\left\vert \Gamma_{2}:\Gamma_{0}\right\vert  &  \leq n_{1}%
\end{align*}
where $p=\mathrm{char}(\mathbb{F}_{q})$ (see \cite{GLS}, Section 2.5).

$\medskip$

\emph{Case 3.} Where $q\leq10$. In this case, $\left\vert \mathrm{Aut}%
(S):\Gamma_{0}\right\vert \leq600$. As in Case 2, we may deduce that
$\mathcal{P}(\mathrm{Aut}(S);D_{2},\frac{1}{8})$ holds where $D_{2}=360,000$.

$\medskip$

\emph{Case 4. }Where $q>10$. If $r<9$ we have $\left\vert \Gamma_{2}%
:\Gamma_{0}\right\vert \leq n_{1}\leq9$; we deduce as before that
$\mathcal{P}(\Gamma_{2};81,\frac{1}{8})$ holds. If $r\geq9$, Proposition
\ref{inndi} gives $\mathcal{P}(\Gamma_{2};8,\frac{1}{2})$. Taking
$D_{3}=81n_{0}^{2}$, we infer in any case that $\mathcal{P}(\Gamma_{1}%
;D_{3},\frac{1}{8})$ holds, whatever the rank $r$.

Now let $\alpha,\beta\in\Gamma^{(D_{3})}$. If $\alpha_{i}$ and $\beta_{i}$ lie
in $\Gamma_{1}$ for every $i$ then we have $\mathcal{P}(\alpha,\beta
;D_{3},\frac{1}{8})$. If not, let us suppose for convenience that $\alpha
_{1}\notin\Gamma_{1}$. Then $\alpha_{1}=ch\phi$ where $c\in\Gamma_{0}$, $h$ is
diagonal, and $\phi\in\Phi$ has order exceeding $50$. Proposition
\ref{largefieldauto} now shows that $\mathcal{P}(\alpha_{1},\beta_{1};1,3/5)$
holds. As in the proof of Proposition \ref{obviousredn}, this in turn implies
$\mathcal{P}(\alpha,\beta;D_{3},3/5)$.

Thus $\mathcal{P}(\Gamma;D_{3},\frac{1}{8})$ holds in either case. Since
$\left\vert \mathrm{Aut}(S):\Gamma\right\vert \leq6$, a final application of
Proposition \ref{firstredn} gives $\mathcal{P}(\mathrm{Aut}(S);D_{4},\frac
{1}{8})$ where $D_{4}=36D_{3}$.

$\medskip$

\emph{Conclusion. }Take $D=\max\{4,D_{1},D_{2},D_{4}\}$ and $\varepsilon
=\min\{\varepsilon_{1},\frac{1}{8}\}$. Then $\mathcal{P}(\mathrm{Aut}%
(S);D,\varepsilon)$ holds in all cases, by Proposition \ref{obviousredn}.

\subsection{Commutators in semisimple groups\label{ssns}}

In this subsection, $D$ and $\varepsilon$ are the constants introduced in
subsection \ref{quasi}. We will say that a multiset $Y$ has the $(k,\eta
)$\emph{-fpp} on a $\left\langle Y\right\rangle $-set $\Omega$ if at least $k$
elements of $Y$ have the $\eta$-fpp on $\Omega$.

\begin{theorem}
\label{comm20}Let $N$ be a finite quasisemisimple group with at least $3$
non-abelian composition factors. Let $\mathbf{y}_{1},\ldots,\mathbf{y}_{10}$
be $m$-tuples of automorphisms of $N$. Assume that for each $i$, the group
$\left\langle \mathbf{y}_{i}\right\rangle $ permutes the set $\Omega$ of
quasisimple factors of $N$ transitively and that $\mathbf{y}_{i}$ has the
$(k,\eta)$-fpp on $\Omega$, where $k\eta\geq4+2D$. For each \thinspace$i$ let
$W(i)\subseteq N^{(m)}$ be a subset with $\left\vert W(i)\right\vert
\geq(1-\varepsilon/6)\left\vert N\right\vert ^{m}$. Then%
\[%
{\displaystyle\prod\limits_{i=1}^{10}}
W(i)\phi(i)=N
\]
where $\phi(i):N^{(m)}\rightarrow N$ is given by%
\[
(x_{1},\ldots,x_{m})\phi(i)=%
{\displaystyle\prod\limits_{i=1}^{m}}
[x_{i},y_{ij}].
\]

\end{theorem}

The action of $\mathrm{Aut}(N)$ lifts to an action on the universal cover
$\widetilde{N}$ of $N$, and $\widetilde{N}=S_{1}\times\cdots\times S_{n}$
where the $S_{i}$ are quasisimple groups. Replacing $N$ by $\widetilde{N}$ and
each $W(i)$ by its inverse image in $\widetilde{N}^{(m)}$, we may suppose that
in fact $N=S_{1}\times\cdots\times S_{n}$. Since $\left\langle \mathbf{y}%
_{1}\right\rangle $ permutes $\Omega=\{S_{1},\ldots,S_{n}\}$ transitively, the
groups $S_{i}$ are all isomorphic to a quasisimple group $S $.

Now let $G=\left\langle g_{1},\ldots,g_{m}\right\rangle \leq\mathrm{Aut}(N)$
and denote by $e_{i}$ the number of cycles (including fixed points) of $g_{i}$
in its action on $\Omega$. Define $\phi:N^{(m)}\rightarrow N$ by%
\[
\mathbf{x}\phi=\mathbf{c}(\mathbf{x},\mathbf{g})=%
{\displaystyle\prod\limits_{i=1}^{m}}
[x_{i},g_{i}].
\]
We shall prove

\begin{proposition}
\label{thm}Suppose that that
\begin{equation}
(m-2)n-%
{\displaystyle\sum\limits_{i=1}^{m}}
e_{i}\geq2D. \label{ineq}%
\end{equation}

\emph{(1)} Let $W\subseteq N^{(m)}$ satisfy $\left\vert W\right\vert
\geq(1-\varepsilon/6)\left\vert N\right\vert ^{m}$. Then%
\[
\left\vert W\phi\right\vert \geq l(S)^{-4/5}\left\vert N\right\vert .
\]

\emph{(2)} If $D$ is replaced by $D_{1}=5D$, then $\phi$ is surjective, and
each fibre of $\phi$ has size at least $\left\vert N\right\vert ^{-2D_{1}%
/n}\left\vert N\right\vert ^{m-1}$.
\end{proposition}

Part (2) is a sharper version of \cite{NS}, Proposition 9.1; we will not be
needing it, and include it in a revisionist spirit, to show how the main
results of \cite{NS} can be reproduced using these methods.

To deduce Theorem \ref{comm20} from (1), note that for each $i=1,\ldots,10$,
the total number of cycles for $y_{i1},\ldots,y_{im}$ on $\Omega$ is at most%
\[
(m-k)n+k(1-\eta/2)n\leq(m-2)n-nD,
\]
which implies condition (\ref{ineq}) since $n\geq3$. So taking $\mathbf{g}%
=\mathbf{y}_{i}$ and writing $\phi(i)$ for the corresponding map $\phi$, we
may infer that%
\[
\left\vert W(i)\phi(i)\right\vert \geq l(S)^{-4/5}\left\vert N\right\vert .
\]
Now $l(S)=l(N)=l$, say, and we have%
\[%
{\displaystyle\prod\limits_{i=1}^{10}}
\left\vert W(i)\phi(i)\right\vert \geq\frac{\left\vert N\right\vert ^{10}%
}{l^{8}}.
\]
It follows by the `Gowers trick' that $%
{\displaystyle\prod\limits_{i=1}^{10}}
W(i)\phi(i)=N$, and this is the statement of Theorem \ref{comm20} since
$W(i)\phi(i)=%
{\displaystyle\prod\limits_{j=1}^{m}}
[W(i),y_{ij}]$.

\subsubsection{Proof of Proposition \ref{thm}}

\begin{lemma}
Suppose that $G=\left\langle g_{1},\ldots,g_{m}\right\rangle $ acts
transitively on a finite set $J$. Fix $t\in J$. Then there is a total order on
$J$ with minimal element $t$ such that for each $j>t$ there exist
$i(j)\in\lbrack m]$ and $\varepsilon_{j}\in\{\pm1\}$ such that $j\cdot
g_{i(j)}^{\varepsilon_{j}}<j$.
\end{lemma}

\begin{proof}
Let $X$ with $1\in X$ be a Schreier transversal to the right cosets of
$\mathrm{stab}_{G}(t)$: thus $X$ is a set of words on $\{g_{1},\ldots,g_{m}\}$
such that (1) $x\longmapsto t\cdot x$ is a bijection $X\rightarrow J$ and (2)
each initial segment of a word in $X$ is again in $X$, i.e. if a word
$vg_{i}^{\pm1}$ is in $X$ then $v\in X$. Now define the size of $j=t\cdot x$
to be the length of $x$, and finally order $J$ lexicographically by size.
\end{proof}

\bigskip

Keeping $G$ and $J$ as above, we label the elements of $J$ as $\{1,2,\ldots
,n\}$ in the given order, and fix $i(j),$ $\varepsilon_{j}$ ($j=2,\ldots,n$)
as in the lemma. Say $g_{i}$ has cycles $\Delta_{il}$, $l=1,\ldots,e_{i}$
(including cycles of length $1$); we also write%
\[
\Delta_{il}=\Delta_{i}(j)\text{ if }j\in\Delta_{il}.
\]
Let $\delta_{il}=\delta_{i}(j)$ denote the least member of $\Delta_{il}%
=\Delta_{i}(j)$, and set%
\[
\widehat{j}=\delta_{i(j)}(j),
\]
i.e. $\widehat{j}$ is the least element in the $\left\langle g_{i(j)}%
\right\rangle $-orbit of $j$. This implies that $\widehat{j}<j$ if $j>1$.

Put%
\begin{align*}
\Delta_{il}^{^{\prime}}  &  =\Delta_{il}\smallsetminus\{\delta_{il}\},\\
J_{i}^{^{\prime}}  &  =%
{\displaystyle\bigcup\limits_{l=1}^{e_{i}}}
\Delta_{il}^{^{\prime}}.
\end{align*}
In writing products labelled by $\Delta_{il}$, we will assume that
$\Delta_{il}$ is ordered as a $g_{i}$-cycle starting with $\delta_{il}$ (not
with the induced order from $J$).

Let $S$ be a finite group, $N=S^{J}$, and suppose that $G$ acts on $N$,
permuting the factors according to the action of $G$ on $J$. Write elements of
$N$ as $x=(x(j))_{j\in J}$.

For any subset $T$ of $[m]\times J$ write $\pi_{T}:N^{(m)}\rightarrow S^{T}$
for the projection map%
\[
(x_{1},\ldots,x_{m})\pi_{T}=(x_{i}(j))_{(i,j)\in T}.
\]

For $x\in S$ an expression $x^{\ast}$ will mean $x^{\alpha}$ where $\alpha$ is
some fixed automorphism of $S$, depending on the context but not on $x$, and
$x^{-\ast}=(x^{\ast})^{-1}$.

We write%
\[
\lbrack\mathbf{x},\mathbf{g}]=([x_{1},g_{1}],\ldots,[x_{m},g_{m}]).
\]

\begin{lemma}
\label{L2}Let $x,y\in N$. Then $[x,g_{i}]=y$ if and only if%
\begin{align}
y(\delta_{il})  &  =x(\delta_{il})^{-1}x(\delta_{il})^{\ast}%
{\displaystyle\prod\limits_{j\in\Delta_{il}^{\prime}}}
y(j)^{-\ast}\label{subst}\\
x(j)  &  =x(j^{-})^{\ast}y(j)^{-1}\qquad(j\in\Delta_{il}^{\prime})
\label{rest}%
\end{align}
for $1\leq l\leq e_{i}$, where $j^{-}=j\cdot g_{i}^{-1}$.
\end{lemma}

\begin{proof}
Compare the $j$-components of $u=[x,g_{i}]$ and of $y$ as $j$ runs over a
given cycle $\Delta_{il}$. To simplify notation let's suppose that
$\Delta_{il}=(1,2,\ldots,s)$, with $\delta_{il}=1$. For $1\leq j\leq s$ we
have%
\[
u(j)=x(j)^{-1}x(j-1)^{\alpha_{j}}%
\]
(writing $x(0)=x(s)$) where $\alpha_{j}\in\mathrm{Aut}(S)$ depends on $j$ and
$g_{i}$. Using these to eliminate $x(2),\ldots,x(s)$ in turn we get%
\[
x(1)^{-1}x(1)^{\beta}=u(1)u(s)^{\alpha_{s}}u(s-1)^{\alpha_{s-1}\alpha_{s}%
}\ldots u(2)^{\alpha_{2}\ldots\alpha_{s}},
\]
where $\beta=\alpha_{1}\ldots\alpha_{s}$ is the automorphism induced by
$g_{i}^{s}$ on the first component of $S^{\Delta_{il}^{{}}}$. Thus
(\ref{subst}) and (\ref{rest}) hold with $u$ in place of $y$. The lemma
follows since these equations determine $y$ uniquely, given $x$.
\end{proof}

\bigskip

Put%
\begin{align*}
C  &  =\left\{  (i,\delta_{il})\mid1\leq i\leq m,~1\leq l\leq e_{i}\right\} \\
K  &  =\left\{  (i,j)\mid1\leq i\leq m,~j\in J_{i}^{^{\prime}}\right\} \\
K^{\prime}  &  =K\smallsetminus\left\{  (i(j),j)\mid j=2,\ldots,n\right\}  .
\end{align*}
Define $\Theta:N^{(m)}\rightarrow S^{C}\times S^{K}=S^{(mn)}$ by%
\[
\mathbf{x}\Theta=(\mathbf{x}\pi_{C},[\mathbf{x},\mathbf{g}]\pi_{K}).
\]
Lemma \ref{L2} shows that $\Theta$ is bijective.

Now define $\phi:N^{(m)}\rightarrow N=S^{(n)}$ by%
\begin{align*}
\mathbf{x}\phi &  =%
{\displaystyle\prod\limits_{i=1}^{m}}
[x_{i},g_{i}]\\
&  =(\mathbf{x}\phi_{1},\ldots,\mathbf{x}\phi_{n}).
\end{align*}
\bigskip Define $\Psi:N^{(m)}\rightarrow S^{C}\times S^{K^{\prime}}\times
S^{(n-1)}=S^{(mn)}$ by%
\[
\mathbf{x}\Psi=(\mathbf{x}\pi_{C},[\mathbf{x},\mathbf{g}]\pi_{K^{\prime}%
},(\mathbf{x}\phi_{2},\ldots,\mathbf{x}\phi_{n})).
\]

\begin{lemma}
The mapping $\Psi:S^{(mn)}\rightarrow S^{(mn)}$ is bijective.
\end{lemma}

\begin{proof}
Let $(\mathbf{u},\mathbf{v},z_{2},\ldots,z_{n})\in S^{C}\times S^{K^{\prime}%
}\times S^{(n-1)}$. We have to show that there exists a unique $\mathbf{x}\in
N^{(m)}$ such that $\mathbf{x}\pi_{C}=\mathbf{u},$ $[\mathbf{x},\mathbf{g}%
]\pi_{K^{\prime}}=\mathbf{v}$ and $\mathbf{x}\phi_{j}=z_{j}$ for
$j=2,\ldots,n$.

Since $\Theta$ is bijective, for each tuple $\eta=(\eta_{2},\ldots,\eta
_{n})\in S^{(n-1)}$ there exists a unique $\mathbf{x}\in N^{(m)}$ with%
\begin{align*}
\mathbf{x}\pi_{C}  &  =\mathbf{u},[\mathbf{x},\mathbf{g}]\pi_{K^{\prime}%
}=\mathbf{v,}\\
\lbrack x_{i(j)},g_{i(j)}](j)  &  =\eta_{j}\qquad(j=2,\ldots,n).
\end{align*}
Write $y_{i}=[x_{i},g_{i}]$. Then%
\[
\mathbf{x}\phi_{j}=y_{1}(j)y_{2}(j)\ldots y_{m}(j).
\]
If $(i,j)\in K^{\prime}$ then $y_{i}(j)$ is the $(i,j)$-component of
$[\mathbf{x},\mathbf{g}]\pi_{K^{\prime}}=\mathbf{v}$. If $(i,j)\in C$ then
$y_{i}(j)$ is determined by equation (\ref{subst}); this involves $x_{i}(j)$,
a component of $\mathbf{x}\pi_{C}=\mathbf{u}$, and further factors $y_{i}(r)$
where $r>j$.

If $(i,j)\notin C\cup K^{\prime}$ then $i=i(j)$ and $y_{i}(j)=\eta_{j}$. Now
we can solve the equations%
\begin{equation}
\eta_{j}=y_{i-1}(j)^{-1}\ldots y_{1}(j)^{-1}z_{j}y_{m}(j)^{-1}\ldots
y_{i+1}(j)^{-1} \label{eta}%
\end{equation}
successively for $j=n,$ $n-1,\ldots,2$, uniquely for $\eta$. The result follows.
\end{proof}

\bigskip

Observe now that $\mathbf{x}\phi=(z_{1},\ldots,z_{n})$ if and only if%
\begin{equation}
\mathbf{x}\phi_{1}=z_{1} \label{eq1}%
\end{equation}
and%
\begin{equation}
\mathbf{x}\Psi=(\mathbf{u},\mathbf{v},z_{2},\ldots,z_{n}) \label{eq2}%
\end{equation}
for some $(\mathbf{u},\mathbf{v})\in S^{C}\times S^{K^{\prime}}$.

Putting $y_{i}=[x_{i},g_{i}]$ as above we have%
\begin{equation}
\mathbf{x}\phi_{1}=y_{1}(1)y_{2}(1)\ldots y_{m}(1). \label{phi}%
\end{equation}
Now the following hold:

If $(i,j)\in C$ then $j=\delta_{il}$ for some $l\leq e_{i}$, and%
\begin{equation}
y_{i}(j)=x_{i}(j)^{-1}x_{i}(j)^{\ast}%
{\displaystyle\prod\limits_{k\in\Delta_{il}^{\prime}}}
y_{i}(k)^{-\ast}; \tag{S($i,j$)}\label{S(i,j)}%
\end{equation}
note that for each factor $y_{i}(k)$ occurring on the right we have
$(i,k)\notin C$ and $k>j$.

If $i=i(j)$ then%
\begin{equation}
y_{i}(j)^{-1}=y_{i+1}(j)\ldots y_{m}(j)z_{j}^{-1}y_{1}(j)\ldots y_{i-1}(j);
\tag{S($j$)}\label{S(j)}%
\end{equation}
note that for each factor $y_{r}(j)$ occurring on the right we have $r\neq
i(j)$.

Now we are going to successively transform the right-hand member of
(\ref{phi}) in the following manner: for some $(i,j)\in C$, substitute for the
factor $y_{i}(j)$ the expression on the right-hand side of (\ref{S(i,j)});
then use (S($k$)) to eliminate one of the newly introduced factors
$y_{i}(k)^{-1}$.

To analyse this process, for the time being we consider the $y_{i}(j),$
$y_{i}(j)^{-1},$ $x_{i}(j),~x_{i}(j)^{-1}$ and $z_{j}^{-1}$ as abstract
symbols (but allowing the automorphisms denoted by $\ast$ to distribute over
the factors in the usual way). If $U$ is a product of such symbols, possibly
decorated with $\ast$s, the \emph{support} $\sup(U)$ is the multiset of
symbols that occur in $U,$ with their multiplicities. For $(i,j)\in C$ let
$Y_{ij}$ denote the right-hand side of (\ref{S(i,j)}), and for $(i,j)\notin C$
set $Y_{ij}=y_{i}(j)$. For $j=2,\ldots,n$ put%
\[
Z_{j}=Y_{i(j)+1,j}\ldots Y_{mj}z_{j}^{-1}Y_{1j}\ldots Y_{i(j)-1,j}.
\]
Then%
\begin{align*}
\sup(Y_{ij})  &  =\{x_{i}(j)^{-1},x_{i}(j),y_{i}(k)^{-1}\mid k\in\Delta
_{i}^{\prime}(j)\}\text{ if }(i,j)\in C,\\
\sup(Y_{ij})  &  =\{y_{i}(j)\}\text{ if }(i,j)\notin C
\end{align*}
and%
\[
\sup(Z_{j})=\{z_{j}^{-1}\}\cup%
{\displaystyle\bigcup\limits_{i\neq i(j)}}
\sup(Y_{ij})
\]
(disjoint union).

Now set%
\[
U_{1}=%
{\displaystyle\prod\limits_{i=1}^{m}}
Y_{i1}\text{.}%
\]
Then%
\[
\sup(U_{1})=%
{\displaystyle\bigcup\limits_{i}}
\sup(Y_{i1})\ni y_{i(2)}(2)^{-1},
\]
because $(i,1)\in C$ for every $i$, and $2\in\Delta_{i(2)}^{\prime}(1)$. Let
$U_{2}$ be the expression obtained from $U_{1}$ on replacing $y_{i(2)}%
(2)^{-1}$ by $Z_{2}$. Then%
\begin{align*}
\sup(U_{2})  &  =\sup(U_{1})\cup\sup(Z_{2})\smallsetminus\{y_{i(2)}%
(2)^{-1}\}\\
&  =%
{\displaystyle\bigcup\limits_{i}}
\sup(Y_{i1})\cup\{z_{2}^{-1}\}\cup%
{\displaystyle\bigcup\limits_{i\neq i(2)}}
\sup(Y_{i2})\smallsetminus\{y_{i(2)}(2)^{-1}\}.
\end{align*}

Iterating this process, suppose that after $j-1<n-1$ steps we obtain $U_{j}$,
where $\sup(U_{j})$ contains
\begin{equation}%
{\displaystyle\bigcup\limits_{r=1}^{j}}
\left(
{\displaystyle\bigcup\limits_{i\neq i(r)}}
\sup(Y_{ir})\smallsetminus\{y_{i(r)}(r)^{-1}\}\right)  . \label{supp}%
\end{equation}
Say $\widehat{j+1}=r$, so $r\leq j$ and $j+1\in\Delta_{i(j+1)}^{\prime}(r)$.
Then $(i(j+1),r)\in C$ and (if $r>1$) $i(j+1)\neq i(r)$, so $y_{i(j+1)}%
(j+1)^{-1}\in\sup(Y_{i(j+1),r})\subseteq\sup(U_{j})$. Now replace
$y_{i(j+1)}(j+1)^{-1}$ in $U_{j}$ by $Z_{j+1}$ to obtain $U_{j+1}$. Then the
analogue of (\ref{supp}) holds with $j+1$ for $j$.

After $n-1$ such steps we obtain an expression $U=U_{n}$ with%
\[
\sup(U)=\mathcal{X}\cup\mathcal{Y}\cup\mathcal{Z}%
\]
where%
\begin{align*}
\mathcal{X}  &  =\{x_{i}(j),x_{i}(j)^{-1}\mid(i,j)\in C\},~~\mathcal{Y}%
=\{y_{i}(j),y_{i}(j)^{-1}\mid(i,j)\in K^{\prime}\},\\
\mathcal{Z}  &  =\{z_{2}^{-1},\ldots,z_{n}^{-1}\}.
\end{align*}

To any formal product $V$ of factors $x_{i}(j)^{\pm\ast},$ $y_{i}(j)^{\pm\ast
},$ $z_{j}^{-\ast}$ we assign a numerical sequence $\tau(V) $ as follows:
reading $V$ from left to right, ignore all factors $x_{i}(j)^{\pm\ast}$ and
$z_{j}^{-\ast}$; to each factor $y_{i}(j)^{\ast}$ assign the label $i$, and to
each maximal product of consecutive terms of the form $y_{i}(k)^{-\ast}$
(fixed $i,$ varying $k$) assign the label $i$.

$\medskip$

\textbf{Claim 1}: For each $j=1,\ldots,n$, $\tau(U_{j})$ is a subsequence of%
\[
\mathcal{S}(j)=(1,\ldots,m,1,\ldots,m,\ldots,1,\ldots,m)
\]
where $1,\ldots,m$ is repeated $j$ times.

$\medskip$

\emph{Proof.} This is clear for $j=1$. Let $j\geq1$ and suppose inductively
that $\tau(U_{j})$ is a subsequence of $\mathcal{S}(j)$. Put $i=i(j+1)$; then
$y_{i}(j+1)^{-\ast}$ is a factor in $U_{j}$, and we obtained $U_{j+1}$ by
replacing it with $Z_{j+1}^{\ast}$.

Thus%

\[
\tau(U_{j})=(I_{1},P,i,Q,I_{2})
\]
where $(P,i,Q)$ is a subsequence of $(1,\ldots,m)$, $I_{1}$ is a subsequence
of $\mathcal{S}(p)$ and $I_{2}$ is a subsequence of $\mathcal{S}(q)$ and
$p+1+q=j$ (here $p$ or $q$ could be $0$, with $\mathcal{S}(0)=\varnothing$);
the displayed $i$ is due to $y_{i}(j+1)^{-\ast}$. Substituting $Z_{j+1}^{\ast
}$ for $y_{i}(j+1)^{-\ast}$ has the effect of replacing $i$ by $(\underline
{i},i+1,,\ldots,m,1,\ldots,i-1,\underline{i})$, where the underlined $i$s may
or may not be present (depending on whether $y_{i}(j+1)^{-\ast}$ appears in
the middle or at either end of a product of consecutive terms of the form
$y_{i}(k)^{-\ast}$). In any case,%
\[
(P,\underline{i},i+1,,\ldots,m,1,\ldots,i-1,\underline{i},Q)
\]
is a subsequence of $\mathcal{S}(2)$, and so $\tau(U_{j+1})$ is a subsequence
of
\[
(\mathcal{S}(p),\mathcal{S}(2),\mathcal{S}(q))=\mathcal{S}(p+2+q)=\mathcal{S}%
(j+1).
\]

$\medskip$

\textbf{Claim 2: }There exist $2D$ distinct elements $\xi_{1},\eta_{1}%
,\ldots,\xi_{D},\eta_{D}$ of $\mathcal{Y}$ such that the following holds.
There exist $R$, $A_{i},~B_{i},C_{i},~D_{i}$ ($i=1,\ldots,D$), each of which
is a product of factors $t^{\ast}$ with $t\in\mathcal{X}\cup\mathcal{Z}%
\cup\mathcal{Y}\smallsetminus\{\xi_{1},\eta_{1},\ldots,\xi_{D},\eta_{D}\}$,
such that%
\begin{equation}
U_{n}\simeq%
{\displaystyle\prod\limits_{i=1}^{D}}
(A_{i}\xi_{i}B_{i})^{-\ast}(C_{i}\eta_{i}D_{i})^{-\ast}(A_{i}\xi_{i}%
B_{i})^{\ast}(C_{i}\eta_{i}D_{i})^{\ast}\cdot R, \label{U-T}%
\end{equation}
meaning that the two sides represent the same element in the free group on all
the occurring symbols $t^{\ast}$, $t\in\mathcal{X}\cup\mathcal{Y}%
\cup\mathcal{Z}$.

$\medskip$

\emph{Proof.} This follows from hypothesis (\ref{ineq}) and Claim 1 by (the
proof of) \cite{NS}, Prop. 8.4.

$\medskip$

To complete the proof of Proposition \ref{thm} we need one further lemma:

\begin{lemma}
\label{subsets}Let $K^{\prime\prime}\subseteq K^{\prime}$ be a set of size
$\left\vert K^{\prime}\right\vert -2D\geq0$. Suppose that $W\subseteq N^{(m)}$
satisfies $\left\vert W\right\vert \geq(1-\varepsilon/q)\left\vert
N^{(m)}\right\vert $. Let $\mathcal{P}$ be the set of elements $\mathbf{z}\in
S^{(n-1)}$ for which there exist $\mathbf{u}\in S^{C}$, $\mathbf{v}\in
S^{K^{\prime\prime}}$ such that%
\[
\left\vert \left\{  w\in W\mid w\Psi\pi_{C\cup K^{\prime\prime}\cup\lbrack
n-1]}=(\mathbf{u},\mathbf{v},\mathbf{z})\right\}  \right\vert \geq
(1-\varepsilon)\left\vert S\right\vert ^{2D}.
\]
Then $\left\vert \mathcal{P}\right\vert \geq(1-\frac{1}{q})\left\vert
S\right\vert ^{n-1}$.
\end{lemma}

\begin{proof}
Put $\sigma=\left\vert S\right\vert $. Recall that $\left\vert C\right\vert
+\left\vert K^{\prime\prime}\right\vert =mn-(n-1)-2D$, and that $\Psi$ is
bijective. Suppose that $\left\vert \mathcal{P}\right\vert =\lambda
\sigma^{n-1}$. Then%
\[
\left\vert W\right\vert \leq\lambda\sigma^{n-1}\sigma^{mn-n+1-2D}\cdot
\sigma^{2D}+(1-\lambda)\sigma^{n-1}\sigma^{mn-n+1-2D}\cdot(1-\varepsilon
)\sigma^{2D}.
\]
It follows that%
\[
1-\varepsilon/q\leq\lambda+(1-\lambda)(1-\varepsilon),
\]
which implies that $\lambda\geq1-\frac{1}{q}$.
\end{proof}

\bigskip

Now, for some subset $L\subseteq K^{\prime}$ of size $2D$ we have%
\[
\{\xi_{1},\eta_{1},\ldots,\xi_{D},\eta_{D}\}=\{y_{i}(j)\mid(i,j)\in L\}.
\]
Put $K^{\prime\prime}=K^{\prime}\smallsetminus L$. Recall that $W\subseteq
N^{(m)}$ satisfies $\left\vert W\right\vert \geq(1-\varepsilon/6)\left\vert
N\right\vert ^{m}$. Let $\mathcal{P}\subseteq S^{(n-1)}$ be the set defined in
Lemma \ref{subsets}; thus $\left\vert \mathcal{P}\right\vert \geq\frac{5}%
{6}\left\vert S\right\vert ^{n-1}$.

By definition, for each $\mathbf{z}\in\mathcal{P}$ there exist $\mathbf{u}%
_{\mathbf{z}}\in S^{C}$, $\mathbf{v}_{\mathbf{z}}\in S^{K^{\prime\prime}}$ and
$W_{\mathbf{z}}\subseteq W$ with $\left\vert W_{\mathbf{z}}\right\vert
\geq(1-\varepsilon)\left\vert S\right\vert ^{2D}$ such that%
\[
W_{\mathbf{z}}\Psi\pi_{C\cup K^{\prime\prime}\cup\lbrack n-1]}=\{(\mathbf{u}%
_{\mathbf{z}},\mathbf{v}_{\mathbf{z}},\mathbf{z})\}.
\]
As $\Psi$ is a bijection this implies that $\left\vert W_{\mathbf{z}}\Psi
\pi_{L}\right\vert =\left\vert W_{\mathbf{z}}\right\vert \geq(1-\varepsilon
)\left\vert S\right\vert ^{2D}$.

Now let $\mathbf{x}\in W_{\mathbf{z}}$. Then%
\begin{align*}
\mathbf{x}\phi &  =(\mathbf{x}\phi_{1},\ldots,\mathbf{x}\phi_{n})\\
&  =(\mathbf{x}\phi_{1},\mathbf{z})
\end{align*}
and%
\[
\mathbf{x}\phi_{1}=U(\mathcal{X},\mathcal{Y},\mathcal{Z}).
\]
In the expression (\ref{U-T}) for $U(\mathcal{X},\mathcal{Y},\mathcal{Z})$,
each of the factors $R$, $A_{i},~B_{i},C_{i},~D_{i}$ is a product of terms
$t^{\pm\ast}$ where $t$ is a component of $\mathbf{x}\Psi\pi_{C\cup
K^{\prime\prime}\cup\lbrack n-1]}=(\mathbf{u}_{\mathbf{z}},\mathbf{v}%
_{\mathbf{z}},\mathbf{z})$. Therefore%
\begin{align}
\mathbf{x}\phi_{1}  &  =%
{\displaystyle\prod\limits_{i=1}^{D}}
(a_{i}\xi_{i}b_{i})^{-\alpha_{i}}(c_{i}\eta_{i}d_{i})^{-\beta_{i}}(a_{i}%
\xi_{i}b_{i})^{\gamma_{i}}(c_{i}\eta_{i}d_{i})^{\delta_{i}}\cdot
r\label{T-eq}\\
&  =%
{\displaystyle\prod\limits_{i=1}^{D}}
T_{\sigma_{i},\tau_{i}}(\overline{\xi_{i}},\overline{\eta_{i}})\cdot
r\nonumber
\end{align}
where $a_{i},~b_{i},~c_{i},~d_{i}$ and $r$ depend only on $\mathbf{z}$ and
$\alpha_{i},~\beta_{i},~\gamma_{i},~\delta_{i}$ are certain automorphisms of
$S$, independent of everything else, and%
\begin{align*}
\overline{\xi_{i}}  &  =(a_{i}\xi_{i}b_{i})^{-\alpha_{i}},~\overline{\eta_{i}%
}=(c_{i}\eta_{i}d_{i})^{-\beta_{i}}\\
\sigma_{i}  &  =\alpha_{i}^{-1}\gamma_{i},~\tau_{i}=\beta_{i}^{-1}\delta_{i}.
\end{align*}

Now $(\xi_{1},\eta_{1},\ldots,\xi_{D},\eta_{D})=\mathbf{x}\Psi\pi_{L}$ takes
$\left\vert W_{\mathbf{z}}\Psi\pi_{L}\right\vert \geq(1-\varepsilon)\left\vert
S\right\vert ^{2D}$ values as $\mathbf{x}$ ranges over $W_{\mathbf{z}}$; hence
so does the tuple $(\overline{\xi_{1}},\overline{\eta_{1}},\ldots
,\overline{\xi_{D}},\overline{\eta_{D}})$. According to Theorem \ref{T1} this
implies that $%
{\displaystyle\prod\limits_{i=1}^{D}}
T_{\sigma_{i},\tau_{i}}(\overline{\xi_{i}},\overline{\eta_{i}})$ takes at
least $\lambda\left\vert S\right\vert $ values, where $\lambda=l(S)^{-3/5}$ if
$l(S)\geq3$, $\lambda=1$ if $l(S)=2$; therefore so does $\mathbf{x}\phi_{1}$,
by (\ref{T-eq}). It follows that%
\begin{align*}
\left\vert W\phi\right\vert \geq%
{\displaystyle\sum\limits_{\mathbf{z}\in\mathcal{P}}}
\left\vert W_{\mathbf{z}}\phi_{1}\right\vert \geq &  \frac{5}{6}\left\vert
S\right\vert ^{n-1}\cdot\lambda\left\vert S\right\vert \\
&  \geq l(S)^{-4/5}\left\vert N\right\vert
\end{align*}
since $\left\vert S\right\vert ^{n}=\left\vert N\right\vert $ and $(\frac
{6}{5})^{5}<3$. This completes the proof of (1).$\medskip$

To prove (2), we replace $D$ by $D_{1}=5D$ in the above. Let $(\mathbf{u}%
,\mathbf{v}_{0},\mathbf{z})$ be an arbitrary element of $S^{C}\times
S^{K^{\prime\prime}}\times S^{(n-1)}$, and let $z\in S$. For each
$\mathbf{\xi}=(\xi_{1},\eta_{1},\ldots,\xi_{D_{1}},\eta_{D_{1}})\in S^{L}$
there exists $\mathbf{x}\in N^{(m)}$ such that%
\begin{align*}
\mathbf{x}\Psi\pi_{C\cup K^{\prime\prime}\cup\lbrack n-1]}  &  =(\mathbf{u}%
,\mathbf{v}_{0},\mathbf{z})\\
\mathbf{x}\Psi\pi_{L}  &  =\mathbf{\xi.}%
\end{align*}
Take $\mathbf{u}_{z}=\mathbf{u}$ and $\mathbf{v}_{z}=\mathbf{v}_{0}$ in the
above discussion. Then $\mathbf{x}\phi_{1}$ is given by (\ref{T-eq}). Now
Corollary \ref{NS2Thm1.1} says that%
\[
S=%
{\displaystyle\prod\limits_{i=1}^{D_{1}}}
T_{\sigma_{i},\tau_{i}}(S,S).
\]
We may therefore choose $\mathbf{\xi}$ so that%
\[%
{\displaystyle\prod\limits_{i=1}^{D_{1}}}
T_{\sigma_{i},\tau_{i}}(\overline{\xi_{i}},\overline{\eta_{i}})=zr^{-1},
\]
and so ensure that $\mathbf{x}\phi=(z,\mathbf{z})$. It follows that%
\begin{align*}
\left\vert (z,\mathbf{z})\phi^{-1}\right\vert \geq\left\vert S^{C}\times
S^{K^{\prime\prime}}\right\vert  &  =\left\vert S\right\vert ^{mn-(n-1)-2D_{1}%
}\\
&  >\left\vert N\right\vert ^{-2D_{1}/n}\left\vert N\right\vert ^{m-1}.
\end{align*}

\section{Applications\label{profsec}}

\subsection{Subgroups of finite index\label{Finite-index-sec}}

Here we re-prove the main result of \cite{NS}:

\begin{theorem}
\label{serre}If $G$ is a finitely generated profinite group then every
subgroup of finite index in $G$ is open.
\end{theorem}

\begin{proof}
Let $H$ be a subgroup of finite index in $G$. Then $H$ contains a normal
subgroup $N$ of finite index in $G$. The closure $M=\overline{N}$ of $N$ is
open in $G$, so $M$ is again a finitely generated profinite group. If $N=M$
then $N$ is open and so $H$ is open.

Suppose that $N<M$. Then Corollary \ref{normal_mod_G_0} shows that at least
one of%
\begin{align*}
NM^{\prime}  &  <M,\\
NM_{0}  &  <M
\end{align*}
holds. Put $q=\left\vert M/N\right\vert $, so we have $M^{q}\leq N$. Note that
$M^{\prime}$ is closed, by Theorem \ref{ThB-prof}.

Now $M/M^{\prime}M^{q}$ is a finitely generated abelian profinite group of
finite exponent, so it is finite, hence discrete; as $NM^{\prime}/M^{\prime
}M^{q}$ is a dense subgroup it follows that $NM^{\prime}=M$.

To derive a contradiction it remains to show that $NM_{0}=M$; to this end we
may as well replace $G$ by $G/M_{0}$, and so assume that $M_{0}=1$. Then $M$
has a closed semisimple normal subgroup $T$ such that $M/T$ is soluble. It
follows from the preceding paragraph that $NT=M$.

A theorem of Martinez-Zelmanov \cite{MZ} and Saxl-Wilson \cite{SW} shows that
$T^{q}$ is closed in $T$ (because the word $x^{q}$ has bounded width in all
finite simple groups). As $T^{q}\leq N$ we may factor it out and assume
further that $T^{q}=1$. Now the definition of $M_{0}$ ensures that in fact $T$
is a product of finite simple groups each of which is normal in $M$; and these
simple groups have bounded orders \cite{J}. Therefore $M/\mathrm{C}_{M}(T)$ is
finite, and so $T$ is finite. Hence $N\cap T$ is closed. Thus%
\[
T=[T,M]=[T,\overline{N}]\leq\overline{[T,N]}\leq T\cap N
\]
whence $M=NT=N$, as required.
\end{proof}

\subsection{Finite to profinite\label{routine}}

Here we recall some standard compactness arguments. We refer to subsection
\ref{introfgpg} for the statements of the following theorems, concerning a
finitely generated profinite group $G$ with closed normal subgroups $K$ and
$H$.

\bigskip

\noindent\textbf{Proof of Theorem \ref{prof-gen}. }Write $I=\{(i,j)\mid1\leq
i\leq r,~1\leq j\leq f_{0}\}$. For each open normal subgroup $N$ of $G$ let%
\[
X(N)=\left\{  \mathbf{x}=(x_{ij})\in K^{(rf_{0})}\mid G=N\left\langle
y_{i}^{x_{ij}}\mid(i,j)\in I\right\rangle \right\}  .
\]
Theorem \ref{ThmA}, applied to the finite group $G/N$, shows that each set
$X(N)$ is non-empty. Also $X(N)$ is closed in $K^{(rf_{0})}$, being a union of
cosets of $(N\cap K)^{(rf_{0})}$, and if $N>M$ then $X(N)\supseteq X(M)$. It
follows by compactness that $\bigcap_{N}X(N)$ is non-empty, taking the
intersection over all open normal subgroups $N$ of $G$. Let $\mathbf{x}$ be in
this intersection. Then%
\[
G=\bigcap_{N}N\left\langle y_{i}^{x_{ij}}\mid(i,j)\in I\right\rangle
=\overline{\left\langle y_{i}^{x_{ij}}\mid(i,j)\in I\right\rangle }.
\]

$\medskip$

\noindent\textbf{Proof of Theorems \ref{ThB-prof} and \ref{ThmC-prof}.} Let
$R$ denote the right-hand side of equation (\ref{rhsB}) or equation
(\ref{rhsC}) (see Subsection \ref{introfgpg}). Then $R$ is a closed subset of
$G$. Now let $N$ be an open normal subgroup of $G$. Then Theorem \ref{ThmB},
respectively Theorem \ref{ThmC}, applied to the finite group $G/N$ shows that
$[H,G]N=RN$. As $R$ is closed, intersecting over all open normal subgroups $N$
of $G$ we get%
\[
R=\bigcap_{N}RN\supseteq\lbrack H,G],
\]
and the results follow since $R\subseteq\lbrack H,G]$.

\subsection{Verbal subgroups\label{vbl}}

Here we show how the main results of \cite{NSP} may be quickly derived from
Theorems \ref{ThmA} and \ref{ThmB}.

Let $w$ be a group word in $k$ variables, and $G$ a group. The corresponding
\emph{verbal subgroup} is $w(G)=\left\langle G_{w}\right\rangle $, where%
\[
G_{w}=\left\{  w(\mathbf{g})^{\pm1}\mid\mathbf{g}\in G^{(k)}\right\}
\]
denotes the (symmetrized) set of $w$-values in $G$. We say that $w$ has
\emph{width} $m$ in $G$ if%
\[
w(G)=G_{w}^{\ast m}\text{;}%
\]
if this holds for some finite $m$ we denote the least such $m$ by $m_{w}(G)$,
and say that $w$ has finite width in $G$.

The following elementary result is Proposition 2.1.2 of \cite{S2}:

\begin{lemma}
\label{AF}If $G$ is abelian-by-finite then $m_{w}(G)$ is finite.
\end{lemma}

Now define%
\[
\beta(w,G)=\left\vert G:w(G)\right\vert .
\]
Let us call the word $w$ $d$-\emph{bounded} if there exists $\beta_{w}%
=\beta_{w}(d)\in\mathbb{N}$ such that $\beta(w,G)\leq\beta_{w}(d)$ whenever
$G$ is a $d$-generator finite group. The positive solution of the Restricted
Burnside Problem \cite{Z} asserts that the word $w=x^{q}$ is $d$-bounded, for
all natural numbers $d$ and $q$. This implies that every \emph{non-commutator}
word $w$ is $d$-bounded, since for any group $G$ we have $w(G)\geq G^{q}$
where $q=\left\vert \mathbb{Z}/w(\mathbb{Z})\right\vert $. (In fact it is easy
to see that, conversely, every $d$-bounded word is a non-commutator word.)

\begin{proposition}
\label{simplew}Suppose that $w$ is $d$-bounded (for some $d\geq1$). Then there
exists $m_{0}=m_{0}(w)$ such that $w$ has with $m_{0}$ in every finite
semisimple group.
\end{proposition}

\begin{proof}
It suffices to prove this for a simple group $G$. Take $q=\left\vert
\mathbb{Z}/w(\mathbb{Z})\right\vert $. We consider three cases.

(i) Where $w(G)=1$. Then $m_{w}(G)=1$.

(ii) Where $w(G)\neq1$ but $G^{q}=1$. There are only finitely many
possibilities for $G$ in this case \cite{J}. Since $G_{w}$ generates $G$ it
follows that $G=G_{w}^{\ast n}$ where $n=n(q)$ is the maximal order of any
such group $G$.

(iii) Where $G^{q}\neq1$. In this case, the theorem of Martinez-Zelmanov
\cite{MZ} and Saxl-Wilson \cite{SW} shows that every element of $G$ is a
product of $h(q)$ $q$th powers; as each $q$th power is a $w$-value it follows
that $m_{w}(G)\leq h(q)$.
\end{proof}

\bigskip

The main result is now

\begin{theorem}
\label{resthm}Let $w$ be a $d$-bounded word and $G$ a finite $d$-generator
group. Then $m_{w}(G)\leq f(w,d)$ where $f(w,d)$ depends only on $w$ and $d$.
\end{theorem}

\begin{proof}
Let $\mathcal{M}$ denote the (finite) set of (non-abelian) simple groups $M$
such that $w(M)=1$. For $n\in\mathbb{N}$ set%
\[
\mu(n)=\left\vert F_{n}:K(w)\right\vert
\]
where $F_{n}$ is free of rank $n$ and $K(w)$ is the intersection of all
$\ker\theta$ where $\theta$ ranges over homomorphisms $F_{n}\rightarrow
\mathrm{Aut}(M)$ with $M\in\mathcal{M}$.

Put $W=w(G)$ and set $\beta=\left\vert G:W\right\vert $, so that $\beta
\leq\beta_{w}(d)$. By Schreier's formula we then have%
\[
\mathrm{d}(W)\leq d_{0}:=1+\beta(d-1).
\]
Set%
\[
H=\bigcap\mathrm{C}_{W}(M)
\]
where $M$ ranges over all chief factors of $W$ that belong to $\mathcal{M}$.
Then $W/H$ is an image of $F_{d_{0}}/K(w)$, so $\left\vert W:H\right\vert
\leq\mu(d_{0})$ and $\left\vert G:H\right\vert \leq\beta_{1}:=\beta\mu(d_{0}%
)$. It follows that $\mathrm{d}(H)\leq d_{1}:=1+\beta_{1}(d-1)$.

Now let $K$ be the intersection of the kernels of all homomorphisms
$F_{d}\rightarrow\mathrm{Sym}(\beta_{1})$. Lemma \ref{AF} shows that $w$ has
finite width $m_{1}=m_{1}(d,\beta)$ in the group $F_{d}/K^{\prime}$. As
$G/H^{\prime}$ is an image of $F_{d}/K^{\prime}$ it follows that
$m_{w}(G/H^{\prime})\leq m_{1}$, so%
\begin{equation}
W=H^{\prime}\cdot G_{w}^{\ast m_{1}}. \label{moduloW'}%
\end{equation}
Put $X_{1}=G_{w}^{\ast m_{1}}$, so $W=H^{\prime}X_{1}$. Then $H=H^{\prime
}X_{2}$ where $X_{2}=H\cap X_{1}$. There exist $Y_{0}\subseteq X_{1}$ and
$Y_{1}\subseteq X_{2}$ such that%
\begin{align*}
W  &  =H^{\prime}\left\langle Y_{0}\right\rangle ,~~\left\vert Y_{0}%
\right\vert \leq d_{0},\\
H  &  =H^{\prime}\left\langle Y_{1}\right\rangle ,~~\left\vert Y_{1}%
\right\vert \leq d_{1}.
\end{align*}

Now recall Theorem \ref{ThmA}: this associates to $W$ a characteristic
subgroup $W_{0}$, contained in $H$, such that $W^{(3)}W_{0}/W_{0}$ is
semisimple. Put $D=W^{(3)}W_{0}\cap H$. Applying Theorem \ref{ThmA} to the
soluble group $H/D$, we find a set $Y_{2}\subseteq Y_{1}^{H}$ such that%
\[
H=D\left\langle Y_{2}\right\rangle ,~~\left\vert Y_{2}\right\vert \leq
h_{1}=f_{0}(d_{1},d_{1}).
\]
By the definition of $H$, the semisimple group $D/W_{0}$ is a product of
simple groups $S$ such that $w(S)=S$. Hence by Proposition \ref{simplew} we
have $D=W_{0}D_{w}^{\ast m_{0}}$. Using Lemma \ref{Gasch}, we may therefore
lift $Y_{2}$ to a set $Y_{3}\subseteq D_{w}^{\ast m_{0}}Y_{2}$ so that%
\[
H=W_{0}\left\langle Y_{3}\right\rangle ,~~\left\vert Y_{3}\right\vert \leq
h_{1}.
\]
Now put $Y_{4}=Y_{3}\cup Y_{0}$. Then $W=W^{\prime}\left\langle Y_{4}%
\right\rangle =W_{0}\left\langle Y_{4}\right\rangle $ and $\left\vert
Y_{4}\right\vert \leq h_{1}+d_{0}$; and $Y_{4}\subseteq G_{w}^{\ast
(m_{1}+m_{0})}$.

A further application of Theorem \ref{ThmA} now provides a set $Y$ such that
$W=\left\langle Y\right\rangle $, $\left\vert Y\right\vert \leq h_{2}%
:=(h_{1}+d_{0})f_{0}(h_{1}+d_{0},d_{0})$, and each element of $Y$ is conjugate
to one of $Y_{4}$.

Put $\widetilde{Y}=Y\cup Y^{-1}$. It then follows by Theorem \ref{ThmB} that%
\[
W^{\prime}=\left(
{\displaystyle\prod\limits_{y\in\widetilde{Y}}}
[W,y]\right)  ^{\ast f_{1}(2h_{2},d_{0})}\subseteq G_{w}^{\ast m_{2}},
\]
where $m_{2}=4h_{2}f_{1}(2h_{2},d_{0})(m_{1}+m_{0})$. With (\ref{moduloW'})
this shows that $w$ has width $f(w,d):=m_{2}+m_{1}$ in $G$.
\end{proof}

\bigskip

Now let $G$ be a $d$-generator profinite group. Suppose that $m_{w}(Q)\leq
m<\infty$ for every continuous finite quotient $Q$ of $G$. Then%
\[
w(G)N=G_{w}^{\ast m}\cdot N
\]
for every open normal subgroup $N$ of $G$. But $G_{w}^{\ast m}$ is a closed
subset of $G$, because $w:G^{(k)}\rightarrow G$ is continuous; therefore%
\[
w(G)\subseteq%
{\displaystyle\bigcap\limits_{N}}
G_{w}^{\ast m}\cdot N=G_{w}^{\ast m},
\]
so $w(G)=G_{w}^{\ast m}$ is a closed subgroup of $G$.

If also $w$ is $d$-bounded, then $\beta(w,Q)\leq\beta_{w}(d)=\beta$, say, for
for every continuous finite quotient $Q$ of $G$. Thus%
\[
\left\vert G:w(G)N\right\vert \leq\beta
\]
for each open normal subgroup $N$ of $G$. Choosing an open normal subgroup $M$
for which $\left\vert G:w(G)M\right\vert $ is maximal we infer (given that
$w(G)$ is closed) that%
\[
w(G)=%
{\displaystyle\bigcap\limits_{N}}
w(G)N=w(G)M.
\]
Thus $w(G)$ is an open subgroup of $G$.

Theorem \ref{resthm} now gives

\begin{theorem}
\label{maindresthm}Let $G$ be a $d$-generator profinite group and $w$ a
$d$-bounded word. Then the verbal subgroup $w(G)$ is open in $G$.
\end{theorem}

This shows that $w(G)$ \emph{is open whenever} $G$ \emph{is a finitely
generated profinite group and} $w$ \emph{is any non-commutator word}, the main
result of \cite{NSP}. The point of our clumsier formulation is that the
theorem as stated is independent of the Restricted Burnside Problem.

\subsection{Verbal subgroups in compact groups}

Throughout this subsection, we suppose that $G$ is a compact group and that
\emph{the profinite quotient} $G/G^{0}$ \emph{is finitely generated}. Several
of the preceding results can be generalized.

\begin{corollary}
If $w$ is a non-commutator word then $w(G)$ is open in $G$.
\end{corollary}

\begin{proof}
Set $q=\left\vert \mathbb{Z}/w(\mathbb{Z})\right\vert $. Every element of
$G^{0}$ is a $q$th power (\cite{HM1}, Theorem 9.35), so $G^{0}\leq w(G)$, and
so $w(G/G^{0})=w(G)/G^{0}$. The result follows by the remark following Theorem
\ref{maindresthm}.
\end{proof}

The next corollary follows likewise from Theorem \ref{serre}:

\begin{corollary}
\label{cpctserre}Every subgroup of finite index in $G$ is open.
\end{corollary}

\begin{lemma}
\label{comm-lemma}Suppose that $A\leq\mathrm{Z}(G^{0})$ is closed and normal
in $G$. Then $[A,G]$ is closed in $G$.
\end{lemma}

\begin{proof}
Let us consider $A$ as an additively-written $\Gamma=G/G^{0}$-module. By
hypothesis, $\Gamma$ has a dense finitely generated (abstract) subgroup
$X=\left\langle x_{1},\ldots,x_{d}\right\rangle $. Now%
\[
\lbrack A,X]=A(x_{1}-1)+\cdots+A(x_{d}-1);
\]
this is an $X$-submodule of $A$, and it is closed in $A$ because $A$ is
compact. Therefore $[A,X]$ is a $\Gamma$-submodule, because $X$ is dense in
$\Gamma$. Therefore $C:=\mathrm{C}_{\Gamma}(A/[A,X])$ is closed in $\Gamma$,
and as $X\leq C$ it follows that $C=\Gamma$. Hence $[A,G]=[A,\Gamma]=[A,X]$ is closed.
\end{proof}

\begin{corollary}
\label{derived-cpct}The derived group $G^{\prime}$ is closed in $G$.
\end{corollary}

\begin{proof}
Let $P=(G^{0})^{\prime}$ denote the derived group of $G^{0}$. Then $P$ is
closed, by \cite{HM1}, Theorem 9.2. So replacing $G$ by $G/P$ we may suppose
that \thinspace$G^{0}$ is abelian. Then $[G^{0},G]$ is closed by the preceding
Lemma, so we may factor it out and reduce to the case where $G^{0}$ is central
in $G$. Now according to \cite{HM1}, theorem 9.41, we have $G=G^{0}D$ for some
closed profinite subgroup $D$. Since $D/(D\cap G^{0})\cong G/G^{0}$ is
finitely generated, $D=(D\cap G^{0})H$ for some finitely generated profinite
group $H$. Then $G^{\prime}=H^{\prime}$ is closed by the remark following
Theorem \ref{ThB-prof}.
\end{proof}

$\medskip$

\noindent\emph{Remark. }More generally, we can show that $[H,G]$ \emph{is
closed for every closed normal subgroup} $H$ \emph{of} $G$. When $G^{0}=1$
this follows from Theorem \ref{ThB-prof}, and when $G=G^{0}$ it follows from
the known structure of connected compact groups. The general case depends on a
modified form of the `Key Theorem', Theorem \ref{KT}, in which $d=\mathrm{d}%
(G)$ is replaced by $\mathrm{d}(G/\mathrm{C}_{G}(H))$; the proof will appear elsewhere.

\subsection{Quotients of semisimple compact groups\label{simplequot}}

In this subsection we consider a topological group
\begin{equation}
G=\prod_{i\in I}S_{i}, \label{semis-prod-def}%
\end{equation}
where $I$ is an index set, and \emph{either: }

\begin{description}
\item[(a)] each $S_{i}$ is a nonabelian finite simple group, and for each $n$
the set%
\[
I(n)=\left\{  i\in I\mid\left\vert S_{i}\right\vert \leq n\right\}
\]
is finite; \emph{or}

\item[(b)] each $S_{i}$ is a compact connected simple Lie group
\end{description}

\noindent(here, by a `simple Lie group' we mean the analogue of a
\emph{quasisimple} finite group: i.e. it may have a non-trivial centre, but is
simple modulo the centre and perfect).$\medskip$

\noindent\emph{Remarks:} \textbf{i.} (a) holds in particular when $G$ is a
semisimple finitely generated profinite group.

\textbf{ii.} Hofmann and Morris \cite{HM1} call a compact connected group $G$
`semisimple' if it is perfect, i.e. if $G=\overline{G^{\prime}}$, equivalently
if $G=G^{\prime}$ (\emph{loc. cit.} Theorem 9.2). However, this holds if and
only if $G=\widetilde{G}/C$ where $\widetilde{G}$ is a product of compact
connected simply-connected simple Lie groups and $C$ is a totally disconnected
normal subgroup (\emph{loc. cit.} Theorem 9.19); thus any quotient of $G$ is
also a quotient of a group of the form (\ref{semis-prod-def}).

\begin{theorem}
\label{unctble}Let $Q$ be an infinite quotient of (the underlying abstract
group) $G$. Then $\left\vert Q\right\vert \geq2^{\aleph_{{0}}}$.
\end{theorem}

This depends on the following technical device:

\begin{proposition}
\label{lamda-function}Let $L$ be \emph{(a)} a nonabelian finite simple group
or \emph{(b)} a compact connected simple Lie group. In Case (b), let $T$ be a
maximal torus of $L$, in Case (a) let $T=L$. There is a function
$\lambda=\lambda_{L}:T\rightarrow\lbrack0,1]$ with the following properties:

\begin{description}
\item[(i)] $\lambda(s)=0\Longleftrightarrow s\in\mathrm{Z}(L)$;

\item[(ii)] $\lambda(s^{-1})=\lambda(s^{t})=\lambda(s)$ and $\lambda
(st)\leq\lambda(s)+\lambda(t)$ for all $s,~t\in T$;

\item[(iii)] if $t\in T$ and $\lambda(t)\geq\varepsilon>0$ then%
\[
L=(t^{L}\cup t^{-L})^{\ast f(\varepsilon)}%
\]
where $f(\varepsilon)\in\mathbb{N}$ depends only on $\varepsilon$;

\item[(iv)] in \emph{Case (a)}, $1\neq s\in L$ implies $\lambda_{L}%
(s)\geq\varepsilon(r)$ where $\varepsilon(r)>0$ depends only on
$r=\mathrm{rank}(L)$;

\item[(va)] in \emph{Case (a)}: given $\beta,\varepsilon\in(0,1)$, there
exists $s\in L$ with%
\[
\left\vert \lambda_{L}(s)-\beta\right\vert <\varepsilon,
\]
provided that $\mathrm{rank}(L)\geq n(\varepsilon)$, where $n(\varepsilon)$
depends only on $\varepsilon$;

\item[(vb)] in \emph{Case (b)}: for each $\beta\in\lbrack0,1]$ there exists
$s\in T$ with $\lambda_{L}(s)=\beta$.
\end{description}
\end{proposition}

\noindent Recall that $\mathrm{rank}(L)$ means the (untwisted) Lie rank of $L$
if $L$ is of Lie type, $n$ if $L\cong\mathrm{Alt}(n)$, and $0$ otherwise. The
proof is postponed to the following subsections.

Given an ultrafilter $\mathcal{U}$ on $I$, one defines the \emph{ultralimit}
of a bounded family $(a_{i})_{i\in I}$ of real numbers to be the unique number
$\alpha=\lim_{\mathcal{U}}a_{i}$ such that%
\[
\epsilon>0\Longrightarrow\left\{  i\in I\ |\ |a_{i}-\alpha|<\epsilon\right\}
\in\mathcal{U}%
\]
(cf. \cite{KL}, Section 3.1)\textbf{.} We remark that if $\mathcal{U}$ is the
principal ultrafilter $\mathcal{U}(j)$ over some element $j\in I$, then
$\lim_{\mathcal{U}}a_{i}=a_{j}$.

In Case (a), set $T_{i}=S_{i}$ for each $i$; in Case (b), we choose a maximal
torus $T_{i}$ in $S_{i}$. In either case, let $G_{\bullet}=\prod_{i\in I}%
T_{i}$. Now define a function $h_{\mathcal{U}}:G_{\bullet}\rightarrow
\lbrack0,1]$ by
\[
h_{\mathcal{U}}(\mathbf{g})=\lim\nolimits_{\mathcal{U}}\lambda_{S_{i}}%
(g_{i})\text{ for }\mathbf{g}=(g_{i})_{i}.
\]
The analogue of property (ii) obviously holds for the function $h_{\mathcal{U}%
}$. This implies that the set
\begin{equation}
K_{\mathcal{U}}:=h_{\mathcal{U}}^{-1}(0) \label{KU}%
\end{equation}
is a normal subgroup of $G_{\bullet}$, and that $h_{\mathcal{U}}$ is constant
on the cosets of $K_{\mathcal{U}}$.

For a subset $J$ of $I$ we set%
\[
N(J)=\prod_{i\in J}Z_{j}\times\prod_{i\in I\smallsetminus J}S_{i},
\]
the kernel of the projection $G\rightarrow\prod_{j\in J}S_{j}/Z_{j},$ where
$Z_{j}=\mathrm{Z}(S_{j})$. Each $N(J)$ is a closed normal subgroup of $G$.

\bigskip

Now we can prove Theorem \ref{unctble}. Let $Q=G/H$ where $H$ is a normal
subgroup of infinite index in $G$. Suppose we are in Case (b) (Lie groups); if
$H\leq N(j)$ then $Q$ maps onto $S_{j}/Z_{j}$ and the result is clear. Suppose
we are in Case (a), and let $J$ be the set of indices $j$ such that $H\leq
N(j)$. Then $\left\vert Q\right\vert =\left\vert G/N(J)\right\vert \left\vert
G_{1}/H_{1}\right\vert $ where $G_{1}=\prod_{i\in I\smallsetminus J}S_{j}$ and
$H_{1}$ denotes the projection of $H$ into $G_{1}$. If $J$ is infinite, then
$G/N(J)$ is an infinite profinite group and again the result is clear. If $J$
is finite, then $H_{1}$ has infinite index in $G_{1}$, and we can replace $G$
by $G_{1}$.

Thus in any case, we may assume that $H\nleqq N(j)$ for every $j\in I$. We
shall show that in this case,

\begin{description}
\item[(*)] \emph{There exists a non-principal ultrafilter }$\mathcal{U}$
\emph{on} $I$ \emph{such that }$H_{\bullet}:=H\cap G_{\bullet}\leq
K_{\mathcal{U}}$;

\item[(**)] $\left\vert G_{\bullet}/K_{\mathcal{U}}\right\vert \geq
2^{\aleph_{{0}}}$.
\end{description}

\noindent(Recall that $G_{\bullet}=G$ in Case (a).)$\medskip$

\textbf{Proof of (**)}.$\medskip$

\emph{Case 1.} The $S_{i}$ are finite simple groups, and for some
$m\in\mathbb{N}$, the set%
\[
D=D(m)=\{i\ |\ \mathrm{rank}(S_{i})\leq m\}
\]
belongs to $\mathcal{U}$. Then $N(D)\leq K_{\mathcal{U}}$, so
$G/K_{\mathcal{U}}\cong G_{1}/K_{\mathcal{U}_{1}}$ where $G_{1}=\prod_{i\in
D}S_{i}$ and $\mathcal{U}_{1}$ is the restriction of $\mathcal{U}$ to $D$. Now
property (iv) of the functions $\lambda_{S_{i}}$ implies that $\mathbf{g\in
}K_{\mathcal{U}_{1}}$ precisely when the set $\{i\in D\mid g_{i}=1\}$ belongs
to $\mathcal{U}_{1}$. Therefore the quotient $G_{1}/K_{\mathcal{U}_{1}}$
coincides with the ultraproduct $\prod_{i\in D}S_{i}/\mathcal{U}_{1}$. But an
ultraproduct of finite sets is either finite or has cardinality at least
$2^{\aleph_{{0}}}$ (\cite{FMS}, Theorem 1.31). The first possibility is
excluded since each of the sets $I(n)$ is finite, hence cannot belong to
$\mathcal{U}_{1}$, and (**) follows.$\medskip$

\emph{Case 2.} The $S_{i}$ are finite simple groups and $D(m)\notin
\mathcal{U}$ for each $m\in\mathbb{N}$. Let $\beta\in(0,1)$. For each $i\in I$
we choose $g_{i}\in S_{i}$ so as to minimize%
\[
\left\vert \lambda_{S_{i}}(g_{i})-\beta\right\vert =\varepsilon_{i},
\]
say. Property (va) ensures that for any $\varepsilon>0$, we have
$\varepsilon_{i}<\varepsilon$ whenever $\mathrm{rank}(S_{i})\geq
n(\varepsilon)$. We claim that $h_{\mathcal{U}}(\mathbf{g})=\beta$. Indeed,
suppose that $h_{\mathcal{U}}(\mathbf{g})=\beta^{\prime}\neq\beta$, and put
$\varepsilon=\left\vert \beta^{\prime}-\beta\right\vert $. Then%
\[
\left\vert \lambda_{S_{i}}(g_{i})-\beta^{\prime}\right\vert <\varepsilon
/2\Longrightarrow\varepsilon_{i}=\left\vert \lambda_{S_{i}}(g_{i}%
)-\beta\right\vert >\varepsilon/2\Longrightarrow i\in D(m)
\]
where $m=n(\varepsilon/2)$; thus $D(m)$ contains a member of $\mathcal{U}$ and
so $D(m)\in\mathcal{U}$, a contradiction.

It follows that $h_{\mathcal{U}}(G)=[0,1]$. Since $h_{\mathcal{U}}$ is
constant on cosets of $K_{\mathcal{U}}$ this now implies that
$G/K_{\mathcal{U}}$ has the cardinality of $[0,1]$, and (**) follows.$\medskip
$

\emph{Case 3.} The $S_{i}$ are connected simple Lie groups. Let $\beta
\in(0,1)$. Using Property (vb), choose $g_{i}\in T_{i}$ with $\lambda_{S_{i}%
}(g_{i})=\beta$ for each $i$. Then $\mathbf{g}=(g_{i})\in G_{\bullet}$ and
$h_{\mathcal{U}}(\mathbf{g})=\beta$; and (**) follows as in the preceding
case.$\medskip$

\textbf{Proof of (*)}.$\medskip$

$H$ is a normal subgroup of infinite index in $G$, and $H\nleqq N(j)$ for any
$j\in I$. For $\mathbf{t}=(t_{i})_{i}\in H_{\bullet}$ and $\epsilon>0$ put%
\[
A(\mathbf{t},\epsilon)=\left\{  i\in I\ |\ \lambda_{S_{i}}(t_{i}%
)<\epsilon\right\}  ,
\]
and let $U$ be the collection of all subsets $A(\mathbf{t},\epsilon)$ with
$\mathbf{t}\in H_{\bullet}$ and $\epsilon>0$.

We claim that \emph{every finite subset of} $U$ \emph{has nonempty
intersection}. Indeed, suppose that
\[
A(\mathbf{t}_{1},\epsilon_{1})\cap A(\mathbf{t}_{2},\epsilon_{2})\cap
\ldots\cap A(\mathbf{t}_{k},\epsilon_{k})=\varnothing.
\]
Put $\epsilon=\min_{i}\{\epsilon_{i}\}$ and suppose that $\mathbf{t}%
_{i}=(t_{i,j})_{j}$ with $t_{i,j}\in T_{j}$.

Then for each index $j\in I$ there is some $i\leq k$ such that $j\not \in
A(\mathbf{t}_{i},\epsilon)$, so $\lambda_{S_{j}}(t_{i,_{j}})\geq\epsilon$. Now
(iii) gives%
\[
S_{j}=\left(  t_{i,j}^{S_{j}}\cup t_{i,j}^{-S_{j}}\right)  ^{\ast n},
\]
where $n=f(\epsilon)$. Considering independently each coordinate $j\in I$ we
see that%
\[
G=\prod_{i=1}^{k}\left(  \mathbf{t}_{i}^{G}\cup\mathbf{t}_{i}^{-G}\right)
^{\ast n}\subseteq H,
\]
a contradiction.

On the other hand, \emph{the intersection of the collection }$U$\emph{ is
empty}. Let $T_{j}^{\ast}$ denote the projection of $H_{\bullet}$ into $S_{j}%
$. If $j$ belongs to every member of $U$ then $\lambda_{S_{j}}(t)=0$ for every
$t\in T_{j}^{\ast}$, whence $T_{j}^{\ast}\leq\mathrm{Z}(S_{j})$ by property
(i). Since the conjugates of $T_{j}^{\ast}$ generate the projection of $H$
into $S_{j}$, this implies that $H\leq N(j)$, contrary to hypothesis.

Now a standard application of Zorn's lemma establishes the existence of a
non-principal ultrafilter $\mathcal{U}$ on $I$ containing $U$. From the
definition of $U$ it follows that $h_{\mathcal{U}}(\mathbf{t})=0$ for all
$\mathbf{t}\in H_{\bullet}$, and (*) follows.

\subsubsection{The profinite case}

In Case (a) we can say rather more:

\begin{theorem}
\label{main2}Suppose that $G=\prod_{i\in I}S_{i}$ where each $S_{i}$ is a
finite (non-abelian) simple group and $\left\{  i\in I\mid\left\vert
S_{i}\right\vert \leq n\right\}  $ is finite for each $n$. Then

\begin{itemize}
\item every proper normal subgroup of $G$ is contained in a maximal one;

\item the maximal proper normal subgroups of $G$ are precisely the subsets
$K_{\mathcal{U}}$ for ultrafilters $\mathcal{U}$ on $I$;

\item the normal subgroup $K_{\mathcal{U}}$ is closed in $G$ if and only if
$\mathcal{U}$ is principal.
\end{itemize}
\end{theorem}

\begin{proof}
If $\mathcal{U}=\mathcal{U}(j)$ is principal then $K_{\mathcal{U}}=N(j)$ is a
closed maximal normal subgroup. If $\mathcal{U}$ is non-principal, then
$K_{\mathcal{U}}$ has infinite index in $G$, by (**). We claim that in this
case too, $K_{\mathcal{U}}$ is a maximal normal subgroup. Suppose that
$\mathbf{g}=(g_{i})_{i}\in G$ is not in $K=K_{\mathcal{U}}$. This means that
$h_{\mathcal{U}}(\mathbf{g})>0$, which in turn implies that for some
$\alpha>0$ the set
\[
A=\{i\in I\mid\lambda_{S_{i}}(g_{i})>\alpha\}
\]
belongs to $\mathcal{U}$.

Now if $i\in A$, we see from (iii) in Proposition \ref{lamda-function} that%
\[
S_{i}=(g_{i}^{S_{i}}\cup g_{i}^{-S_{i}})^{\ast n}%
\]
where $n=f(\alpha)$. It follows that%
\[
G=N(A)\cdot\left(  \mathbf{g}^{G}\cup\mathbf{g}^{-G}\right)  ^{\ast n}.
\]
As $\mathcal{U}$ is a filter and $A\in\mathcal{U}$ it is easy to see that
$N(A)\leq K$, and so
\[
G=K\left(  \mathbf{g}^{G}\cup\mathbf{g}^{-G}\right)  ^{\ast n}\subseteq
K\left\langle \mathbf{g}^{G}\cup\mathbf{g}^{-G}\right\rangle .
\]
Since $\mathbf{g}$ was an arbitrary element of $G\smallsetminus K$ it follows
that $G/K$ is simple.

Now suppose that $H$ is any proper normal subgroup of $G$. Then either $H\leq
N(j)=K_{\mathcal{U}(j)}$ for some $j\in I$, or (*) provides a non-principal
ultrafilter $\mathcal{U}$ such that $H\leq K_{\mathcal{U}}$.

It remains only to observe that if $\mathcal{U}$ is a non-principal
ultrafilter then $K_{\mathcal{U}}$ contains the restricted direct product of
the $S_{i}$, which is dense in $G$, and so $K_{\mathcal{U}}$ cannot be closed.
\end{proof}

\subsubsection{The connected case: automorphisms}

The material in this subsection will only be needed for the proof of Theorem
\ref{vdthm} in Subsection \ref{dense}. We consider $G=\prod_{i\in I}S_{i}$
where $I$ is an infinite set and each $S_{i}$ is a compact connected simple
Lie group. In this case, our functions $h_{\mathcal{U}}$ were only defined on
$G_{\bullet}=\prod_{i\in I}T_{i},$ which depends on a choice of maximal torus
$T_{i}$ in each $S_{i}$. Suppose that in each $S_{i}$ we choose maximal tori
$T_{i}^{(l)}$, $l=1,\ldots,d$. Let $\lambda_{S_{i}}^{(l)}:T_{i}^{(l)}%
\rightarrow\lbrack0,1]$ be as in Proposition \ref{lamda-function}, put
$\mathbf{T}^{(l)}=\prod_{i\in I}T_{i}^{(l)}$, and define $h_{\mathcal{U}%
}^{(l)}:\mathbf{T}^{(l)}\rightarrow\lbrack0,1]$ and $K_{\mathcal{U}}%
^{(l)}:=h_{\mathcal{U}}^{-1}(0)\leq\mathbf{T}^{(l)}$ as before, using the maps
$\lambda_{S_{i}}^{(l)}$. A subgroup of the form $\mathbf{T}^{(l)}$ will be
called a `maximal pro-torus' of $G$ (cf. \cite{HM1}). We will write
$\lambda_{i}$ for $\lambda_{S_{i}}$ where the meaning is clear.

\begin{lemma}
\label{several tori}Let $H$ be a proper normal subgroup of $G$ with $H\nleqq
N(j)$ for all $j\in I$. Then there exists a non-principal ultrafilter
$\mathcal{U}$ on $I$ such that $H_{(l)}:=H\cap\mathbf{T}^{(l)}\leq
K_{\mathcal{U}}^{(l)}$ for $l=1,\ldots,d$.
\end{lemma}

\begin{proof}
For $\mathbf{t}\in H_{(l)}$ and $\epsilon>0$ define $A^{(l)}(\mathbf{t}%
,\epsilon)$ as in the proof of (*), above, using $\lambda^{(l)}$ in place of
$\lambda$. Let $U^{(l)}$ be the collection of all subsets $A^{(l)}%
(\mathbf{t},\epsilon)$ with $\mathbf{t}\in H_{(l)}$ and $\epsilon>0$. As
above, it will suffice to show that every finite subcollection of $U^{(1)}%
\cup\ldots\cup U^{(d)}$ has non-empty intersection. Arguing as before, we see
that if%
\[
\bigcap_{l=1}^{d}\left(  A^{(l)}(\mathbf{t}_{1}^{(l)},\epsilon_{l1})\cap
A^{(l)}(\mathbf{t}_{2}^{(l)},\epsilon_{l2})\cap\ldots\cap A^{(l)}%
(\mathbf{t}_{k}^{(l)},\epsilon_{lk})\right)  =\varnothing,
\]
then for each $j\in I$ there exist $l\leq d$ and $i\leq k$ such that
$\lambda_{S_{j}}^{(l)}(t_{i,_{j}}^{(l)})\geq\epsilon$ where $\epsilon
=\min\epsilon_{l^{\prime}i^{\prime}}$. As before this yields the contradiction%
\[
G=\prod_{l=1}^{d}\prod_{i=1}^{k}\left(  \mathbf{t}_{i}^{(l)G}\cup
\mathbf{t}_{i}^{(l)-G}\right)  ^{\ast n}\subseteq H.
\]

\end{proof}

\bigskip

Now let $y$ be a continuous automorphism of $G$. The action of $y$ induces a
permutation $y^{\cdot}$ on the index set $I$, so that $S_{i}^{y}=S_{iy^{\cdot
}}$ for each $i$. Let $\mathcal{C}$ denote the set of orbits of $\left\langle
y^{\cdot}\right\rangle $ on $I,$ and for each $J\in\mathcal{C}$ pick $i(J)\in
J$. Then%
\[
\prod_{i\in J}S_{i}=\left\{
\begin{array}
[c]{ccc}%
\prod_{n\in\mathbb{Z}}S^{y^{n}} &  & (J\text{ infinite})\\
&  & \\
\prod_{n=0}^{e-1}S^{y^{n}} &  & (\left\vert J\right\vert =e<\infty)
\end{array}
\right.
\]
where $S=S_{i(J)}$. Choose a maximal torus $T_{i(J)}$ in $S_{i(J)}$, and for
$i=i(J)y^{\cdot n}$ (where $0\leq n<e$ if $\left\vert J\right\vert =e<\infty$)
set $T_{i}=T_{i(J)}^{y^{n}}$. Thus $\mathbf{T}=\prod_{i\in I}T_{i}$ becomes a
maximal pro-torus in $G$, and $\mathbf{T}$ is `almost' $y$-invariant, in the
following sense. For each $J\in\mathcal{C}$ with $\left\vert J\right\vert
<\infty$ put $l(J)=i(J)y^{\cdot-1}$, and set%
\begin{align*}
Z  &  =\left\{  l(J)\mid J\in\mathcal{C},~~\left\vert J\right\vert
<\infty\right\}  ,\\
\mathbf{T}(Z)  &  =\left\{  \mathbf{t}=(t_{i})\in\mathbf{T}\mid t_{i}%
=1~\forall i\in Z\right\}  ;
\end{align*}
for $i\notin Z$ we may identify $S_{iy^{\cdot}}$ with $S_{i}$ via the action
of $y$, and then for $\mathbf{t}=(t_{i})\in\mathbf{T}(Z)$ we have%
\begin{equation}
(\mathbf{t}^{y})_{iy^{\cdot}}=t_{i}~\forall i\in I\text{,} \label{ydot}%
\end{equation}
so $\mathbf{T}(Z)^{y}\leq\mathbf{T}$.

Set $Z^{c}=I\smallsetminus Z$. For $\alpha\in\lbrack0,1]$ and $\epsilon>0$
define%
\[
A(\mathbf{t},\alpha,\epsilon)=\left\{  i\in I\ |\ \left\vert \lambda_{S_{i}%
}(t_{i})-\alpha\right\vert <\epsilon\right\}  .
\]

\begin{lemma}
Let $\mathcal{U}$ be a non-principal ultrafilter on $I$ with $Z^{c}%
\in\mathcal{U}$, and put $\mathcal{U}^{\prime}=\mathcal{U}\left\vert _{Z^{c}%
}\right.  $ Then%
\[
\mathcal{U}^{\prime}=\left\{  A(\mathbf{t},1/2,1/4)\mid\mathbf{t}\in
\mathbf{T}(Z),~h_{\mathcal{U}}(\mathbf{t})=1/2\right\}  .
\]

\end{lemma}

\begin{proof}
Let $\mathcal{V}$ denote the family of sets on the right-hand side of the
equation. Then $\mathcal{V}\subseteq\mathcal{U}^{\prime}$ by the definition of
$h_{\mathcal{U}}(\mathbf{t})$.

Now suppose that $Y\subseteq Z^{c}$ and $Y\in\mathcal{U}$. Choose $t_{i}\in
T_{i}$ so that%
\begin{align*}
t_{i}  &  =1~\text{for }i\in Z\\
\lambda_{i}(t_{i})  &  =1/2~\text{for }i\in Y\\
\lambda_{i}(t_{i})  &  =1~\text{for }i\notin Y\cup Z.
\end{align*}
Then $\mathbf{t}=(t_{i})\in\mathbf{T}(Z)$ and $A(\mathbf{t},\frac{1}%
{2},\epsilon)=Y$ for every $\epsilon\in(0,\frac{1}{2}]$, so $h_{\mathcal{U}%
}(\mathbf{t})=\frac{1}{2}$. Therefore $Y\in\mathcal{V}$. Thus $\mathcal{U}%
^{\prime}\subseteq\mathcal{V}$.
\end{proof}

\begin{lemma}
Suppose that $Z^{c}\in\mathcal{U}$ and that $\mathbf{t}^{-1}\mathbf{t}^{y}\in
K_{\mathcal{U}}$ for all $\mathbf{t}\in\mathbf{T}(Z)$. Then $\mathcal{U}%
^{y^{\cdot}}=\mathcal{U}$.
\end{lemma}

\begin{proof}
Let $X\in\mathcal{U}$. Then $X\supseteq X\cap Z^{c}=A(\mathbf{t},\frac{1}%
{2},\frac{1}{4})$ for some $\mathbf{t}\in\mathbf{T}(Z)$ with $h_{\mathcal{U}%
}(\mathbf{t})=\frac{1}{2}$. Now%
\begin{align*}
h_{\mathcal{U}}(\mathbf{t}^{y})  &  =h_{\mathcal{U}}(\mathbf{t}.\mathbf{t}%
^{-1}\mathbf{t}^{y})\leq h_{\mathcal{U}}(\mathbf{t})+h_{\mathcal{U}%
}(\mathbf{t}^{-1}\mathbf{t}^{y})=h_{\mathcal{U}}(\mathbf{t}),\\
h_{\mathcal{U}}(\mathbf{t})  &  =h_{\mathcal{U}}(\mathbf{t}^{-1}%
)=h_{\mathcal{U}}(\mathbf{t}^{-1}\mathbf{t}^{y}.\mathbf{t}^{-y})\leq
h_{\mathcal{U}}(\mathbf{t}^{-1}\mathbf{t}^{y})+h_{\mathcal{U}}(\mathbf{t}%
^{-y})=h_{\mathcal{U}}(\mathbf{t}^{y}),
\end{align*}
so $h_{\mathcal{U}}(\mathbf{t}^{y})=\frac{1}{2}$. Now it follows from
(\ref{ydot}) that%
\[
A(\mathbf{t},1/2,1/4)^{y^{\cdot}}=A(\mathbf{t}^{y},1/2,1/4)=B,
\]
say, and $B\in\mathcal{U}$ since $h_{\mathcal{U}}(\mathbf{t}^{y})=\frac{1}{2}%
$. Therefore $X^{y^{\cdot}}\supseteq B\in\mathcal{U}$ and so $X^{y^{\cdot}}%
\in\mathcal{U}$. Thus $\mathcal{U}^{y^{\cdot}}\subseteq\mathcal{U},$ and the
result follows since $\mathcal{U}^{y^{\cdot}}$ is an ultrafilter.
\end{proof}

\begin{lemma}
If $\mathcal{U}^{y^{\cdot}}=\mathcal{U}$ then $\mathrm{fix}(y^{\cdot}%
)\in\mathcal{U}$.
\end{lemma}

\begin{proof}
Here $\mathrm{fix}(y^{\cdot})$ denotes the set of fixed points of $y^{\cdot}$.
We can partition $I$ as
\[
I=A_{1}\overset{\cdot}{\cup}A_{2}\overset{\cdot}{\cup}A_{3}\overset{\cdot
}{\cup}\mathrm{fix}(y^{\cdot})
\]
where $A_{i}^{y^{\cdot}}\cap A_{i}=\varnothing$ for $i=1,2,3$. To see this, it
suffices to partition each $\left\langle y^{\cdot}\right\rangle $-orbit $J$ of
length at least $2$ into three pieces $J_{i}$ such that $J_{i}^{y^{\cdot}}\cap
J_{i}=\varnothing$. Identifying $J$ with $\mathbb{Z}$ or with $(1,2,\ldots,e)$
where $y^{\cdot}$ takes $i$ to $i+1$ $(\operatorname{mod}e)$, let%
\begin{align*}
J_{1}  &  =2\mathbb{Z},~J_{2}=2\mathbb{Z}+1,~J_{3}=\varnothing\text{ if
}\left\vert J\right\vert =\infty;\\
J_{1}  &  =2\mathbb{Z}\cap J,~J_{2}=(2\mathbb{Z}+1)\cap J,~J_{3}%
=\varnothing\text{ if }\left\vert J\right\vert \text{ is even;}\\
J_{1}  &  =\{2,\ldots,2n\},~J_{2}=\{1,\ldots,2n-1\},~J_{3}=\{2n+1\}\text{ if
}\left\vert J\right\vert =2n+1.
\end{align*}
Then set $A_{i}=\cup_{J\in\mathcal{C}}J_{i}$ for $i=1,2,3$.

If $\mathcal{U}^{y^{\cdot}}=\mathcal{U}$ then $A_{i}\notin\mathcal{U}$ for
each $i$, since $\varnothing\notin\mathcal{U}$. Therefore $A_{i}^{c}%
\in\mathcal{U}$ for each $i$, whence%
\[
\mathrm{fix}(y^{\cdot})=A_{1}^{c}\cap A_{2}^{c}\cap A_{3}^{c}\in
\mathcal{U}\text{.}%
\]
(We are grateful to Martin Kassabov for pointing us to this lemma, which
suggested the possibility of Proposition \ref{ultrafilterautos}, below.)
\end{proof}

\begin{lemma}
Suppose that $Z\in\mathcal{U}$ and that $\mathbf{t}^{-1}\mathbf{t}^{y}\in
K_{\mathcal{U}}$ for all $\mathbf{t}\in\mathbf{T}(Z)$. Then $\mathrm{fix}%
(y^{\cdot})\in\mathcal{U}$.
\end{lemma}

\begin{proof}
If $J$ is an orbit of $\left\langle y^{\cdot}\right\rangle $ of length at
least $2$, choose $t_{J}\in T_{i(J)}$ with $\lambda_{i(J)}(t_{J})=1$. Then set%
\begin{align*}
t_{i(J)y^{\cdot n}}  &  =t_{J}^{y^{n}}\text{ }\forall n\in\mathbb{Z}\text{ if
}J\text{ is infinite,}\\
t_{i(J)y^{\cdot n}}  &  =t_{J}^{y^{n}}~(0\leq n\leq e-2),~t_{l(J)}=1\text{ if
}\left\vert J\right\vert =e<\infty;
\end{align*}
and set $t_{i}=1$ for each $i\in\mathrm{fix}(y^{\cdot})$ (recall that
$l(J)=i(J)y^{\cdot(e-1)}$). Then $\mathbf{t}=(t_{i})\in\mathbf{T}(Z)$, and
whenever $\infty>\left\vert J\right\vert \geq2$ we have%
\[
(\mathbf{t}^{-1}\mathbf{t}^{y})_{l(J)}=t_{J}.
\]
Now $\mathbf{t}^{-1}\mathbf{t}^{y}\in K_{\mathcal{U}}$ implies that
$A(\mathbf{t}^{-1}\mathbf{t}^{y},0,\frac{1}{2})\in\mathcal{U}$; consequently
$A(\mathbf{t}^{-1}\mathbf{t}^{y},0,\frac{1}{2})\cap Z\in\mathcal{U}$. As
$Z=\{l(J)\mid2\leq\left\vert J\right\vert <\infty\}\cup\mathrm{fix}(y^{\cdot
})$, we see that $A(\mathbf{t}^{-1}\mathbf{t}^{y},0,\frac{1}{2})\cap
Z=\mathrm{fix}(y^{\cdot})$.
\end{proof}

\begin{proposition}
\label{ultrafilterautos}Let $y_{1},\ldots,y_{d}$ be continuous automorphisms
of $G$ and let $H$ be a proper normal subgroup of $G$ with $[G,y_{l}]\subseteq
H$ for each $l$. Suppose that $H\nleqq N(j)$ for all $j\in I$. Then there
exists a non-principal ultrafilter $\mathcal{U}$ on $I$ such that%
\[
\bigcap_{l=1}^{d}\mathrm{fix}(y_{l}^{\cdot})\in\mathcal{U}\text{.}%
\]
Hence $\bigcap_{l=1}^{d}\mathrm{fix}(y_{l}^{\cdot})$ is infinite.
\end{proposition}

\begin{proof}
For each $l$ choose a maximal pro-torus $\mathbf{T}^{(l)}$ corresponding to
$y_{l}$ as above, and apply Lemma \ref{several tori} to find a non-principal
ultrafilter $\mathcal{U}$ such that $H\cap\mathbf{T}^{(l)}\leq K_{\mathcal{U}%
}^{(l)}$ for $l=1,\ldots,d$. Now the last three lemmas show that
$\mathrm{fix}(y_{l}^{\cdot})\in\mathcal{U}$ for each $l$, and the result follows.
\end{proof}

\subsubsection{Proposition \ref{lamda-function}, finite case}

Now $L$ is a finite simple group. We define%
\[
\lambda(s)=\frac{\log\left\vert s^{L}\right\vert }{\log\left\vert L\right\vert
}.
\]
Properties (i) and (ii) are clear, and (iii) follows from Proposition
\ref{conjclass}. (iv) follows from Proposition \ref{bcpbound}.

It remains to establish property (v). Given $\beta,\varepsilon\in(0,1)$, we
have to show that provided $\mathrm{rank}(L)$ is sufficiently large, there
exists $g\in L$ such that%
\[
\frac{\log\left\vert \mathrm{C}_{L}(g)\right\vert }{\log\left\vert
L\right\vert }\in(\alpha-\varepsilon,\alpha+\varepsilon)
\]
where $\alpha=1-\beta$. As we only need to consider groups of large rank, we
may suppose that $L$ is either alternating or a classical group.$\medskip$

If $L=\mathrm{Alt}(n)$, take $g$ to be an even cycle of length $l\thicksim
\beta n$ in $\mathrm{Alt}(n)$. Note that $\left\vert \mathrm{C}_{L}%
(g)\right\vert $ is roughly\textbf{ }$l\cdot\overline{l}!/2$ where
$\overline{l}\thicksim\alpha n$. By Stirling's formula, $\log(n!)\sim n\log n$
and hence $\log(l\cdot\overline{l}!/2)\sim\alpha\log(n!/2)$ as $n\rightarrow
\infty$.\medskip

If $L$ is a simple classical group, consider the corresponding universal
quasisimple classical group $\widetilde{L}$ acting on its natural module $V$
over a finite field of size $q$ equipped with a bilinear form $f$ (symmetric,
sesquilinear, alternating or just equal to $0$ in case $L$ has type
$\mathrm{PSL}_{n}$). Note that $\dim(V)\rightarrow\infty$ as $\mathrm{rank}%
(L)\rightarrow\infty$. We have $L=\widetilde{L}/Z$ where $Z$ is the centre of
$\widetilde{L}$; and if $g=\tilde{g}Z\in L$ with $\tilde{g}\in\widetilde{L}$ then%

\[
|g^{L}|\leq|\tilde{g}^{\widetilde{L}}|\leq|Z||g^{L}|.
\]
Since $Z$ has asymptotically negligible size compared to $L$ it is enough to
find an element $\tilde{g}\in\widetilde{L}$ with $\log\left\vert
\mathrm{C}_{\widetilde{L}}(\tilde{g})\right\vert \sim\alpha\log|\widetilde
{L}|$.

We can decompose $V$ as $V_{0}\oplus V_{1}\oplus V_{2}$ so that:

\begin{itemize}
\item $\dim V_{0}$ is about $\sqrt{\alpha} \dim V$, and $\dim V_{1} = \dim
V_{2}$,

\item $V_{1} \oplus V_{2}$ is orthogonal to $V_{0}$, and

\item The form $f$ is nondegenerate on both $V_{0}$ and $V_{1} \oplus V_{2}$
and is isotropic on $V_{1}$ and on $V_{2}$
\end{itemize}

Let $\tilde{g}\in\widetilde{L}$ be equal to the identity on $V_{0}$ and act on
each of $V_{1}$ and $V_{2}$ as a cyclic transformation without fixed vectors.
In other words there is a vector $v_{i}\in V_{i}$, $(i=1,2)$ such that
$v_{i},\tilde{g}v_{i},\tilde{g}^{2}v_{i},\ldots$ is a basis for $V_{i}$.

Now $\mathrm{C}_{\widetilde{L}}(\tilde{g})$ contains the classical group $H$
on $V_{0}$ preserving $f$, and by the choice of $\dim V_{0}$ we have
$\log|H|/\log|\widetilde{L}|\sim(\dim V_{0}/\dim V)^{2}$ which tends to
$\alpha$ as $\dim V\rightarrow\infty$.

On the other hand if $s\in\widetilde{L}$ commutes with $\tilde{g}$ then $s$
must stabilize $V_{0}$, the fixed space of $\tilde{g}$. Since $V_{1}$ and
$V_{2}$ are cyclic modules for $\tilde{g}$, the action of $s$ on $V_{1}$ and
$V_{2}$ is determined by $s\cdot v_{1}$ and $s\cdot v_{2}$. Hence $s$ is
completely known from its restriction to $V_{0}$ and from the two vectors
$sv_{1},sv_{2}\in V$. Denote by $\mathrm{Gf}(V_{0})$ the subgroup of
$\mathrm{GL}(V_{0})$ which preserves $f$. We have $|\mathrm{Gf}(V_{0})|\leq
q\left\vert H\right\vert $.

Therefore
\[
\left\vert H\right\vert \leq\mathrm{C}_{\widetilde{L}}(\tilde{g}%
)\leq|\mathrm{Gf}(V_{0})||V|^{2}\leq q^{1+2\dim V}\left\vert H\right\vert
\]
which gives
\[
\log|\mathrm{C}_{\widetilde{L}}(\tilde{g})|/\log\left\vert \widetilde
{L}\right\vert \sim\log\left\vert H\right\vert /\log\left\vert \widetilde
{L}\right\vert \rightarrow\alpha
\]
as $\dim V$ tends to infinity.

\subsubsection{Proposition \ref{lamda-function}, connected case}

We shall need some information about the tori and roots of compact simple Lie
groups; see for example \cite{Bump}, Chapter 19, \cite{HM1}, Chapter 6. By
$S^{1}$ we shall denote the group of complex numbers of absolute value $1$
under multiplication. It is a compact torus of dimension $1$.

Let $L$ be a compact simple Lie group with centre $Z$ (possibly nontrivial).
Let $T$ be a maximal torus of $L$ (this is unique up to conjugacy). Every
element of $L$ is conjugate to an element of $T$. Let $\Phi$ be a set of roots
with respect to $T$. We choose and fix a set of fundamental roots $\Pi
=\{\beta_{1},\ldots,\beta_{r}\}$; $r$ is the rank of $L$. Every root
$\alpha\in\Phi$ corresponds to a character $T\rightarrow S^{1}$ which we will
also denote by $\alpha$. We have%
\[
\bigcap_{i=1}^{r}\ker\beta_{i}=Z.
\]
There is also a cocharacter $h_{\alpha}:S^{1}\rightarrow T$ such that
$\alpha(h_{\alpha}(\mu))=\mu^{2}$ fo all $\mu\in S^{1}$. For every pair
$(\pm\alpha)$ of opposite roots of $\Phi$ there is a homomorphism $f_{\alpha
}:\mathrm{SU}(2)\rightarrow L$ such that $h_{\alpha}$ is the restriction of
$f_{\alpha}$ to the diagonal subgroup $\mathrm{diag}(\mu,\mu^{-1})$ of
$\mathrm{SU}(2)$ (and $h_{-\alpha}=h_{\alpha}^{-1}$). Let $S_{\alpha
}=S_{-\alpha}$ be the image of $\mathrm{SU}(2)$ in $L$ under $f_{\alpha}$.
Then $S_{\alpha}$ is either $\mathrm{SU}(2)$ or $\mathrm{PSU}(2)\cong%
\mathrm{SO}(3)$. Moreover $S_{\alpha}$ commutes elementwise with the closed
subgroup $T_{\alpha}:=\{g\in T\ |\ \alpha(g)=1\}$ of $T$, and the central
product $S_{\alpha}T_{\alpha}$ contains $T$.

Now we have to define $\lambda:T\rightarrow\lbrack0,1]$ so that properties (i)
-- (iii) and (v) of Proposition \ref{lamda-function} hold.

We can write a complex number $\mu\in S^{1}$ in a unique way as $\mu
=e^{i\theta}$ with $\theta\in(-\pi,\pi]$. Set $l(\mu):=|\theta|$. We shall
refer to $l(\mu)$ as the \emph{angle} of $\mu$.$\medskip$

\noindent\textbf{Definition.}\label{ang} For an element $g\in T$ define%
\[
\lambda(g)=\frac{1}{\pi r}\sum_{i=1}^{r}l(\beta_{i}(g))
\]

$\medskip$

Clearly $\lambda(g)$ is the same as $\lambda(\bar{g})$ for $\bar{g}=gZ$, if
$\lambda$ is defined taken with respect to the torus $T/Z$ of $L/Z$.

It is also clear that (i) $\lambda(g)=0$ if and only if $g\in Z$, and (ii)
$\lambda(h_{1})=\lambda(h_{1}^{-1})$ and $\lambda(h_{1}h_{2})\leq\lambda
(h_{1})+\lambda(h_{2})$ for any $h_{1},h_{2}\in T$. Since $l(\mu)$ takes all
values in $[0,\pi]$ and $T$ is a torus, we see that $\lambda(T)=[0,1]$, which
is property (v).$\medskip$

\emph{Example}: If $L=\mathrm{SU}(2)$ and $g$ is an element of the diagonal
subgroup of $L$ with eigenvalues $\mu$ and $\mu^{-1}$ then $l(g)$ is the angle
of $\mu^{2}$. From here and the isomorphism $\mathrm{PSU}(2)\cong%
\mathrm{SO}(3)$ we see that if $g\in\mathrm{SU}(2)$ then $\lambda(g)$ is
$\left\vert \theta\right\vert /\pi$ where $\theta$ is the angle of the image
$\bar{g}\in\mathrm{PSU}(2)=\mathrm{SO}(3)$ considered as a rotation of
$\mathbb{R}^{3}$.$\medskip$

Property (iii) follows from

\begin{lemma}
\label{lprod} There is an absolute constant $C>0$ such that if $g\in T$ and
$C/(\lambda(g))^{2}<M\in\mathbb{N}$ then $K^{\ast M}=L$, where $K=g^{L}\cup
g^{-L}$.
\end{lemma}

First we consider a special case:

\begin{lemma}
\label{su}If $L=\mathrm{SU}(2)$ and $g\in L$ with $\lambda(g)=\epsilon>0$ then
every element of $L$ is a product of $N=\left\lceil 2/\epsilon\right\rceil $
conjugates of $g$. Moreover $L=[L,g]^{\ast N}$.
\end{lemma}

\begin{proof}
Consider the realization of $\mathrm{SU}(2)<\mathrm{GL}_{2}(\mathbb{C)}$ by
unitary matrices:
\[
\mathrm{SU}(2)=\left\{  \left(
\begin{array}
[c]{cc}%
a & b\\
-\bar{b} & \bar{a}%
\end{array}
\right)  \mid|a|^{2}+|b|^{2}=1\right\}  .
\]

The conjugacy class of an element $h\in\mathrm{SU}(2)$ is uniquely determined
by its trace $\mathrm{tr}(h)\in\lbrack-2,2]$. Write $\mathrm{tr}%
(g)=2\cos\gamma$ with $\gamma\in\lbrack0,\pi]$; then for fundamental roots
$\alpha$ we have $\alpha(g)=e^{\pm2i\gamma}$, and so $\lambda(g)=2\gamma/\pi$
if $\gamma\in\lbrack0,\pi/2]$, $\lambda(g)=2(\pi-\gamma)/\pi$ otherwise. Of
course $\lambda(g)=\lambda(-g)$ and $(g^{L})^{\ast N}=L$ is equivalent to
$((-g)^{L})^{\ast N}=L$. So by replacing $g$ with $-g$ if necessary we may
assume that $\lambda(g)=2\gamma/\pi=\epsilon>0$ and $\gamma=\pi\epsilon
/2\in(0,\pi/2]$. Now a direct computation shows that if $h\in L$ is a diagonal
element with $\mathrm{tr}(h)=2\cos\theta$ then for any $\theta_{1}\in
\lbrack\theta-\gamma,\theta+\gamma]$ we may find a matrix $g^{\prime}\in
g^{L}$. (i.e. such that $\mathrm{tr}(g^{\prime})=2\cos\gamma)$) with
$\mathrm{tr}(hg^{\prime})=2\cos\theta_{1}$. This shows that for any integer
$m>1$, any element of $L$ with trace $2\cos\theta_{2}$ with $\theta_{2}%
\in\lbrack0,m\gamma]$ is a product of $m$ conjugates of $g$. Taking
$N=\left\lceil 2/\epsilon\right\rceil $ we have $N\gamma\geq\pi$ and so
$(g^{L})^{\ast N}=L$.

This proves the first claim of the Lemma. The second claim follows since
$[L,g]=(g^{-1})^{L}g=g^{L}\cdot g$ and
\[
\lbrack L,g]^{\ast N}=(g^{L}\cdot g)^{\ast N}=(g^{L})^{\ast N}g^{N}=Lg^{N}=L.
\]

\end{proof}

\bigskip

We now consider the general case of Lemma \ref{lprod}. It is enough to prove
it when $L$ is simply connected, since the definition of $\lambda$ was the
same for $L$ and $L/\mathrm{Z}(L)$. We shall assume this from now on.

Let us write $H_{\alpha}$ for the one-parameter torus $\{h_{\alpha
}(t)\ |\ t\in S_{1}\}$ given by the image of the cocharacter $h_{\alpha}$.
Thus we have $T=H_{\beta_{1}}\times\cdots\times H_{\beta_{r}}$. Take an
element $g\in L$ with $\lambda(g)=\epsilon>0$. Then for at least one
fundamental root $\beta_{j}$ we have $l(\beta_{j}(g))\geq\epsilon\pi$. Fix
$N=\left\lceil 2/\epsilon\right\rceil $ as above.\medskip

\emph{Case 1:} Assume that the rank $r$ of $L$ satisfies $r\leq\max
\{10,4/\epsilon\}$.$\medskip$

In the central product $S_{\beta}T_{\beta}$ we can write $g$ as $g=g_{1}g_{2}$
where $g_{1}\in S_{\beta}$ and $g_{2}\in T_{\beta}$. Now $S_{\beta}$ is a copy
of $\mathrm{SU}(2)$ (not $\mathrm{PSU}(2)$ since $L$ is simply connected), and
by Lemma \ref{su} we can express any $h\in S_{\beta}$ as $h=\prod_{i=1}%
^{N}[s_{i},g_{1}]$ for some $s_{i}\in S_{\beta}$. Then
\[
h=\prod_{i=1}^{N}[s_{i},g]
\]
and in particular the subgroup $H_{\beta}\leq S_{\beta}$ is contained in
$K^{\ast2N}$. Recall that the Weyl group $W$ acts on $H$. For a pair of roots
$\gamma_{1},\gamma_{2}$ of $\Phi$ of the same length there is some element
$v\in W$ such that $\gamma_{1}^{v}=\gamma_{2}$ and consequently $H_{\gamma
_{1}}^{v}=H_{\gamma_{2}}$. Moreover, if $\gamma,\delta$ are two roots of
different lengths in $\Phi$ then $\gamma$ is in the linear span of roots
$\delta_{1}$ and $\delta_{2}$ in the orbit of $\delta$ under $W$, and then%
\[
H_{\gamma}\leq H_{\delta_{1}}H_{\delta_{2}}=H_{\delta}^{u_{1}}H_{\delta
}^{u_{2}}\quad\text{ for some }u_{1},u_{2}\in W.
\]

Therefore each of the groups $H_{\beta_{i}}$ is contained in $K^{\ast4N}$. But
$T$ is a product of all the $H_{\beta_{i}}$ for $i=1,\ldots,r$ and hence
\[
T\subseteq K^{\ast4rN}.
\]
Now the right-hand side is a union of conjugacy classes of $L$; since every
conjugacy class intersects $T$ we have $L=K^{\ast4rN}=K^{\ast M}$ as long as
$M\geq4rN=O(\epsilon^{-2})$, since $r\leq\max\{10,4/\epsilon\}$.\medskip

\emph{Case 2:} The Lie rank of $L$ exceeds both $10$ and $4/\epsilon$. This
means that $L$ is a classical Lie group of type $A_{r},B_{r},C_{r}$ or $D_{r}%
$. In all these cases we can label the fundamental roots in $\Pi$ so that
$\beta_{1},\ldots,\beta_{r-1}$ span a root system of type $A_{r-1}$ and the
angle between $\beta_{i}$ and $\beta_{i+1}$ is $2\pi/3$ for $i=1,\ldots,r-2$.
(This is the labelling on the vertices of the Dynkin diagram of $L$ where we
number the vertices on the $A_{r-1}$ part of the diagram consecutively.) The
last root $\beta_{r}$ may have different length from the others.

Put $\eta=\epsilon/8$. It is immediate that for a subset $\Delta\subseteq\Pi$
of size at least $4\eta r$ we must have $l(\beta_{i}(g))\geq\epsilon\pi/2$ for
all $i\in\Delta$: otherwise, the average on $\Pi$ could not be $\epsilon\pi$
since each $l(\beta_{i}(g))\leq\pi$. Define
\[
\Pi_{1}=\{\beta_{i}|\ 1\leq i\leq r-1\text{ and }i\text{ even}\},\quad\Pi
_{2}=\Pi\smallsetminus(\Pi_{1}\cup\{\beta_{r}\}).
\]
Then each $\Pi_{i}$ consists of pairwise orthogonal roots and their union is
$\Pi\smallsetminus\{\beta_{r}\}$.

Observe that $\left\vert \Delta\right\vert \geq\epsilon r/2\geq2$ since
$r\geq4/\epsilon$. Put $\Delta_{i}=\Pi_{i}\cap\Delta$. Since $\left\vert
\Delta_{1}\right\vert +\left\vert \Delta_{2}\right\vert \geq\left\vert
\Delta\right\vert -1\geq\left\vert \Delta\right\vert /2\geq2\eta r$ we have
either $\left\vert \Delta_{1}\right\vert \geq\eta r$ or $\left\vert \Delta
_{2}\right\vert \geq\eta r$. Without loss of generality assume that
$\left\vert \Delta_{1}\right\vert \geq\eta r$.

The roots in $\Delta_{1}$ are pairwise orthogonal. The group $Q:=\langle
T,S_{\beta}\ |\ \beta\in\Delta_{1}\rangle$ is therefore isomorphic to the
central product
\[
\left(  \prod_{\beta\in\Delta_{1}}S_{\beta}\right)  \circ T_{\Delta_{1}},
\]
where $T_{\Delta_{1}}=\{h\in T\ |\ \beta(h)=1\ \forall\beta\in\Delta_{1}\}$
and $\prod_{\beta\in\Delta_{1}}S_{\beta}$ is the direct product of the
$S_{\beta}$.

Now if $\beta\in\Delta_{1}$ we have $l(\beta(g))\geq\epsilon\pi/2$. Just as in
Case 1, working independently in each $S_{\beta}$ and using Lemma \ref{su} we
deduce that
\begin{equation}
\prod_{\beta\in\Delta_{1}}H_{\beta}\subseteq\lbrack Q,g]^{\ast N_{1}}\subseteq
K^{2N_{1}} \label{delta}%
\end{equation}
where $N_{1}=\left\lceil 4/\epsilon\right\rceil $. We now refer to the
following straightforward

\begin{lemma}
\label{An} Let $\Psi$ be the set of roots in the root system of type $A_{n}$.
For an integer $m\leq n/2$ let $X,Y\in\Psi^{(m)}$ be two $m$-tuples of
elements of $\Psi$ each consisting of pairwise orthogonal roots. Then
$X=Y^{w}$ for an element $w$ in the Weyl group of $\Psi$.
\end{lemma}

\begin{proof}
This can be done directly from the realization of $\Psi$ and the fact that
$W=\mathrm{Sym}(n+1)$. Alternatively it follows by induction on $m$ and using
that for any root $\alpha\in\Psi$, the orthogonal complement $\Psi\cap
\alpha^{\perp}$ is a root system of type $A_{n-2}$.
\end{proof}

$\medskip$

Now the set $\Pi_{1}$ is a union of at most $\left\vert \Pi_{1}\right\vert
/\eta r+1$ subsets of size $\left\vert \Delta_{1}\right\vert $ and the same
holds for $\Pi_{2}$. Altogether $\Pi_{1}\cup\Pi_{2}$ is a union of at most
$r/\eta r+2=1/\eta+2$ subsets of size $\left\vert \Delta_{1}\right\vert $.
Using Lemma \ref{An} and (\ref{delta}) we see that
\[
\prod_{i=1}^{r-1}H_{\beta_{i}}\subseteq K^{\ast N_{2}}%
\]
where $N_{2}=2\left\lceil 1/\eta+2\right\rceil N_{1}$. Finally $H_{\beta_{r}%
}\subseteq K^{\ast4N}$, and hence $T\subseteq K^{N_{2}+4N}$. Again, it follows
that $L=K^{\ast M}$ as long as $M\geq N_{2}+4N=O(\epsilon^{-2})$.

\subsection{Countable quotients of compact groups\label{fgquot}}

In this subsection, by a \emph{quotient} of a topological group $G$ we mean a
quotient of the underlying abstract group, unless stated otherwise. We will be
interested in countable quotients: in this subsection, one can always replace
`countable' with `of cardinality strictly less than $2^{\aleph_{{0}}}$'.

Until further notice, we assume that $G$ is \emph{a compact group such that
the profinite quotient} $G/G^{0}$ \emph{is finitely generated (topologically)}%
. Recall (Corollary \ref{derived-cpct}) that the derived group $G^{\prime}$ of
$G$ is closed; this applies likewise if $G$ is replaced by any open subgroup
of $G$.

The following observation is an immediate consequence of Corollary
\ref{cpctserre}:

\begin{corollary}
\label{resfinq}If $M$ is a normal subgroup of $G$ and $G/M$ is residually
finite then $M$ is closed.
\end{corollary}

\noindent Indeed, $M$ is an intersection of normal subgroups of finite index,
each of which is open.

$\medskip$

Suppose to begin with that $G$ is infinite and \emph{abelian}. If $G/G^{0}$
has $\mathbb{Z}_{p}$ as a quotient for some prime $p$ then, as observed in the
introduction, we obtain a homomorphism%
\[
G\rightarrow\mathbb{Z}_{p}\rightarrow\mathbb{Q}_{p}\rightarrow\mathbb{Q}%
\]
with countably infinite image. If $G/G^{0}$ is infinite but does not have any
quotient of type $\mathbb{Z}_{p}$, then $G/G^{0}$ must have infinitely many
Sylow subgroups, and so has a quotient $Q=\prod_{p\in\pi}C_{p}$ where $\pi$ is
an infinite set of primes. We may identify $Q$ with the additive group of
$S=\prod_{p\in\pi}\mathbb{F}_{p}$, which maps onto a non-principal
ultraproduct $\widetilde{S}$ of the $\mathbb{F}_{p}$. Now $\widetilde{S}$ is a
field of characteristic zero, hence admits an additive epimorphism to
$\mathbb{Q}$; thus $G$ admits an epimorphism to $\mathbb{Q}$. (We are indebted
to J. Kiehlmann for pointing out a gap in our original argument.)

If $G/G^{0}$ is finite then $G^{0}\neq1$, and then $G^{0}$ maps onto a torus
$T$. Let $D$ be the torsion subgroup of $T$. Then $T/D$ is a divisible
torsion-free abelian group, so a vector space over $\mathbb{Q}$; choosing an
epimorphism $T/D\rightarrow\mathbb{Q}$ we obtain an epimorphism (of abstract
groups) $G^{0}\rightarrow\mathbb{Q}$.

Now suppose that $G$ has an open normal subgroup $K$ such that $K/K^{\prime}$
is infinite. The preceding remarks shows that $K$, and therefore also $G$, has
a countably infinite quotient.

A group $Q$ is said to be \emph{FAb} if every virtually-abelian quotient of
$Q$ is finite; when $Q$ is a topological group, this refers to
\emph{continuous} quotients.

\begin{theorem}
\label{FAb}Let $G$ be a compact group such that $G/G^{0}$ is (topologically)
finitely generated. Then every countable FAb quotient of $G$ is finite.
\end{theorem}

Before giving the proof, let us deduce

\begin{corollary}
$G$ has a countably infinite quotient if and only if $G$ is not FAb.
\end{corollary}

\noindent We remark that many familiar compact groups are FAb: among connected
groups, these are just the semisimple ones; among profinite groups, examples
include $\mathfrak{G}(\mathbb{Z}_{p})$ for Chevalley groups $\mathfrak{G}$.

$\medskip$

\begin{proof}
The remarks above show that if $G$ is not FAb then $G$ has a countably
infinite quotient. Suppose conversely that $G$ has a countably infinite
quotient $G/N$. By Theorem \ref{FAb}, we may suppose that $G/N$ is virtually
abelian, so $G$ has a normal subgroup $K$ of finite index with $K^{\prime}\leq
N\leq K$. Now $K$ is open by Corollary \ref{cpctserre} and so $K^{\prime}$ is
closed. Thus $G/K^{\prime}$ is an infinite virtually-abelian continuous
quotient of $G$, so $G$ is not FAb.
\end{proof}

$\medskip$

\begin{proof}
[Proof of Theorem \ref{FAb}]Let $H$ be a normal subgroup of $G$ such that
$G/H$ is countable and FAb, and suppose that $G/H$ is infinite.

Set $P=(G^{0})^{\prime}$. Then $P$ is closed in $G$ and $P$ is a semisimple
connected compact group, hence has no proper countable quotient, by Theorem
\ref{unctble} (and the \emph{remark} preceding it). So $H\geq P$, and
replacing $G$ by $G/P$ we may suppose that $G^{0}$ is abelian.

Since $G^{0}H/H$ is abelian, $G/G^{0}H$ must be infinite. Replacing $G$ by
$G/G^{0}$ and $H$ by $G^{0}H/G^{0}$, we may suppose that $G$ is a finitely
generated profinite group. Put $K=\overline{H}$; then $K$ is open in $G$, so
$K$ is again a finitely generated profinite group. Now $G/K^{\prime}H$ is
virtually abelian and therefore finite. Thus $K^{\prime}H$ is open by Theorem
\ref{serre}, and so $K^{\prime}H=K$.

Now recall the definition of $K_{0}$ (see the Introduction). This is a
characteristic closed subgroup of $K$ such that $K^{(3)}K_{0}/K_{0}$ is
semisimple, where $K/K^{(3)}$ is soluble of derived length at most $3$. Since
any soluble FAb group is finite, we infer that $G/K^{(3)}K_{0}H$ is finite,
and as before conclude that $K^{(3)}K_{0}H=K$. Thus $K/HK_{0}$ is a countable
image of the finitely generated semisimple group $K^{(3)}K_{0}/K_{0}$; so
$K/HK_{0}$ is finite by Theorem \ref{unctble}, and as above it follows that
$HK_{0}=K$.

Now Corollary \ref{normal_mod_G_0} shows that $H=K$. Hence $G/H$ is finite, a contradiction.
\end{proof}

\bigskip

Now we consider arbitrary compact groups:

\begin{theorem}
\label{fgabs}Let $G$ be a compact group and $N$ a normal subgroup of (the
underlying abstract group) $G$. If $G/N$ is finitely generated then $G/N$ is finite.
\end{theorem}

\begin{proof}
Suppose that $G/N$ is finitely generated and infinite. Then $G=N\left\langle
X\right\rangle $ for some finite subset $X$. Let $K=\overline{\left\langle
X\right\rangle }$ be the subgroup topologically generated by $X$. Then
$G/N\cong K/(K\cap N)$, so replacing $G$ by $K$ we may suppose that $G$ is
topologically finitely generated. Now $G/N$ is countable, hence by Theorem
\ref{FAb} there exists $M\vartriangleleft G$ with $M\geq N$ such that $G/M$ is
infinite and virtually abelian. But a finitely generated virtually abelian
group is residually finite; hence $M$ is closed in $G$, by Corollary
\ref{resfinq}. Thus $G/M$ is both countably infinite and compact, a contradiction.
\end{proof}

\subsection{Dense normal subgroups\label{dense}}

Let $G$ be a compact group such that $G/G^{0}$ is (topologically) finitely
generated. If $N\vartriangleleft G$ and $G/N$ is countable then the closure
$\overline{N}$ of $N$ is open in $G$; in this case, we say that $N$ is
\emph{virtually dense}. Generalizing the preceding subsection, we can ask:
under what conditions does $G$ have a virtually dense normal subgroup $N$ of
infinite index? Note that $N$ has infinite index if and only $N$ is not
closed, in view of Corollary \ref{cpctserre}.

Suppose that $G$ is abelian. If $G/G^{0}$ is infinite, then $G/G^{0}$ contains
a dense (abstractly) finitely generated subgroup. If $G^{0}$ is infinite, then
$G^{0}$ has a dense proper subgroup (necessarily of infinite index), because
it maps onto a torus.

A group of the form $\prod_{i\in I}S_{i}$ is said to be \emph{strictly
infinite semisimple} if the index set $I$ is infinite and \emph{either} each
$S_{i}$ is a finite (non-abelian) simple group \emph{or} each $S_{i}$ is a
connected compact simple Lie group. Such a group has a \emph{characteristic}
dense subgroup of infinite index, namely the restricted direct product $N$ of
the $S_{i}$. Note that $N$ is \emph{countable} if $I$ is countable and the
$S_{i}$ are finite groups.

It turns out that these examples essentially account for all possibilities:

\begin{theorem}
\label{vdthm}Let $G$ be a compact group such that $G/G^{0}$ is (topologically)
finitely generated. Then $G$ has a virtually dense normal subgroup of infinite
index if and only if $G$ has an open normal subgroup $H$ and a closed normal
subgroup $K<H$ such that $H/K$ is either infinite and abelian or strictly
infinite semisimple.
\end{theorem}

In one direction, this follows quickly from the preceding observations.
Supposing that $H$ and $K$ exist as indicated, we may as well assume that
$K=1$. Necessarily $H\geq G^{0}$. If $H$ is strictly infinite semisimple, then
$H$ has a characteristic dense subgroup $N$ of infinite index, and then $N$ is
normal in $G$.

Now suppose that $H$ is abelian. If $G/G^{0}$ is infinite then $H/G^{0}$ has a
countable dense subgroup $M/G^{0}$. Then $N:=\left\langle M^{G}\right\rangle
=M^{g_{1}}\ldots M^{g_{n}}$ is virtually dense and normal in $G$, where
$\{g_{1},\ldots,g_{n}\}$ is a set of coset representatives for $G/H$, and
$N/G^{0}$ is countable, so $N$ has infinite index in $G$. Suppose finally that
$G/G^{0}$ is finite. As $G^{0}$ is a compact connected abelian group, it has a
subgroup $T$ such that $G^{0}/T$ is a one-dimensional torus. Put $S=T^{g_{1}%
}\cap\ldots\cap T^{g_{n}}$ where $\{g_{1},\ldots,g_{n}\}$ is a set of coset
representatives for $G/G^{0}$. Then $G^{0}/S$ is a torus, so has a countable
dense subgroup $M/S$ (in fact we can choose $M/S$ to be cyclic). Now take
$N=\left\langle M^{G}\right\rangle =M^{g_{1}}\ldots M^{g_{n}}$ as before.

$\medskip$

For the converse, let $N$ be a normal subgroup of infinite index in (the
abstract group) $G$ such that $L=\overline{N}$ is open in $G$. Note that
$L\geq G^{0}$ and that $L/G^{0}$ is a finitely generated profinite group. It
will suffice to find an open normal subgroup $H$ of $G$ and a closed normal
subgroup $K$ of $H$ such that $H/K$ is either infinite and abelian or strictly
infinite semisimple; for if $\{g_{1},\ldots,g_{n}\}$ is a set of coset
representatives for $G/H$ then $K_{\ast}=K^{g_{1}}\cap\ldots\cap K^{g_{n}}$ is
closed and normal in $G$, and $H/K_{\ast}$ is a subdirect product of copies of
$H/K$, hence shares the given property of $H/K$.

Now we separate cases.$\medskip$

\emph{Case 1}: where $G^{0}=1$, i.e. $G$ is profinite.$\medskip$

Recall that $L^{\prime}$ is closed, by Corollary \ref{derived-cpct}. Suppose
that both $L/L^{\prime}$ and $L/L_{0}$ are finite. Then both $L^{\prime}$ and
$L_{0}$ are open in $L$, so $NL^{\prime}=NL_{0}=L$. It follows by Corollary
\ref{normal_mod_G_0} (applied to the finitely generated profinite group $L$)
that $N=L$, a contradiction. Therefore at least one of $L/L^{\prime}$,
$L/L_{0}$ is infinite.

If $L/L^{\prime}$ is infinite we set $H=L$ and $K=L^{\prime}$. Suppose finally
that $L/L_{0}$ is infinite, and put $T=L^{(3)}L_{0}$; recall that $T/L_{0}$ is
semisimple (a consequence of Proposition \ref{schreier}). If $L/T$ is finite,
then $T/L_{0}$ is infinite; in this case, set $H=T$ and $K=L_{0}$. If $L/T$ is
infinite, then some term $S$ of the derived series of $L$ must satisfy: $L/S$
is finite and $S/S^{\prime}$ is infinite. In this case, we take $H=S$ and
$K=S^{\prime}$.$\medskip$

\emph{Case 2}: where $G$ is connected.$\medskip$

In this case, $N$ is dense in $G$. According to \cite{HM1}, Theorem 9.24, $G$
is a quotient $(A\times P)/Z$ where $A$ is a connected compact abelian group,
$P=\prod_{i\in I}S_{i}$ is a connected compact semisimple group, and
$Z\leq\mathrm{Z}(P)$. If we assume that $G$ has no infinite abelian image, it
follows that $G\cong P/(P\cap Z)$. If $G$ has a proper dense normal subgroup,
then so does $P$. Now the claim (*) in Subsection \ref{simplequot}, above,
shows that there exists a non-principal ultrafilter on the index set $I$: but
this implies that $I$ is infinite. Thus $G\cong P/(P\cap Z)$ has a strictly
infinite semisimple quotient $G/K$ isomorphic to the product $\prod_{i\in
I}S_{i}/\mathrm{Z}(S_{i})$.$\medskip$

\emph{The General Case}.$\medskip$

If $NG^{0}<L$ the result follows by Case 1 applied to $G/G^{0}$. So we may
assume that $NG^{0}=L$. Let $Z=\mathrm{Z}(G^{0})$. If $L/ZN$ is finite then
$H:=ZN$ is open in $G$ and $K:=H^{\prime}$ is closed (Corollary
\ref{derived-cpct}); and $K$ has infinite index in $H$ because $K\leq N$.

So replacing $G$ by $G/Z$ and $N$ by $ZN/N$ we may assume that $\mathrm{Z}%
(G^{0})=1$. In this case, $G^{0}=\prod_{i\in I}S_{i}$ where each $S_{i}$ is a
connected (and centreless) simple Lie group (\cite{HM1}, \emph{loc. cit.}).
Put $D=G^{0}\cap N$. Then $[G^{0},N]\leq D$. It follows that%
\[
G^{0}=G^{0\prime}\leq\lbrack G^{0},L]=[G^{0},\overline{N}]\leq\overline{D},
\]
so $D$ is dense in $G^{0}$. In particular, in view of Case 2 above, the index
set $I$ must be infinite.

Since $G/G^{0}$ is finitely generated, so is $L/G^{0}$; thus $L=G^{0}%
\overline{\left\langle y_{1},\ldots,y_{d}\right\rangle }$ for some $y_{l}\in
N$. Then $[G^{0},y_{l}]\subseteq D$ for each $l$. Applying Proposition
\ref{ultrafilterautos} we deduce that there exists an infinite subset $J$ of
$I$ such that each $y_{l}$ normalizes $S_{i}$ for every $i\in J$. As
$\mathrm{N}_{L}(S_{i})$ is closed and contains $G^{0}$, it follows that
$S_{i}$ is normal in $L$ for every $i\in J$. Put $C_{i}=\mathrm{C}_{L}(S_{i}%
)$. Then $L/C_{i}S_{i}$ embeds in the outer automorphism group of $S_{i}$,
which embeds in $\mathrm{Sym}(3)$ (cf. \cite{HM1}, page 256). As the finitely
generated profinite group $L/G^{0}$ admits only finitely many homomorphisms
into $\mathrm{Sym}(3)$ and $C_{i}S_{i}\geq G^{0}$, it follows that $L$ has a
characteristic open subgroup $H\geq G^{0}$ such that $C_{i}S_{i}\geq H$ for
all $i\in J$.

Thus putting $X=\prod_{i\in J}S_{i}$ we have $H=\mathrm{C}_{H}(X)\times X$;
indeed, if $h\in H$ then $h=c_{i}s_{i}$ ($c_{i}\in C_{i},~s_{i}\in S_{i}$) for
each $i\in J$, and if $x=(s_{i})_{i\in J}$ then $[hx^{-1},s_{j}]=1$ for every
$j\in J$, so $hx^{-1}\in\mathrm{C}_{H}(X)$. To complete the proof we may
therefore take $K=\mathrm{C}_{H}(X)$.

\bigskip

\noindent\emph{Remark.} It might be more natural to ask: when does $G$ have a
\emph{virtually normal} virtually dense subgroup? ($N$ is \emph{virtually
normal }if the normalizer $\mathrm{N}_{G}(N)$ has finite index in $G$).

\begin{corollary}
$G$ has a virtually normal virtually dense subgroup of infinite index if and
only if $G$ has a normal virtually dense subgroup of infinite index.
\end{corollary}

\noindent This follows from the theorem: suppose that $R$ is a subgroup of
finite index in $G$, that $H$ is open and normal in $R$, and that $K<H$ is a
closed normal subgroup of $R$. Then as above we can replace $K$ by a closed
normal subgroup $K_{\ast}$ of $G$ such that $H/K_{\ast}$ is a subdirect
product of $\left\vert G:R\right\vert $ copies of $H/K,$ and replace $H$ by
$H_{\ast}\vartriangleleft G$, where $H_{\ast}$ is normal of finite index in
$G$. Then $H_{\ast}$ is open by Corollary \ref{cpctserre}, whence $H_{\ast
}/K_{\ast}$ is again an infinite abelian or semisimple group of the same type
as $H/K$.\bigskip

The conditions for the existence of a \emph{proper dense normal subgroup} are
more delicate, and we merely state the result. The proof, which depends on
Corollary \ref{normal_mod_G_0} and further arguments in the spirit of
Subsection \ref{quasi}, will appear elsewhere.$\medskip$

\noindent\textbf{Definition.} \textbf{(a) }Let $S$ be a finite simple group.
Then $Q(S)$ denotes the following subgroup of $\mathrm{Aut}(S)$:%
\begin{align*}
&  \mathrm{InnDiag}(S)\left\langle \tau\right\rangle \text{ if }%
S=D_{n}(q),~n\geq5\\
&  \mathrm{InnDiag}(S)\left\langle [q]\right\rangle \text{ if }S=^{2}%
\!D_{n}(q)\\
&  \mathrm{InnDiag}(S)\text{ if }S\text{ is of another Lie type}\\
&  \mathrm{Aut}(S)\text{ in all other cases}%
\end{align*}
where $\tau$ is the non-trivial graph automorphism of $D_{n}(q)$ and $[q]$
denotes the field automorphism of order $2$ of $^{2}\!D_{n}(q)$.

\textbf{(b) }Let $S$ be a connected simple Lie group. Then%
\[
Q(S)=\left\{
\begin{array}
[c]{ccc}%
\mathrm{Aut}(S) & \text{if} & S=\mathrm{PSO}(2n),~n\geq3\\
&  & \\
\mathrm{Inn}(S) &  & \text{else}%
\end{array}
\right.  .
\]

\textbf{(c) }A topological group $H$ is \emph{Q-almost-simple} if
$S\vartriangleleft H\leq Q(S)$ where $S$ is a finite simple group or a
connected simple Lie group..\bigskip

\noindent If $H$ is Q-almost-simple as above, the \emph{rank} of $H$ is then
the rank of $S$, namely the (untwisted) Lie rank if $S$ is of Lie type, $n$ if
$S\cong\mathrm{Alt}(n)$, and zero otherwise.

\begin{theorem}
Let $G$ be a compact group with $G/G^{0}$ finitely generated. Then $G$ has a
proper dense normal subgroup if and only if one of the following holds:

\begin{itemize}
\item $G^{\mathrm{ab}}$ is infinite, \emph{or}

\item $G$ has a strictly infinite semisimple quotient, \emph{or}

\item $G$ has \emph{Q}-almost-simple quotients of unbounded ranks.
\end{itemize}
\end{theorem}


\begin{thebibliography}{99999}                                                                                            %


\bibitem[AG]{AG}M. Aschbacher and R. M. Guralnick, Some applications of the
first cohomology group, \emph{J.Algebra }\textbf{90} (1984), 446-460.

\bibitem[As]{As}M. Aschbacher, \emph{Finite group theory,} Cambridge Univ.
Press, Cambridge, 1988.

\bibitem[B]{B}H. Blau, A fixed-point theorem for central elements in
quasisimple groups, \emph{Proc. AMS} \textbf{122} (1994), 79-84.

\bibitem[BCP]{BCP}L. Babai, P. J. Cameron and P. P\'{a}lfy, On the orders of
primitive groups with restricted non-abelian composition factors, \emph{J.
Algebra }\textbf{79} (1982), 161-168.

\bibitem[BNP]{BNP}L. Babai, N. Nikolov and L. Pyber, Product Growth and Mixing
in Finite Groups, \emph{19th ACM-SIAM Symposium on Discrete Algorithms}, SIAM,
2008, Pages 248-257.





\bibitem[Bu]{Bump}D. Bump, \emph{Lie groups}, Springer-Verlag, New York, 2004.

\bibitem[C]{C}R. W. Carter, \emph{Finite groups of Lie type: conjugacy classes
and complex characters}, Wiley and Sons, London, 1985.

\bibitem[DM]{DM}J. D. Dixon and B. Mortimer, \emph{Permutation groups},
Springer-Verlag, New York, 1996.

\bibitem[FG]{FG}J. Fulman, R. Guralnick, Bounds on the number and sizes of
conjugacy classes in finite Chevalley groups with applications to
derangements, http://arxiv.org/abs/0902.2238

\bibitem[FMS]{FMS}T. Frayne, A. Morel and D. Scott, Reduced direct products,
\emph{Fund. Math.} \textbf{51} (1962),195-228.



\bibitem[FJ]{FJ}M. Fried and M. Jarden, \emph{Field arithmetic, }%
Springer-Verlag, Berlin -- Heidelberg, 1986.

\bibitem[GaSh]{GaSh}S. Garion and A. Shalev, Commutator maps, measure
preservation, and $T$-systems, \emph{Trans. Amer. Math. Soc.} \textbf{361}
(2009), 4631--4651.

\bibitem[Gch]{Gch}W. Gasch\"{u}tz, Zu einem von B. H. und H. Neumann
gestellten Problem, \emph{Math. Nachhrichten} \textbf{14} (1955), 249-252.

\bibitem[GL]{GL}R. M. Guralnick and F. L\"{u}beck, On $p$-singular elements in
Chevalley groups in characteristic $p$, in \emph{Groups and computation III},
169-182, Ohio State Univ. Math. Res. Inst. Publ. \textbf{8}, de Gruyter,
Berlin, 2001.

\bibitem[GLS]{GLS}D. Gorenstein, R. Lyons and R. Solomon, \emph{The
classification of the finite simple groups, no.3}, American Math. Soc.,
Providence, Rhode Island, 1998.

\bibitem[Go]{Go}D. Gorenstein, \emph{Finite Groups, }2nd ed., Chelsea, New
York, 1980.

\bibitem[Gt]{Goto}M. Goto, A theorem on compact semisimple groups. \emph{J.
Math. Soc. Japan} \textbf{1} (1949), 270-272.

\bibitem[GFSG]{GFSG}D. Gorenstein, \emph{Finite simple groups, }Plenum Press,
New York and London, 1982.

\bibitem[GSS]{GSS}D. Gluck, A. Seress and A. Shalev, Bases for primitive
permutation groups and a conjecture of Babai, \emph{J. Algebra} \textbf{199}
(1998), 367--378.

\bibitem[HM]{HM1}K. H. Hofmann and S. A. Morris, \emph{The structure of
compact groups}. 2nd edn., de Gruyter Studies in Mathematics, \textbf{25}.
Walter de Gruyter \& Co., Berlin, 2006.



\bibitem[J]{J}G. A. Jones, Varieties and simple groups, \emph{J. Austral.
Math. Soc}. \textbf{17} (1974), 163--173.

\bibitem[JZ]{JZ}A. Jaikin-Zapirain, On linear just infinite pro-$p$ groups,
\emph{J. Algebra }\textbf{255} (2002), 392-404.

\bibitem[KlL]{KlLi}P. Kleidman and M. Liebeck, \emph{The subgroup structure of
the finite classical groups}, LMS Lect. Notes \textbf{129}, Cambridge Univ.
Press, Cambridge, 1990.

\bibitem[KL]{KL}M. Kapovich and B. Leeb, On asymptotic cones and
quasi-isometry of fundamental groups of 3-manifolds, \emph{GAFA} \textbf{5}
(1995), 582-603.

\bibitem[LaS]{LaS}V. Landazuri and G. M. Seitz, On the minimal degrees of
projective representations of the finite Chevalley groups, \emph{J. Algebra}
\textbf{32} (1974), 418--443.

\bibitem[LiSh]{LiSh}M. W. Liebeck and A. Shalev, Diameters of finite simple
groups: sharp bounds and applications, \emph{Annals of Math.} \textbf{154}
(2001), 383-406.

\bibitem[LiSh2]{LiSh2}M. W. Liebeck and A. Shalev, Fuchsian groups, finite
simple groups and representation varieties. \emph{Invent. Math.} \textbf{159}
(2005), 317--367.

\bibitem[LOST]{LOST}M. Liebeck, E. O'Brien, A. Shalev and P. Tiep, The Ore
conjecture, \emph{J. European Math. Soc.} \textbf{12} (2010), 939-1008.

\bibitem[LOST2]{LOST2}M. Liebeck, E. O'Brien, A. Shalev and P. Tiep,
Commutators in finite quasisimple groups, \emph{to appear}

\bibitem[MZ]{MZ}C. Martinez and E. Zelmanov, Products of powers in finite
simple groups, \emph{Israel J. Math.} \textbf{96} (1996), 469--479.

\bibitem[NS]{NS}N. Nikolov and D. Segal, On finitely generated profinite
groups, I: strong completeness and uniform bounds, \emph{Annals of Math.
}\textbf{165} (2007), 171--238.

\bibitem[NS2]{NS2}N. Nikolov and D. Segal, On finitely generated profinite
groups, II: products in quasisimple groups, \emph{Annals of Math.
}\textbf{165} (2007), 239--273.

\bibitem[NSP]{NSP}N. Nikolov and D. Segal, Powers in finite groups,
\emph{Groups, Geometry and Dynamics}, to appear; arXiv:0909.6439

\bibitem[SW]{SW}J. Saxl and J. S. Wilson, A note on powers in simple groups,
\emph{Math. Proc. Cambridge Philos. Soc.} \textbf{122} (1997), 91--94.

\bibitem[S1]{S1}D. Segal, Closed subgroups of profinite groups, \emph{Proc.
London Math. Soc.} \textbf{81 }(2000), 29--54.

\bibitem[S2]{S2}D. Segal, \emph{Words: notes on verbal width in groups,
}London Math. Soc. Lecture Notes Series \textbf{361}, Cambridge Univ. Press,
Cambridge, 2009.

\bibitem[SGT]{SGT}J.-P.Serre, \emph{Topics in Galois Theory,} Res. Notes
Math\emph{. }\textbf{1}, Jones and Bartlett, Boston -- London, 1992.

\bibitem[W]{W}J. S. Wilson, On simple pseudofinite groups, \emph{J. London
Math. Soc.} \textbf{51} (1995), 471--490.

\bibitem[Z]{Z}E. I. Zelmanov, On the restricted Burnside problem, \emph{Proc.
Internl. Congress Math. Kyoto 1990,} Math. Soc. Japan, Tokyo, 1991, pp. 395-402.
\end{thebibliography}
\end{document}